\documentclass[10pt]{elsarticle}
\usepackage{amsmath}
\usepackage{amsthm}
\usepackage{amssymb}
\usepackage{mathtools}
\usepackage{hyperref}
\usepackage{makecell}
\usepackage{mathpartir}
\usepackage{stmaryrd}
\SetSymbolFont{stmry}{bold}{U}{stmry}{m}{n}
\usepackage{tikz}
\usepackage{quiver}
\usepackage{subcaption}
\usepackage{adjustbox}
\usepackage{listings}
\usepackage[shortcuts]{extdash}
\usetikzlibrary{decorations.pathreplacing,nfold,positioning,cd}
\usepackage{xspace}
\usepackage{siunitx}
\usepackage{csquotes}
\usepackage{thm-restate}
\usepackage[normalem]{ulem}
\usepackage[format=plain,
            labelfont={it},
            textfont=it]{caption}

\newtheorem{theorem}{Theorem}[section]
\newtheorem{corollary}[theorem]{Corollary}
\newtheorem{lemma}[theorem]{Lemma}
\newtheorem{proposition}[theorem]{Proposition}

\newtheorem{definition}[theorem]{Definition}

\newtheorem{example}[theorem]{Example}

\newcommand{\catt}{\ensuremath{\mathsf{CaTT}}\xspace}

\DeclarePairedDelimiter{\angl}{\langle}{\rangle}

\DeclareMathOperator{\Var}{Var}

\newcommand{\src}{\partial^-}
\newcommand{\tgt}{\partial^+}
\DeclareMathOperator{\id}{id}
\DeclareMathOperator{\op}{op}
\DeclareMathOperator{\coh}{coh}
\DeclareMathOperator{\comp}{comp}

\DeclareMathOperator{\supp}{supp}
\DeclareMathOperator{\inc}{in}

\DeclareMathOperator{\inj}{in}

\DeclareMathOperator{\depth}{depth}

\DeclareMathOperator{\U}{Up}
\DeclareMathOperator{\susp}{\Sigma}

\DeclareMathOperator{\padded}{\Theta}

\DeclareMathOperator{\EH}{EH}
\DeclareMathOperator{\eh}{H}

\DeclareMathOperator{\hexcomp}{hexcomp}
\DeclareMathOperator{\repad}{\Theta}

\DeclareMathOperator{\Ctx}{Ctx}
\DeclareMathOperator{\Ty}{Ty}
\DeclareMathOperator{\Tm}{Tm}
\DeclareMathOperator{\Sub}{Sub}
\newcommand{\ehctx}{\mathbb{E}}
\newcommand{\ehty}{E}
\newcommand{\inl}{\inj^-}

\newcommand{\disc}{\mathbb{D}}
\newcommand{\sphere}{\mathbb{S}}
\newcommand{\spheretype}{S}

\newcommand{\N}{\ensuremath{\mathbb{N}}}

\newcommand{\obj}{\ensuremath{\star}}

\newcommand{\emptycontext}{\ensuremath{\varnothing}}
\newcommand{\ps}{\mathsf{ps}}
\newcommand{\vdashps}{\ensuremath{\vdash_\ps}}
\newcommand{\sub}{\angl}
\newcommand{\point}{\mathbb{P}}
\newcommand{\fun}{\overrightarrow}

\newcommand{\s}{*}
\newcommand{\cir}{\mathop{\circ}}
\newcommand{\boundarycolour}{\color{black}}
\newcommand{\sourcecolour}{\color{black}}
\newcommand{\targetcolour}{\color{black}}
\DeclarePairedDelimiter{\app}{\llbracket}{\rrbracket}

\newcommand{\inv}[1]{{#1}^{-1}}

\tikzset{every arrow/.append style = -{Latex[scale=0.7]}}

\def\bdot{\text{\scriptsize\textbullet}}

\lstdefinestyle{cattstyle}{
  basicstyle={\ttfamily},
  keywordstyle={\color{black}\ttfamily\bfseries},
}

\lstdefinelanguage{catt}{
  sensitive=false,
  keywords={catt,coh,let,check,benchmark},
  otherkeywords={@}
}

\journal{Advances in Mathematics}

\begin{document}

\begin{frontmatter}

\title{
  {\Large \bf Beyond Eckmann-Hilton:
  \texorpdfstring{\\}{}
  Commutativity in Higher Categories}
}

\author{Thibaut Benjamin}
\ead{thibaut.benjamin@universite-paris-saclay.fr}

\affiliation{
  organisation={
    Université Paris-Saclay, CNRS,
    ENS Paris-Saclay, LMF},
  addressline={4 avenue des Sciences},
  postcode={91190},
  city={Gif-sur-Yvette},
  country={France}
}

\author{Ioannis Markakis}
\ead{ioannis.markakis@cl.cam.ac.uk}
\author{Wilfred Offord}
\ead{wgbo2@cam.ac.uk}
\author{Chiara Sarti}
\ead{cs2197@cam.ac.uk}
\author{Jamie Vicary}
\ead{jamie.vicary@cl.cam.ac.uk}

\affiliation{
  organisation={
    Department of Computer Science and Technology,
    University of Cambridge},
  addressline={15 JJ Thomson Ave},
  postcode={CB3 0FD},
  city={Cambridge},
  country={UK}
}

\begin{abstract}
  We show that in a weak globular $\omega$-category, all composition operations are equivalent and commutative for cells with sufficiently degenerate boundary, which can be considered a higher-dimensional generalisation of the Eckmann-Hilton argument. Our results are formulated constructively in a type-theoretic presentation of $\omega$-categories. The heart of our construction is a family of padding and repadding techniques, which gives an equivalence relation between cells which are not necessarily parallel.   Our work has been implemented, allowing us to explicitly compute suitable witnesses, which grow rapidly in complexity as the dimension increases. These witnesses can be exported as inhabitants of identity types in Homotopy Type Theory, and hence are of relevance in synthetic homotopy theory.
  \end{abstract}

\end{frontmatter}

\section{Introduction}

\subsection{Motivation}

\noindent
Since Grothendieck's original conception of (weak, globular) $\omega$-groupoids, a principal goal of higher category theory has been to give an algebraic language for homotopy theory: ``The task of
working out the foundations of tame topology, and a corresponding
structure theory for `stratified (tame) spaces', seems to me a lot more
urgent and exciting still than any program of homological, homotopical 
or topological algebra"~\cite[p. 2]{grothendieck_pursuing_1983}. One expression of this idea is the \textit{homotopy hypothesis}, which says that $\omega$-groupoids should be the same as homotopy types, and remains unproven in this setting. The definition of $\omega$\=/groupoid has since been generalised by Batanin into that of a weak $\omega$-category~\cite{batanin_monoidal_1998}, and further key contributions have been made by Leinster, Maltsiniotis and Ara \cite{leinster_operads_2000, maltsiniotis_grothendieck_2010, ara_infty_2010} among many others.


 However, using these theories to produce explicit homotopy witnesses is challenging, and therefore Grothendieck's original goal is largely unrealised. One homotopical argument with an $\omega$\=/categorical analogue is the well-known Eckmann-Hilton argument~\cite{Eckman1961StructureMI}, which is traditionally used to establish commutativity of higher homotopy groups of topological spaces~\cite{hatcher_algebraic_2002}. In our setting, the statement of Eckmann-Hilton is that any $2$\=/cells whose boundaries are identity cells commute with each other, up to isomorphism, under (vertical) composition, and moreover their vertical and horizontal composites are equivalent, up to some unitor isomorphisms. In both cases, while the homotopical argument is simple, constructing the isomorphism witnessing it in the weak $\omega$\-/categorical setting is not trivial. Once constructed, however, suspending it produces isomorphisms witnessing the commutativity of any $n$\-/cells with identity boundary under composition in the $(n-1)$\textsuperscript{th}  direction.


Homotopically, it is obvious that the Eckmann-Hilton argument should extend to composites of $n$-cells with sufficiently degenerate boundaries in all $n$ possible composition directions. Indeed, if our definition of $\omega$\=/category is to satisfy the homotopy hypothesis, such composites should all be commutative and equivalent in the appropriate sense. However, this has never previously been established. 

In this paper, we  demonstrate commutativity and equivalence of these composites for the first time.
This provides strong evidence that the algebraic model of \(\omega\)\=/categories behaves in the way that we expect.
Our work exploits a recent type-theoretic approach to $\omega$\-/categories due to Finster and Mimram~\cite{finster_typetheoretical_2017,benjamin_type_2020}.
We crucially leverage a recent construction called \emph{naturality}~\cite{benjamin_naturality_2025}, a non-trivial meta-operation on the type theory, which constructs fillers for certain cylinder types.
We have implemented our construction, allowing us to compute, in principle, witnesses of these results in any dimension, and by the work of Benjamin~\cite{benjamin_generating_2024} export them into Homotopy Type Theory~\cite{theunivalentfoundationsprogram_homotopy_2013}.

\subsection{Globular \texorpdfstring{$\omega$\=/}{ω-}categories}

\noindent
In the globular setting, for $n>0$, an $n$\=/cell $a$ has source and target $(n-1)$-cells $\src(a)$, $\tgt(a)$, which we denote as
$a : \src(a) \to \tgt(a)$. For $n>1$ we impose the \textit{globularity condition}, that whenever $a: u \to v$, then $u$ and $v$ have the same boundary.
We may further define the $k$\=/dimensional source and target $\src_k(a)$,
$\tgt_k(a)$ for any $0 \leq k < n$ by taking successive boundaries.

The structure of a higher category allows $n$\=/cells to be composed in a variety of ways. In particular, if $a,b$ are $n$\=/cells with  $\tgt_k(a) = \src_k(b)$, for $0 \leq k < n$ we may form their \textit{binary $k$\=/composite} $a \s_k b$, which we understand as ``gluing'' $a$ and $b$  along their common $k$\=/dimensional boundary. We illustrate this in Figures~\ref{fig:1comp} and \ref{fig:0comp}, which show the 1- and 0\=/composites, respectively, of a pair of 2-cells.

\begin{figure}[t]
    \begin{subfigure}[valign=c]{.49\textwidth}
    \[\begin{tikzcd}[ampersand replacement=\&]
    \bdot \& \bdot
    \arrow[""{name=0, anchor=center, inner sep=0}, from=1-1, to=1-2]
    \arrow[""{name=1, anchor=center, inner sep=0}, bend left = 60, from=1-1, to=1-2]
    \arrow[""{name=2, anchor=center, inner sep=0}, bend right = 60, from=1-1, to=1-2]
    \arrow["a", shorten <=2pt, shorten >=2pt, Rightarrow, from=1, to=0]
    \arrow["b", shorten <=2pt, shorten >=2pt, Rightarrow, from=0, to=2]
    \end{tikzcd}\]

    \caption{$a \s_1 b$} \label{fig:1comp}
    \end{subfigure}
    \begin{subfigure}[valign=c]{.49\textwidth}
    \[\begin{tikzcd}[ampersand replacement=\&]
        \bdot \& \bdot \& \bdot
        \arrow[""{name=0, anchor=center, inner sep=0}, bend left = 45, from=1-1, to=1-2]
        \arrow[""{name=1, anchor=center, inner sep=0}, bend right = 45, from=1-1, to=1-2]
        \arrow[""{name=2, anchor=center, inner sep=0}, bend left = 45, from=1-2, to=1-3]
        \arrow[""{name=3, anchor=center, inner sep=0}, bend right = 45, from=1-2, to=1-3]
        \arrow["a", shorten <=3pt, shorten >=3pt, Rightarrow, from=0, to=1]
        \arrow["b", shorten <=3pt, shorten >=3pt, Rightarrow, from=2, to=3]
\end{tikzcd}\]

    \caption{$a \s_0 b$} \label{fig:0comp}
    \end{subfigure}
    \caption{\centering Composites of a pair of $2$\=/cells.}
\end{figure}
The axioms of a higher category also give us access to a class of cells, called \emph{coherences}, that serve as witnesses that their source and target are in some sense ``equivalent''. To give some first examples, given a $0$\=/cell $x$, we may construct the \emph{identity}, a $1$\=/cell $\id_x : x \to x$, and the \emph{unbiased unitor}, a 2\=/cell \mbox{$u_x : \id_x \s_0 \id_x \to \id_x  $}; and given a 1\=/cell $f:x \to y$, we may construct the \emph{left unitor} $\lambda_f : \id_x \s_0 f \to f$, and the \textit{right unitor} $\rho_f : f \s_0 \id_x \to f$. All such coherences are part of a broader class of cells called
\emph{equivalences}, which are cells with \emph{inverses} satisfying a cancellation law up to higher equivalences.
For an equivalence $e : u \to v$, we denote its inverse $e^{-1} : v \to u$.

In a higher category, there is no reason to expect that the composition operations are related. For example, for $2$\=/cells $a$ and $b$, we do not in general expect an equivalence between $a \s_0 b$ and $a \s_1 b$ when they are both defined; indeed testing these cells for equivalence would not usually make sense, as they have different boundaries. However, if the cell boundaries are degenerate -- that is, given by identities -- we can make this question well-posed by composing $a \s_0 b$ with  the coherences $u_x$ and $u_x^{-1}$, to change its boundary, a procedure which we call \emph{padding}. We can then construct the following equivalence, which we illustrate in Figure~\ref{fig:eh-type}: $$\eh(a,b) : a \s_1 b \; \to\;  u_x^{-1} \s_1 (a \s_0 b) \s_1 u_x$$
Furthermore, applying a duality construction, we obtain a dual equivalence: $$\eh^{\op}(a,b): a \s_1 b \;\to\; u_x^{-1}
\s_1 (b \s_0 a) \s_1 u_x$$ Finally, we may combine these to obtain the following equivalence:
\[
\EH(a,b) :=\eh(a,b) \star_2 (\eh^{\op}(b,a))^{-1} : a \s_1 b \;\to\; b \s_1 a
\]
Hence we conclude that the operation $\s_1$ is commutative up to equivalence. This construction has been described, for instance, by Cheng and Gurski~\cite{cheng_periodic_2007}, and can be seen as a categorical version of the classical Eckmann-Hilton argument. 

\begin{figure}[t]
    \centering
\[\begin{tikzcd}[ampersand replacement=\&]
        \bdot \& \bdot \& \bdot \& \bdot \& \bdot
        \arrow[""{name=0, anchor=center, inner sep=0}, bend left = 60,  from=1-1, to=1-2]
        \arrow[""{name=1, anchor=center, inner sep=0},  from=1-1, to=1-2]
        \arrow[""{name=8, anchor=center, inner sep=0},  bend right = 60, from=1-1, to=1-2]
        \arrow["\text{H}", Rightarrow, scaling nfold=3, from=1-2, to=1-3]
        \arrow[""{name=2, anchor=center, inner sep=0}, bend left = 45,  from=1-3, to=1-4]
        \arrow[""{name=3, anchor=center, inner sep=0}, bend right = 45,  from=1-3, to=1-4]
        \arrow[""{name=4, anchor=center, inner sep=0}, bend left = 90,  from=1-3, to=1-5]
        \arrow[""{name=5, anchor=center, inner sep=0}, bend right = 90,  from=1-3, to=1-5]
        \arrow[""{name=6, anchor=center, inner sep=0}, bend left = 45,  from=1-4, to=1-5]
        \arrow[""{name=7, anchor=center, inner sep=0}, bend right = 45,  from=1-4, to=1-5]
        \arrow["a", shorten <=1pt, shorten >=3pt, Rightarrow, from=0, to=1]
        \arrow["b", shorten <=3pt, shorten >=1pt, Rightarrow, from=1, to=8]
        \arrow["a", shorten <=3pt, shorten >=3pt, Rightarrow, from=2, to=3]
        \arrow["{u_x ^{-1}}", shorten <=3pt, Rightarrow, from=4, to=1-4]
        \arrow["b", shorten <=3pt, shorten >=3pt, Rightarrow, from=6, to=7]
        \arrow["u_x", shorten >=3pt, Rightarrow, from=1-4, to=5]
\end{tikzcd}\]

    \caption{\centering The equivalence $\eh : a \s_1 b \to u_x^{-1} \s_1 (a \s_0 b) \s_1 u_x$.}
    \label{fig:eh-type}
\end{figure}

\subsection{Results}

\noindent
In this paper, we use a type-theoretic definition of $\omega$\=/category to generalise this construction, producing a parameterised family of equivalences as follows, where $a,b$ are $n$\=/cells with fully degenerate boundaries\footnote{For $n>1$, an $n$\=/cell is \textit{fully degenerate} when it is of the form $\id^{n}_x :=
{\id(\dots(\id} \hspace{0pt}_x)\dots)$, where $x$ is a $0$\=/cell.}, and where $0 \leq k,l < n$ are distinct composition directions:
\[
\eh ^n_{k,l}(a,b) : a \s_k b \;\to\; \padded^n_{k,l}(a \s_l b)
\]
Here $\Theta_{p,q}^n$ is a \emph{padding} operation, which composes its argument with coherences, generalising our earlier use of $u_x$ and $u_x^{-1}$. Our padding technique is inspired by similar constructions in the work of Finster et al.~\cite{finster_type_2022} and Fujii et al.\cite{fujii_equivalences_2024}, and includes these as a special case.

We interpret the cells $\eh^n_{k,l}(a,b)$ as a \emph{congruence} between the composites $a*_k b$ and $a*_l b$. Congruence is a relation we introduce, extending the notion of equivalence to cells with different boundaries. We define it as the smallest equivalence relation on $n$\=/cells extending equivalence such that any cell is congruent to its composite in any dimension with a coherence. We note that in a strict $\omega$\=/category, coherence cells are all identities, and so congruence and equivalence coincide. Thus congruence is a minimal extension of equivalence that allows \enquote{fixing} the boundaries of cells using coherences.

We can apply a duality operation as in dimension $2$ to produce another $({n\!+\!1})$\=/cell with type  $ b \s_k a \to \padded^n_{k,l} ({ a\s_l b})$, and compose to obtain an equivalence:
\[
\EH^n_{k,l}({a,b}) : a \s_k b \;\to\; b \s_k a
\]
This witnesses  that composition of $n$\=/cells with degenerate boundary is {{commutative}} in any choice of direction $k$, up to equivalence. Such a construction can reasonably be described  as a \textit{higher Eckmann-Hilton argument}. Furthermore, we have a family of such witnesses, according to the choice of the second parameter $l$.

Furthermore, our $\eh^n_{k,l}$ construction can be extended to the case where paddings can appear on both sides, yielding cells $\eh^{n}_{p,k,l}$ as follows:
\[
\eh^{n}_{p,k,l} : \padded^n_{p,k}(a \s_k b) \;\to\; \padded^n_{p,l}(a \s_l b)
\]
When $p=k$ the operation $\Theta^n_{p,k}(-)$ becomes the identity, and so this construction is an extension of our main result, with:
\begin{align*}
    \eh^n_{k,k,l} &= \eh^n_{k,l} & \eh^n_{l,k,l} &= (\eh^n_{k,l}) ^{-1}
\end{align*}
These cells $\eh^n_{p,k,l}$ together with their opposites can be arranged in a structure we  call the \textit{Eckmann-Hilton sphere}, which we illustrate for $n=3$ in Figure~\ref{ehsphere}. This structure has been studied before in semi-strict models \cite{cheng_higher-eh_2024}, but not in the fully weak setting.

\begin{figure}
    \centering
\adjustbox{scale=0.8}{$\begin{tikzcd}[ampersand replacement=\&, row sep = .7cm, column sep = 0.20 cm, scale=0.8, font=\normalfont]
        \&\& {\padded^3_{2,1}(a \s_1 b)} \\
        \&\&\& {\padded^3_{2,0}(a \s_0 b)} \\
        {a \s_2 b} \&\&\&\& {b \s_2 a} \\
        \& {\padded^3_{2,0}(b \s_0 a)} \\
        \&\& {\padded^3_{2,1}(b \s_1 a)}
        \arrow[bend left = 35, from=1-3, to=3-5]
        \arrow[curve={height=-24pt},  to=2-4, from=3-1]
        \arrow[curve={height=24pt}, from=1-3, to=4-2,crossing over]
        \arrow[curve={height=-24pt}, to=5-3, from=2-4]
        \arrow[curve={height=24pt}, from=4-2, to=3-5, crossing over]
        \arrow[bend left = 20,  to=2-4, from=1-3]
        \arrow[bend left = 35, from=3-1, to=1-3]
        \arrow[bend right = 20, shorten <=4pt, from=3-1, to=4-2]
        \arrow[bend left = 20, shorten <=4pt,  to=3-5, from=2-4]
        \arrow[bend right = 35,  to=3-5, from=5-3]
        \arrow[bend right = 20, from=4-2, to=5-3]
        \arrow[bend right = 35,  to=5-3, from=3-1]
\end{tikzcd}$}
    \caption{\centering The $3$-dimensional Eckmann-Hilton sphere.}
    \label{ehsphere}
\end{figure}

\subsection{Type theory}

\noindent
We work in the type theory \catt due to Finster and Mimram~\cite{finster_typetheoretical_2017}, whose models correspond to Grothendieck-Maltsiniotis~\cite{benjamin_globular_2024} or to Batanin-Leinster $\omega$\=/categories~\cite{ara_infty_2010, bourke_iterated_2020}. {Due to these equivalences, our results are valid in any of those models.} Benjamin et. al.~\cite{benjamin_catt_2024} have further shown that the type-theoretical approach to higher category theory is equivalent to the \emph{computadic} approach of Dean et. al.~\cite{dean_computads_2024}, and thanks to this, many results for computads carry over directly to our setting. We will we will freely quote their implications for \catt{}. In particular, we will use the suspension and the opposite operations~\cite{benjamin_hom_2024} for \catt{}, as well as the invertibility constructions of Benjamin and Markakis~\cite{benjamin_invertible_2024}, which have been constructed using the computadic approach.

The  theory \catt has an interpretation in Homotopy Type Theory~\cite{benjamin_generating_2024}, and as a result our constructions can be exported as inhabitants of identity types. The Eckmann-Hilton cell $\EH^2_{1,0}$ has been previously constructed in Homotopy Type Theory~\cite{theunivalentfoundationsprogram_homotopy_2013}, which easily yields $\EH^n_{n-1,n-2}$ by changing the base type. However, for the remaining cases of $\text{EH}^n_{k,l}$, this is the first explicit algebraic construction. Furthermore, our setting is more general than that of Homotopy Type Theory in two ways. Firstly, it is \emph{fully weak}, containing equivalences such as the unbiased unitor $u_x : \id_x\to\id_x *_0 \id_x $ that are identities in Homotopy Type Theory. Secondly, it is \emph{directed}, i.e.\ not all cells have to be invertible. Thus our construction applies to a wider class of examples.

The type theory \catt has been implemented as a proof assistant~\cite{benjamin_catt_software_2024}, in which we have automated our construction as a meta-operation. This allows us to evaluate the cells $\eh_{k,l}^n$ for any given parameter values, and in Appendix~\ref{sec:implementation} we show that the complexity of these composites grows quickly with dimension. Due to their intricate structure, defining these cells directly by hand would likely not be feasible.

\begin{figure}[t]
    \centering
\[\adjustbox{scale=0.75}{\begin{tikzcd}[ampersand replacement=\&, column sep = 0.7cm]
        \&\& \bdot \& \bdot \& \bdot \& \bdot \& \bdot \& \bdot \\[6pt]
        \bdot \& \bdot \&{}\& {\s_1} \&{}\&{}\& {\s_1} \&{}\& \bdot \& \bdot \& \bdot \& \bdot \& \bdot \& \bdot \\[6pt]
        \&\& \bdot \& \bdot \& \bdot \& \bdot \& \bdot \& \bdot
        \arrow[""{name=0, anchor=center, inner sep=0}, bend left = 45,  from=1-3, to=1-4]
        \arrow[""{name=1, anchor=center, inner sep=0}, bend right = 45,  from=1-3, to=1-4]
        \arrow[""{name=2, anchor=center, inner sep=0}, bend left = 90, looseness=1.2,  from=1-3, to=1-5]
        \arrow[""{name=3, anchor=center, inner sep=0, looseness=1.2}, bend right = 90,  from=1-3, to=1-5, looseness=1.2]
        \arrow[""{name=4, anchor=center, inner sep=0}, bend left = 45,  from=1-4, to=1-5]
        \arrow[""{name=5, anchor=center, inner sep=0}, bend right = 45,  from=1-4, to=1-5]
        \arrow[""{name=6, anchor=center, inner sep=0}, bend left = 45,  from=1-6, to=1-7]
        \arrow[""{name=7, anchor=center, inner sep=0}, bend right = 45,  from=1-6, to=1-7]
        \arrow[""{name=8, anchor=center, inner sep=0}, bend left = 90,  from=1-6, to=1-8, looseness=1.2]
        \arrow[""{name=9, anchor=center, inner sep=0}, bend right = 90,  from=1-6, to=1-8, looseness=1.2]
        \arrow[""{name=10, anchor=center, inner sep=0}, bend left = 45,  from=1-7, to=1-8]
        \arrow[""{name=11, anchor=center, inner sep=0}, bend right = 45,  from=1-7, to=1-8]
        \arrow[""{name=12, anchor=center, inner sep=0}, bend left = 80,  from=2-1, to=2-2, looseness=1.2]
        \arrow[""{name=13, anchor=center, inner sep=0}, bend right = 80,  from=2-1, to=2-2, looseness=1.2]
        \arrow[""{name=14, anchor=center, inner sep=0},  from=2-1, to=2-2]
        \arrow["{X_1}"{yshift=5pt}, Rightarrow, scaling nfold=3, from=2-2, to=2-3]
        \arrow["{X_2}"{yshift=5pt}, Rightarrow, scaling nfold=3, from=2-5, to=2-6]
        \arrow[""{name=15, anchor=center, inner sep=0}, bend left = 45,  from=3-3, to=3-4]
        \arrow[""{name=16, anchor=center, inner sep=0}, bend right = 45,  from=3-3, to=3-4]
        \arrow[""{name=17, anchor=center, inner sep=0}, bend left = 90,  from=3-3, to=3-5, looseness=1.2]
        \arrow[""{name=18, anchor=center, inner sep=0}, bend right = 90,  from=3-3, to=3-5, looseness=1.2]
        \arrow[""{name=19, anchor=center, inner sep=0}, bend left = 45,  from=3-4, to=3-5]
        \arrow[""{name=20, anchor=center, inner sep=0}, bend right = 45,  from=3-4, to=3-5]
        \arrow[""{name=21, anchor=center, inner sep=0}, bend left = 45,  from=3-6, to=3-7]
        \arrow[""{name=22, anchor=center, inner sep=0}, bend right = 45,  from=3-6, to=3-7]
        \arrow[""{name=23, anchor=center, inner sep=0}, bend left = 90,  from=3-6, to=3-8, looseness=1.2]
        \arrow[""{name=24, anchor=center, inner sep=0}, bend right = 90,  from=3-6, to=3-8, looseness=1.2]
        \arrow[""{name=25, anchor=center, inner sep=0}, bend left = 45,  from=3-7, to=3-8]
        \arrow[""{name=26, anchor=center, inner sep=0}, bend right = 45,  from=3-7, to=3-8]
        \arrow["{X_3}"{yshift=5pt},  Rightarrow, scaling nfold=3, from=2-8, to=2-9]
        \arrow[""{name=27, anchor=center, inner sep=0}, shift left = 2pt, shorten <=-0, shorten >=-0, bend left = 75,  from=2-9, to=2-10, looseness=1.2]
        \arrow[""{name=28, anchor=center, inner sep=0}, shift right = 2pt, shorten <=-0, shorten >=-0, bend right = 75,  from=2-9, to=2-10, looseness=1.2]
        \arrow[""{name=29, anchor=center, inner sep=0},  from=2-9, to=2-10]
        \arrow[""{name=30, anchor=center, inner sep=0}, shift left = 5pt, shorten <=-3pt, shorten >=-3pt, bend left = 90,  from=2-9, to=2-11, looseness=1.2]
        \arrow[""{name=31, anchor=center, inner sep=0}, shift right = 5pt, shorten <=-3pt, shorten >=-3pt, bend right = 90,  from=2-9, to=2-11, looseness=1.2]
        \arrow[""{name=32, anchor=center, inner sep=0},shift left = 2pt, shorten <=-0, shorten >=-0, bend left = 75,  from=2-10, to=2-11, looseness=1.2]
        \arrow[""{name=33, anchor=center, inner sep=0}, shift right = 2pt, shorten <=-0, shorten >=-0, bend right = 75,  from=2-10, to=2-11, looseness=1.2]
        \arrow[""{name=34, anchor=center, inner sep=0},  from=2-10, to=2-11]
        \arrow["{X_4}"{yshift=5pt}, Rightarrow, scaling nfold=3, from=2-11, to=2-12]
        \arrow[""{name=35, anchor=center, inner sep=0}, bend left = 45,  from=2-12, to=2-13]
        \arrow[""{name=36, anchor=center, inner sep=0}, bend right = 45,  from=2-12, to=2-13]
        \arrow[""{name=37, anchor=center, inner sep=0}, bend left = 90,  from=2-12, to=2-14, looseness=1.2]
        \arrow[""{name=38, anchor=center, inner sep=0}, bend right = 90,  from=2-12, to=2-14, looseness=1.2]
        \arrow[""{name=39, anchor=center, inner sep=0}, bend left = 45,  from=2-13, to=2-14]
        \arrow[""{name=40, anchor=center, inner sep=0}, bend right = 45,  from=2-13, to=2-14]
        \arrow["{(\rho^1_{\id_x})^{\text{-}1}}", shorten <=2pt, Rightarrow, from=2, to=1-4]
        \arrow["a", shorten <=3pt, shorten >=3pt, Rightarrow, from=0, to=1]
        \arrow["\rho^1_{\id_x}", shorten >=2pt, Rightarrow, from=1-4, to=3]
        \arrow["\id^2_x", shorten <=3pt, shorten >=3pt, Rightarrow, from=4, to=5]
        \arrow["{u_x^{-1}}", shorten <=2pt, Rightarrow, from=8, to=1-7]
        \arrow["a", shorten <=3pt, shorten >=3pt, Rightarrow, from=6, to=7]
        \arrow["u_x", shorten >=2pt, Rightarrow, from=1-7, to=9]
        \arrow["\id^2_x", shorten <=3pt, shorten >=3pt, Rightarrow, from=10, to=11]
        \arrow["a", shorten <=2pt, shorten >=2pt, Rightarrow, from=12, to=14]
        \arrow["b", shorten <=2pt, shorten >=2pt, Rightarrow, from=14, to=13]
        \arrow["{(\lambda^1_{\id_x})^{\text{-}1}}", shorten <=2pt, Rightarrow, from=17, to=3-4]
        \arrow["\id^2_x", shorten <=3pt, shorten >=3pt, Rightarrow, from=15, to=16]
        \arrow["\lambda^1_{\id_x}", shorten >=2pt, Rightarrow, from=3-4, to=18]
        \arrow["b", shorten <=3pt, shorten >=3pt, Rightarrow, from=19, to=20]
        \arrow["{u_x^{-1}}", shorten <=2pt, Rightarrow, from=23, to=3-7]
        \arrow["\id^2_x", shorten <=3pt, shorten >=3pt, Rightarrow, from=21, to=22]
        \arrow["u_x", shorten >=2pt, Rightarrow, from=3-7, to=24]
        \arrow["b", shorten <=3pt, shorten >=3pt, Rightarrow, from=25, to=26]
        \arrow["a"{pos=0.6}, shorten <=2pt, shorten >=2pt, Rightarrow, from=27, to=29]
        \arrow["\id^2_x"{pos=0.3}, shorten <=2pt, shorten >=2pt, Rightarrow, from=29, to=28]
        \arrow["{u_x^{-1}}"{pos=0.3}, shorten <=2pt, shorten >=10pt, Rightarrow, from=30, to=2-10]
        \arrow["b"{pos=0.3}, shorten <=2pt, shorten >=2pt, Rightarrow, from=34, to=33]
        \arrow["\id^2_x"{pos=0.6}, shorten <=2pt, shorten >=2pt, Rightarrow, from=32, to=34]
        \arrow["u_x"{pos=0.6}, shorten >=2pt, shorten <=10pt, Rightarrow, from=2-10, to=31]
        \arrow["a", shorten <=3pt, shorten >=3pt, Rightarrow, from=35, to=36]
        \arrow["{u_x^{-1}}", shorten <=2pt, Rightarrow, from=37, to=2-13]
        \arrow["b", shorten <=3pt, shorten >=3pt, Rightarrow, from=39, to=40]
        \arrow["u_x", shorten >=2pt, Rightarrow, from=2-13, to=38]
\end{tikzcd}}\]
    \caption{\centering Construction of the base cases $\eh^n_{n-1,0}$ for the case $n=2$.}
    \label{2dEH}
\end{figure}

\subsection{Technical approach}

\noindent
Our main result is Theorem~\ref{thm:half-eh}, which demonstrates congruence for $a \s_k b$ and $a \s_l b$ when $a,b$ have sufficiently degenerate boundary, by constructing an equivalence as follows:
\[
  \eh^{n}_{k,l}(a,b) : a\s_{k} b \;\to\; \padded^{n}_{k,l}(a\s_{l} b)
\]
These cells are constructed by induction, and the primary technical obstacle is definition of the base cases  \(\eh^{n}_{n-1,0}(a,b) \) and \(\eh^{n}_{0,n-1}(a,b) \). Once these cases are established, we obtain the remaining cases using a \emph{naturality} operation which allows us to construct $\eh^{n+1}_{k,l}(a,b) $ from $\eh^n_{k,l}(a,b) $, and a \emph{suspension} operation which allows us to construct $\eh^{n+1}_{k+1,l+1}(a,b)$ from $\eh^n_{kl}(a,b)$.

Here we give an outline of our approach to defining the base cases, to serve as an informal preview of the full construction given in Section~\ref{padding}. The base case $\eh^n_{n-1,0}(a,b)$ is constructed as the following composite, which we illustrate in Figure~\ref{2dEH}:
\[
\eh ^n_{n-1,0}(a,b) = X_1 \s_n X_2 \s_n X_3 \s_n X_4
\]
We now examine these four steps in turn.

\paragraph{Step 1. Application of Generalised Unitors} In the first step, our goal is to apply equivalences to transform the cells $a$ and $b$ into their identity composites $a \s_0 \id^n_x$ and $\id_x^n \s_0 b$, respectively. This cannot be done directly, as these cells do not have the same boundary. Instead we build coherences as follows, which we call \emph{generalised unitors}:
\begin{align*}
  \rho^n_a &:  \padded^n_{\rho}(a *_0 \id^n_x) \to a
  \\
  \lambda^n_b &: \Theta^n_\lambda(\id^n_x \s_{0} \, b) \to b
\end{align*}
Here $\Theta^n_\rho(-)$ and $\Theta^n_\lambda(-)$ are padding operations which iteratively compose their argument with generalised right or left unitors respectively. This yields the following:
\begin{align*}
 X_1 &: a \s_{n-1} b \to \padded_\rho^n(a *_0 \id^n_x) \s_{n-1} \padded_\lambda^n(\id^n_x \s_{0} \, b) \\
 X_1 &= (\rho^n_a)^{-1} \s_{n-1} (\lambda^n_b)^{-1}
\end{align*}

\paragraph{Step 2. Repadding} For the next step, we  apply equivalences which modify the generalised left and right unitors, transforming them into generalised \textit{unbiased} unitors. This ``unbiasing'' process is related to the familiar coherence equation $\rho_{I} = \lambda_{I}$ of monoidal categories~\cite[Equation~5.2]{maclane_naturality_1963}, which recognizes that left and right unitors become equivalent at the monoidal unit $I$. This yields equivalences as follows:
\begin{align*}
  \repad^{n}_{\rho\to u}(a\s_{0}\id^{n}_{x})&: \padded^{n}_{\rho}(a\s_{0}\id^{n}_{x}) \to
  \padded^{n}_{n-1,0}(a\s_{0}\id^{n}_{x})
  \\
  \repad^{n}_{\lambda \to u}(\id^{n}_{x}\s_{0} \, b)&:
  \padded^{n}_{\lambda}(\id^{n}_{x}\s_{0} \, b) \to
  \padded^{n}_{n-1,0}(\id^{n}_{x}\s_{0} \, b)
\end{align*}
which comprise our next cell in the composite:
\begin{align*}
  X_2 &: \padded^n_\rho(a *_0 \id^n_x) \s_{n-1} \padded^n_\lambda(\id^n_x \s_{0} \, b)
  \to \padded^n_{n-1,0}(a *_0 \id^n_x) *_{n-1} \padded^n_{n-1,0}(\id^n_x *_0 \, b) \\
  X_2 &= \repad^{n}_{\rho\to u}(a\s_{0}\id^{n}_{x}) \s_{n-1} \repad^{n}_{\lambda\to u}(\id^{n}_{x}\s_{0} \, b)
\end{align*}

\paragraph{Step 3. Pseudofunctoriality of Padding} At this point, both
\(a\s_{0}\id^{n}_{x}\) and \(\id^{n}_{x} \s_{0} b\) are present in the term with the same
unbiased padding.
This allows us to apply a \emph{pseudofunctoriality of padding} construction \(\Xi\),
which relates the composite of paddings to the padding of the composite.
This construction takes the form of a cell:
\begin{align*}
X_3 &:
  \padded^n_{n-1,0}(a \s_0 \id^n_x) \s_{n-1} \padded^n_{n-1,0}(\id^n_x \s_0 b)
\to \padded^n_{n-1,0}((a \s_0 \id^n_x)\s_{n-1}(\id^n_x \s_0 b)) \\
X_3 &= \Xi^{n}_{n-1,0}(a \s_0 \id^n x, \id^n_x \s_0 b)
\end{align*}
\paragraph{Step 4. Interchanger}
In the final step, we apply the standard interchanger coherence, which recognizes that the following composites are  equivalent:
\[
  \zeta^n : (a \s_0 \id^n_x)*_{n-1}(\id^n_x *_0 b) \to a *_0 b
\]
We apply an unbiased padding to this interchanger to obtain the final cell in our composite. {Below, the notation $\uparrow$ indicates the \emph{functorialisation} construction, which we describe in Section~\ref{sec:functorialisation}:}
\begin{align*}
X_4 &: \padded^n_{n-1,0}((a \s_0 \id^n_x)*_{n-1}(\id^n_x *_0 b))
\to \Theta^n_{n-1,0} (a \s_0 b)\\
X_4 &= (\padded^n_{n-1,0}\uparrow v^n_0)\app{\zeta^n}
\end{align*}
The construction of the cell $\eh^n_{0,n-1}(a,b)$ makes use of similar components.
This concludes the overview of our construction.

\section{CaTT and Previous Work}\label{cattbackground}

\noindent
The type theory $\catt$ gives a convenient inductive syntax for the theory of $\omega$\=/categories. More specifically, contexts $\Gamma$ of $\catt$ correspond to \emph{finite computads}, which are finite generating data for free $\omega$\=/categories.  Terms $\Gamma \vdash t : A$ of $\Gamma$ correspond to the cells of the $\omega$\=/category generated by $\Gamma$, with the type $A$ indicating the source, target, and dimension of $t$. Models of $\catt$ in the sense of Dybjer~\cite{dybjer_internal_1996} are exactly $\omega$\=/categories in the sense of Grothendieck-Maltsiniotis and Batanin-Leinster \cite{benjamin_globular_2024,ara_infty_2010,bourke_iterated_2020}, and thus all our constructions hold for any cells with degenerate boundary in an $\omega$-category.

\subsection{The Type Theory \texorpdfstring{$\catt$}{CaTT}}
\noindent
The type theory $\catt$ has $4$ kinds of syntactic entities: contexts, types, terms, and substitutions. The introduction rules for this untyped syntax is presented in Figure~\ref{fig:untypedsyntax}. We assume a countable alphabet $x,y,z,f,g,h,a,b\dots$ of variable symbols~$\mathcal V$. For a context $\Gamma$, $\Var(\Gamma)$ is the set of variables in the context.

\begin{figure}
    \centering
    \begin{mathpar}
    \inferrule{ }{\varnothing : \Ctx} \and \inferrule{\Gamma : \Ctx \\ A : \Ty}{(\Gamma, x : A) : \Ctx}(x \notin \Var(\Gamma)) \\
    \inferrule{ }{\obj : \Ty}\and \inferrule{A : \Ty \\ u : \Tm \\ v : \Tm}{u \to_A v : \Ty} \\
    \inferrule{ }{x : \Tm}
    \and \inferrule{\Gamma : \Ctx \\ A : \Ty \\ \sigma : \Sub}{\coh(\Gamma : A)[\sigma] : \Tm} \\
    \inferrule{ }{\langle\rangle : \Sub}\and \inferrule{\sigma : \Sub \\ t : \Tm}{\langle \sigma , x \mapsto t\rangle : \Sub}
\end{mathpar}
    \caption{The untyped syntax of \catt.}
    \label{fig:untypedsyntax}
\end{figure}

 \begin{figure}
   \centering
   \begin{align*}
     x\app{\langle\rangle} &:= x
     & x\app{\langle \sigma, y \mapsto t\rangle} &:= x\app{\sigma} \text{ if $x\neq y$}  \\
     x\app{\langle \sigma, x \mapsto t\rangle} &:= t
     &  \coh(\Gamma : A)[\tau]\app{\sigma} &:= \coh(\Gamma : A)[\tau \cir \sigma]\\
     \obj\app{\sigma} &:= \s
     & (u \to_A v)\app{\sigma} &:= u\app{\sigma} \to_{A\app{\sigma}} v\app{\sigma}\\
     \langle\rangle \cir \sigma &:= \sub{}
     & \langle\tau, x \mapsto t\rangle \cir \sigma &:= \langle \tau\cir\sigma, x\mapsto t\app{\sigma}\rangle
   \end{align*}
   \caption{\centering Definition of the action of substitutions}
   \label{def:substitution}
 \end{figure}

 We give the definition of substitution application and composition in
 Figure~\ref{def:substitution}. We define also the \emph{support} $\supp_\Gamma(t)$ (resp.
 $\supp_\Gamma (A)$, $\supp_{\Gamma} (\sigma)$) of a term $t$, (resp. a type
 $A$, or substitution $\sigma$) relative to a context $\Gamma$, to be the union
 ranging over $(x : B) \in \Gamma$ such that \(x\) appears in $t$ (resp. $A$,
 $\sigma$) of the sets $\{x\}\cup \, \supp_\Gamma(B)$. When there is no
 ambiguity, we omit the index $\Gamma$. We say a term $t$ (resp. a type $A$) is
 \emph{full} in a context $\Gamma$ if $\supp_\Gamma(t) = \Var (\Gamma)$ (resp.
 $\supp_\Gamma(A) = \Var(\Gamma)$).


\begin{figure}
    \begin{subfigure}[b]{.32\textwidth}
    $\left(\begin{array}{@{}c@{}}
     x : \obj,\ y : \obj,\\ f : x \to_\obj y,\\  z : \obj,\\  g : y \to_\obj z,\\  h : y \to_\obj z,\\ a : g \to_{(y \to_\obj z)} h,\\  k : y \to_\obj z,\\ b : h \to_{(y \to_\obj z)} k
    \end{array}\right)$
    \caption{Pasting context}
    \end{subfigure}
    \begin{subfigure}[b]{.32\textwidth}
\[         \begin{tikzpicture}[every node/.style={scale=1},baseline=(x.base), inner sep=2pt]
          \node [black,label={left:$x$},label={above:$y$},label={right:$z$},on grid](x) {$\bdot$};
          \node [black,label={above:$f$},above left =0.4 and 0.25 of x, on grid] (x0) {$\bdot$};
          \node [black, above right = 0.4 and 0.25 of x, on grid] (x1) {$\bdot$};
          \node[label={above:$a$}, above left=0.4 and 0.25 of x1] (x10) {$\bdot$};
          \node[label={above:$b$}, above right=0.4 and 0.25 of x1] (x11) {$\bdot$};
          \node[below=1 of x1] {};
          \node[left=0.4 of x0] {};
          \draw[black] (x.center) to (x0.center);
          \draw[black] (x.center) to (x1.center);
          \draw (x1.center) to (x10.center);
          \draw (x1.center) to (x11.center);
          \node [black,label = {left:$g$}, label={above:$h$},label={right:$k$}, above right = 0.4 and 0.25 of x, on grid] (x1) {$\bdot$};
        \end{tikzpicture}
        \]
    \caption{ Batanin tree}
    \end{subfigure}
    \begin{subfigure}[b]{0.32\textwidth}
        \[\begin{tikzcd}[ampersand replacement=\&]
        {\boundarycolour x} \& {\boundarycolour y} \& {\boundarycolour z} \\[10pt]
        \phantom{\bullet}
        \arrow[black,"{\boundarycolour f}", from=1-1, to=1-2]
        \arrow[""{name=0, anchor=center, inner sep=0}, "h"{description}, from=1-2, to=1-3]
        \arrow[black,""{name=1, anchor=center, inner sep=0}, "{\sourcecolour g}", curve={height=-18pt}, from=1-2, to=1-3]
        \arrow[black,""{name=2, anchor=center, inner sep=0}, "{\targetcolour k}"', curve={height=18pt}, from=1-2, to=1-3]
        \arrow["a", shorten <=2pt, shorten >=2pt, Rightarrow, from=1, to=0]
        \arrow["b", shorten <=2pt, shorten >=2pt, Rightarrow, from=0, to=2]
\end{tikzcd}\]

    \caption{Pasting of discs}
    \end{subfigure}
    \caption{\centering The correspondence between pasting contexts, Batanin trees and pastings of discs.}
    \label{batanin tree}
\end{figure}

A special role in $\catt$ is played by \emph{pasting contexts}, a class of contexts that correspond to certain pastings of discs. They may be  characterised formally via a bijection with finite rooted planar trees (\emph{Batanin trees}). Given such a tree, we assign to each node $N$ a sequence of labels $v_1, \dots, v_{n+1}$ taken from our set of variables $\mathcal V$, where $n$ is the number of children of $N$. To each node $N$ we then recursively assign a type $\Ty(N)$ for its variables: the type of the root is $\obj$ and the type of the $i$\textsuperscript{th} child of a node $N$ labelled $v_1,\dots,v_n$ is $(v_i\!\to_{\Ty(N)}\!v_{i+1})$. To produce the context associated with such a labelled tree, it remains to choose an order on the variables: we fix an arbitrary ordering for each tree, such that every positive-dimensional variable is listed after its source and target. We illustrate this correspondence in Figure~\ref{batanin tree}; formal treatments are available in the literature~\cite[Theorem~5.5]{benjamin_catt_2024}.

We write $\Gamma \vdash_\ps$ if $\Gamma$ is a pasting context. A
nontrivial\footnote{The trivial pasting context is the point context
  $\point = (x : \obj)$.} pasting context $\Gamma$ has well-defined
\emph{source} and \emph{target }pasting contexts $\partial^- \Gamma$ and
$\partial^+ \Gamma$, corresponding to the trees obtained by removing the leaves
of maximal height, and keeping either the leftmost or rightmost labels for the
new leaves, respectively. Denoting \(\Gamma\) the pasting context defined in
Figure~\ref{batanin tree}, its source of and targets are given by:
\begin{align*}
  \partial^{-}\Gamma &=(x : \obj,\ y : \obj,\ f : x \to_\obj y,\ z : \obj,\  g : y
                       \to_\obj z)\\
  \partial^{+}\Gamma &=(x : \obj,\ y : \obj,\ f : x \to_\obj y,\ z : \obj,\  k : y \to_\obj z)
\end{align*}
The untyped syntax of \(\catt\) is subject to $4$ judgements, the
derivation rules for which are presented in Figure~\ref{cattrules}:
\begin{itemize}
    \item $\Gamma \vdash$ the judgement that $\Gamma$ is a valid context
    \item $\Gamma \vdash A$ the judgement that $A$ is a valid type in $\Gamma$
    \item $\Gamma \vdash t : A$ the judgement that $t$ is a term of type $A$ in $\Gamma$
    \item $\Gamma \vdash \sigma : \Delta$ the judgement that $\sigma$ is a substitution $\Gamma \to \Delta$
\end{itemize}
\begin{figure}

\begin{mathpar}
    \inferrule{ }{\varnothing \vdash}
    \and
    \inferrule{\Gamma \vdash \\ \Gamma \vdash A}{(\Gamma, x : A)\vdash}\;\text{($x\notin\Var(\Gamma)$)} \\
    \inferrule{ }{\Gamma \vdash \star}\and
    \inferrule{\Gamma\vdash A\\ \Gamma\vdash u : A \\ \Gamma \vdash v : A}{\Gamma \vdash u \to_A v} \\
    \inferrule{\Gamma \vdash \\ (x : A)\in \Gamma}{\Gamma \vdash x : A}\;\;
    \inferrule{\Gamma\vdash_{ps} \\ \Gamma \vdash A \\ \Delta \vdash \sigma : \Gamma}
{\Delta \vdash \coh(\Gamma:A)[\sigma]:A\app{\sigma}}\;\text{($*$)}\\

    \inferrule{ }{\langle\rangle : \varnothing \to \Gamma}\;\;
    \inferrule{\Gamma \vdash \sigma : \Delta \\ \Delta,x:A \vdash \\ \Gamma\vdash t : A\app{\sigma}}{\Gamma \vdash \langle\sigma, x \mapsto t\rangle : (\Delta, x : A)}\;
\end{mathpar}
\caption{Typing rules of $\catt$.}
\label{cattrules}
\end{figure}
Substitution and composition preserve those judgements as expected~\cite[Prop.~3]{benjamin_type_2020}, and together with the identity substitutions $\id_\Gamma$,
contexts of \catt thus form a category.

The $\coh$-introduction rule has a side condition (*) which can be satisfied in
two ways:
\[(*)\quad\begin{cases}
  A = u \to_B v \text{ where }u \text{ is full in }\partial^-\Gamma \text{ and }v \text{ full in }\partial^+\Gamma \\
  A = u \to_B v\text{ where both }u\text{ and }v\text{ are full in }\Gamma
\end{cases}\]
In the first case, \(u\) and \(v\) are ways to compose the source and target of
\(\Gamma\), and the term \(\coh(\Gamma:A)[\id]\) is a composition of the
entire pasting context \(\Gamma\). In the second case, \(\coh(\Gamma:A)[\id]\)
is an equivalence between the operations \(u\) and \(v\).

\begin{definition}
    We say a well-typed term $t$ of the form $\coh(\Gamma : A)[\sigma]$ is a \emph{composite} if its typing derivation uses the first side-condition, and a \emph{coherence} if it uses the second side-condition.
\end{definition}

\begin{definition}
  We define the dimension of types, terms in a context, and contexts as follows:
   \begin{gather*}
    \dim(\obj) := -1  \qquad\qquad
    \dim(u \to_A v) := \dim A + 1  \\
    \dim_\Gamma(t) := \dim A + 1 \qquad \text{for $\Gamma \vdash t : A$}\\
    \dim(\Gamma) := \max\{\dim_\Gamma(x) : x \in \Var(\Gamma)\}
   \end{gather*}
\end{definition}
We write just $\dim(t)$ for $\dim_\Gamma(t)$ and $u \to v$ for $u \to_A v$ where there is no ambiguity. We refer to terms in a context $\Gamma$ as \emph{cells} of $\Gamma$,
and $n$\=/dimensional terms as \emph{$n$-cells}. If $\Gamma \vdash t : u \to_A v$, we write $u = \partial^- t$ and $v = \partial^+ t$. These source and target operations can be iterated, and we write $\partial^-_k t$ and $\partial^+_k t$ for the $k$-dimensional source and target of $t$, respectively.

The rules of $\catt$ allow us to construct familiar categorical operations. We
first define the \emph{sphere} and \emph{disc} contexts:

\begin{definition}[Sphere and Disc contexts]\label{def:disc-ctx}
  We define the contexts $\sphere^{n-1}$ and $\disc^n$ and the type
  $\sphere^{n-1} \vdash \spheretype^{n-1}$ for $n\geq 0$:
  \begin{gather*}
    \begin{aligned}
    \sphere^{-1} &:= \varnothing \qquad&\qquad
    \sphere^{n+1} &:= (\sphere^n, d^{n+1}_- : \spheretype^n, d^{n+1}_+ : \spheretype^n)\\
    \spheretype^{-1} &:= \obj \qquad&\qquad
    \spheretype^{n+1} &:= d^{n+1}_- \to_{\spheretype^n} d^{n+1}_+
    \end{aligned}\\
    \disc^n := (\sphere^{n-1}, d^n : \spheretype^{n-1})
\end{gather*}
\end{definition}

Given a context, its \textit{locally-maximal variables} are those which do not
appear in the type of any other variable. By type inference and unification, a
substitution $\Gamma \vdash \sigma : \Delta$ is fully determined by its action
on locally-maximal variables of $\Delta$. If $t_1, \dots t_n$ are the images of
the locally-maximal variables of $\Delta$ under $\sigma$, we often use the
shorthand $u\app{t_1,\dots,t_n} := u\app{\sigma}$.

\begin{definition}
  Give a well-typed term $a$ of dimension $n$ in context \(\Gamma\), we define
  its \textit{identity} \(\id_{a}\) and \textit{iterated identities}
  \(\id^{k}_{a}\) by induction as follows:
    \begin{align*}
        \id_a &:= \coh(\disc^n : d^n \to d^n)[a] \\
        \id^{k+1}_a &:= \id_{\id^{k}_{a}}
    \end{align*}
\end{definition}

\begin{definition}\label{def:comp}
  Given a pasting context $\Gamma$, we define recursively a term $\comp_\Gamma$,
  called its \emph{composite}, as follows:
\[
\comp_\Gamma := \begin{cases}
    d^n & \Gamma = \disc^n \\
    \coh(\Gamma : \comp_{\partial^- \Gamma} \to \comp_{\partial^+ \Gamma})[\id_\Gamma] & \text{otherwise}
\end{cases}
\]
\end{definition}
\noindent
We will also use the following more familiar notation for composites and whiskering:
$$t_1 \s_k \dots \s_k t_n := \comp_\Gamma\app{t_1,\dots,t_n}$$
in the case where \(\Gamma\) is the pasting context obtained from a sequence of discs,
potentially of different dimensions, by identifying the variable \(d^k_+\) of each
disc with the variable \(d^k_-\) of its successor. For instance, the whiskering
of a $2$-cell $a$, a $1$\=/cell $h$, and a $3$\=/cell $p$ is denoted by
$a \s_0 h \s_0 p := \comp_\Gamma\app{a,h,p}$. The pasting context over which is defined is
illustated below:
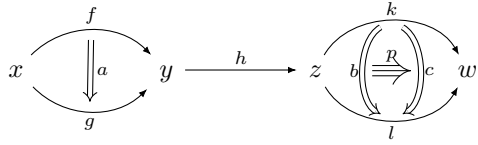
\[\begin{tikzcd}[ampersand replacement=\&, column sep = 1.5cm]
        x \& y \& z \& w
        \arrow[""{name=0, anchor=center, inner sep=0}, "f", bend left = 45, from=1-1, to=1-2]
        \arrow[""{name=1, anchor=center, inner sep=0}, "g"', bend right = 45, from=1-1, to=1-2]
        \arrow["h", from=1-2, to=1-3]
        \arrow[""{name=2, anchor=center, inner sep=0}, "k", bend left = 60, from=1-3, to=1-4]
        \arrow[""{name=3, anchor=center, inner sep=0}, "l"', bend right = 60, from=1-3, to=1-4]
        \arrow["a", shorten <=3pt, shorten >=3pt, Rightarrow, from=0, to=1]
        \arrow[""{name=4, anchor=center, inner sep=0}, "b"', bend right = 60, shorten <=4pt, shorten >=4pt, Rightarrow, from=2, to=3]
        \arrow[""{name=5, anchor=center, inner sep=0}, "c", bend left = 60, shorten <=4pt, shorten >=4pt, Rightarrow, from=2, to=3]
        \arrow["p", shorten <=2pt, shorten >=2pt, Rightarrow, scaling nfold=3, from=4, to=5]
\end{tikzcd}\]
We extend the composition operation
\(t \s_k \dots \s_k t\) to the case $\dim(t_i)\leq k$, by adopting the
convention that in this case $t_1, \dots t_n$ are composable only if they are
all equal, in which case $t \s_k \dots \s_k t := t$. With this convention, for
any \(n,k \in \N\), we have
\[
\partial^\pm (\id^n_x \s_k \id^n_x) = \id^{n-1}_x \s_k \id^{n-1}_x
\]
In weak \(\omega\)\=/categories, those compositions are not stricly associative
nor unital. However, the second side condition of the $\coh$-introduction
rule allows us to construct, for example, unitors:
\begin{align*}
   u_x &:= \coh((x : \obj) :  \id_x \to \id_x \s_0 \id_x )[x]\\
   \rho_f &:= \coh((x,y : \obj, f : x \to y) :  f \s_0 \id_y \to f )[f]
\end{align*}

\begin{definition}~\label{def:id-type} Throughout this article, we will use the
  notation \(\point := (x : \obj)\) for the point context with the variable
  named \(x\). In this context, we also define the following terms and type,
  which play a fundamental role in the construction of the cells \(\eh^{n}_{k,l}\):
  \begin{align*}
    (\id^{n}_x)^{\s_{l}} &:= (\id^{n}_x) \s_{l} (\id^{n}_x)
    &
    I_{k}^{n} &:= (\id^{n}_x)^{\s_{k}} \to (\id^{n}_x)^{\s_{k}}
  \end{align*}
\end{definition}

\subsection{Meta-Operations}

\noindent
Various meta\=/operations have been introduced~\cite{benjamin_type_2020,benjamin_hom_2024,benjamin_invertible_2024,benjamin_naturality_2025} for \catt allowing for the
automatic construction of complex terms. We give a concise presentation of some
of these that we will leverage below.

\paragraph{Suspension}
This meta-operation was defined and implemented for $\catt$ by
Benjamin~\cite[Sec.~3.2]{benjamin_type_2020}, and analogous to the suspension
from topology. Suspending a context $\Gamma$ produces another context
$\Sigma \Gamma$ comprised of two new \(0\)\=/dimensional variables \(N\), \(S\),
as well as all variables of \(\Gamma\). A variable \(x\) of type \(A\) in
context \(\Gamma\) has type \(\Sigma A\), obtained by formally replacing the
base type \(\obj\) with the type \(N\to S\), in context \(\Sigma\Gamma\). This
increases by \(1\) the dimension of the variables, as illustrated in
Figure~\ref{suspensionexample}.

\begin{figure}
    \centering
    \begin{subfigure}{0.49\textwidth}
\[\begin{tikzcd}[ampersand replacement=\&,row sep = 1.1cm, column sep = 1.3 cm]
        x \& y \& z \\
        z \& y \& x
        \arrow[""{name=0, anchor=center, inner sep=0}, "f", bend left = 30, from=1-1, to=1-2]
        \arrow[""{name=1, anchor=center, inner sep=0}, "g"', bend right = 30, from=1-1, to=1-2]
        \arrow["h"', from=1-2, to=1-3]
        \arrow[""{name=2, anchor=center, inner sep=0, pos=0.51}, "{f}", bend left = 30, from=2-2, to=2-3]
        \arrow[""{name=3, anchor=center, inner sep=0}, "{g}"', bend right = 30, from=2-2, to=2-3]
        \arrow["{h}", from=2-1, to=2-2]
        \arrow["a", shorten <=3pt, shorten >=3pt, Rightarrow, from=0, to=1]
        \arrow["{a}", shorten <=3pt, shorten >=3pt, Rightarrow, from=2, to=3]
\end{tikzcd}\]
    \end{subfigure}
    \begin{subfigure}{0.49\textwidth}
\[\begin{tikzcd}[ampersand replacement=\&,row sep = 1cm, column sep = 1.2 cm]
        N \\
        \\
        S
        \arrow[""{name=0, anchor=center, inner sep=0}, "{x}"', bend right = 90, from=1-1, to=3-1]
        \arrow[""{name=1, anchor=center, inner sep=0, pos=0.49}, "{z}", bend left = 90, from=1-1, to=3-1]
        \arrow[""{name=2, anchor=center, inner sep=0}, "{y}"{description}, from=1-1, to=3-1]
        \arrow[""{name=3, anchor=center, inner sep=0, pos=0.53}, "{f}", curve={height = -20pt}, shorten <=8pt, shorten >=8pt, Rightarrow, from=0, to=2]
        \arrow[""{name=4, anchor=center, inner sep=0, pos=0.49}, "{g}"', curve={height = 20pt}, shorten <=8pt, shorten >=8pt, Rightarrow, from=0, to=2]
        \arrow["{h}", shorten <=6pt, shorten >=6pt, Rightarrow, from=2, to=1]
        \arrow["{a}"{description}, shorten <=3pt, shorten >=3pt, Rightarrow, scaling nfold=3, from=3, to=4]
\end{tikzcd}\]
    \end{subfigure}
    \caption{\centering A context (top left), its suspension (right), and its $\{1\}$-opposite (bottom left).}
    \label{suspensionexample}
\end{figure}

\begin{definition}\label{suspdefinition}
  The suspension meta-operation is defined on the syntax of \catt as follows:
  \begin{mathpar}
    \Sigma \varnothing := (N : \obj, S : \obj)
    \and \Sigma(\Gamma, x : A) := (\Sigma\Gamma, x : \Sigma A) \\
    \Sigma \obj := N \to_\obj S
    \and \Sigma (u \to_A v) := \Sigma u \to_{\Sigma A} \Sigma v \\
    \Sigma x := x
    \and \Sigma (\coh(\Gamma : A)[\sigma]) := \coh(\Sigma \Gamma : \Sigma A)[\Sigma \sigma] \\
    \Sigma \langle\rangle := \langle N \mapsto N, S \mapsto S\rangle
    \and \Sigma\langle \sigma, x \mapsto t\rangle :=\langle\Sigma\sigma, x \mapsto \Sigma t \rangle
\end{mathpar}
\end{definition}



\paragraph{Opposites}
Opposites for weak\=/categories have been studied by Benjamin and
Markakis~\cite{benjamin_hom_2024}. Whereas a $1$\=/category has a
single opposite, $\omega$\=/categories have opposites for
each subset $M \subseteq \mathbb{N}_{>0}$, corresponding to flipping the
direction of cells of dimension $n\in M$.
\begin{definition}\label{opdefinition}
  For $M \subseteq \mathbb{N}_{>0}$, the opposite meta-operation \(\op M\) is
  defined on the syntax of \catt as follows:
  \begin{mathpar}
    (\varnothing)^{\op M} := \varnothing \and
    (\Gamma, x : A)^{\op M} := (\Gamma^{\op M}, x : A^{\op M}) \\
    (\obj)^{\op M} := \obj \and x^{\op M} := x \and \langle\rangle^{\op M} := \langle\rangle \\
    (u \to_A v)^{\op M} := \begin{cases}
      u^{\op M} \to_{A^{\op M}} v^{\op M} & \dim u + 1 \notin M \\
      v^{\op M} \to_{A^{\op M}} u^{\op M} & \dim u + 1 \in M \\
    \end{cases} \\
    (\coh(\Gamma : A)[\sigma])^{\op M} := \coh(\Gamma' : A^{\op M}\app{\gamma})[\gamma^{-1}\cir\sigma^{\op M}]\\
    \langle \sigma, x \mapsto t\rangle^{\op M} := \langle \sigma^{\op M}, x^{\op M}\mapsto t^{\op M}\rangle
  \end{mathpar}
  Where $\Gamma'$ is uniquely determined as the pasting context isomorphic to
  $\Gamma^{\op M}$ under a unique isomorphism
  $\gamma : \Gamma^{\op M} \to \Gamma'$ which reorders the variables. When
  $M = \{1\}$, we write $(-)^{\op}$ for $(-)^{\op M}$. This construction is
  illustrated in Figure~\ref{suspensionexample}.
\end{definition}




\paragraph{Chosen inverses}
An \(n\)\=/cell \(f : x \to y\) in an \(\omega\)\=/ is coinductively defined to
be an \emph{equivalence}~\cite{cheng_category_2007} if there is an $n$-cell
$g : y \to x$, together with two equivalences
\begin{align*}
 \varepsilon &: f \s_{n-1} g \to \id_x & \eta &: \id_y \to g \s_{n-1} f
\end{align*}
When this is the case, we say that \(g\) is an \emph{inverse} of \(f\).
Benjamin and Markakis~\cite{benjamin_invertible_2024} have shown that in $\catt$, all coherences
are equivalences, and all composites $t = \coh(\Gamma : A)[\sigma]$ where
$\sigma$ maps all maximal-dimension variables to equivalences are equivalences.
For such equivalences $t : u \to v$, the authors construct a chosen of inverse,
denoted $t^{-1}$, and cancellators $\varepsilon$ and $\eta$. They also prove the
following result

\begin{lemma}\label{diminverse}
Every term \(\Gamma\vdash t : A\) with $\dim(t)>\dim(\Gamma)$ is an equivalence.
\end{lemma}

We extend the notion of equivalence to that of congruence between terms of
\(\catt\). This notion is more generic insofar that it allows for two cells with
different types to be congruent. The cells \(\eh^{n}_{k,l}\) that we will
construct in this paper are congruences.

\begin{definition}\label{def:congruence}
  The \emph{congruence} is the smallest equivalence relation such that an
  $n$\=/cell is congruent to its composite in any dimension with a coherence,
  and equivalent cells are congruent.
\end{definition}

\paragraph{Functorialisation and Naturality}\label{sec:functorialisation}
Composites in \(\omega\)\=/categories are functorial with respect to their
arguments, while coherences are natural. This is made precise by the
\emph{functorialisation}~\cite[§3.4]{benjamin_type_2020} and the naturality
meta-operations~~\cite{benjamin_naturality_2025}. Both operations can be seen as
the \emph{depth-0} and \emph{depth-1} cases of the same inductive scheme
described below. Here, the \emph{depth} is a parameter defined for contexts
$\Gamma$, types \(\Gamma \vdash A\), terms \(\Gamma \vdash t : A\) and
substitutions $\Gamma \vdash \sigma : \Delta$, and for a set of variables
\(X \subseteq \Var \Gamma\) by
\begin{align*}
  \depth_X t &= \max \{\dim t - \dim x : x\in \supp(t)\cap X\} \\
  \depth_X A &= \max \{\dim A - \dim x : x\in \supp(A)\cap X\}
\\
  \depth_X \sigma &= \max\{ \depth_X x[\sigma] : x\in \Var\Delta\}
\\
  \depth_X \Gamma &= \depth_X(\id_\Gamma)
\end{align*}
where \(\max\varnothing = -1\). The scheme further requires that the set \(X\) is
\emph{up-closed}, meaning that if a variable \(x\in X\) appears in the support
of some variable \(y\in \Var(\Gamma)\), then also \(y\in X\). To present the
definition, we introduce the preimage \(X_\sigma\) of a set of variables
\(X\) under a substitution \(\Gamma\vdash \sigma : \Delta\) as follows:
\[X_\sigma = \{ y \in \Var(\Delta) : \supp(y[\sigma])\cap X \neq \varnothing\} \]

The construction proceeds recursively on the derivation tree to produce for
every context \(\Gamma\vdash\) and every up-closed \({X\subseteq \Var(\Gamma)}\)
such that \(\depth_X\Gamma \le 1\), a new context \(\Gamma\uparrow X\) together
with substitutions:
\[
  \Gamma\uparrow X \vdash \inc^\pm : \Gamma
\]
Moreover, it produces for every term \(\Gamma \vdash t : A\) such that \(0\le \depth_Xt \le 1\),
a new term:
\[
  \Gamma\uparrow X \vdash t\uparrow X : A \uparrow^t X
\]
and for every substitution \(\Gamma \vdash \sigma : \Delta\) such that \(\depth_X\sigma \le 1\),
a new substitution:
\[
  \Gamma\uparrow X\vdash \sigma\uparrow X : \Delta\uparrow X_\sigma
\]
When \(\Gamma\vdashps\) is a pasting context, it also produces for full
types \(\Gamma\vdash A\) such that \(0\le \depth_X(\coh(\Gamma:A)[\id])\le 1\), a
term:
\[
  \Gamma\uparrow X \vdash \coh(\Gamma:A)\uparrow X : A\uparrow^{\coh(\Gamma : A)[\id]} X
\]

\paragraph{- Contexts} This proceedure duplicates the variables in \(X\) and adds
a connecting variable relating the two copies. More formally, it is given by:
\begin{gather*}
  \varnothing \uparrow \varnothing
  := \varnothing \\
  \intertext{If \(x\notin X\), we define:}
  (\Gamma, x : A)\uparrow X
   := (\Gamma \uparrow X, x : A)\\
  \intertext{If \(x\in X\), we let \(X' = X\setminus\{x\}\) and define:}
  (\Gamma, x : A)\uparrow X
  :=  (\Gamma \uparrow X', x^- : A, x^+ : A,\fun{x} : A\uparrow^x X)
\end{gather*}
The inclusion substitutions \(\inc^\pm\) are determined by:
  \[y\app{\inc^\pm} = \begin{cases}
    y & \text{if } y \notin X \\
    y^\pm &\text{if }  y\in X
\end{cases}\]

\paragraph{- Types} The type \(A\uparrow^t X\) relates the terms \(t\app{\inc^-}\)
and \(t\app{\inc^+}\). It it an arrow type of the form \(L_{A,t,X} \to R_{A,t,X}\),
where if \(A = \obj\) we define $L_{A,t,X}$ as \(t\app{\inj^-}\) and $R_{A,t,X}$ as \(t\app{\inj^+}\),
and when \(A = u\to v\) they are given by:
  \begin{align*}
    L_{A,t,X} &:=\begin{cases}
      t\app{\inj^- } \s_{n-1} (v \uparrow X)
      & \text{if } \supp(v) \cap X \neq \varnothing\\
      t\app{\inj^-} &\text{otherwise}
    \end{cases} \\
    R_{A,t,X} &:=\begin{cases}
      (u \uparrow  X)\s_{n-1} t\app{\inj^+}
      & \text{if }\supp(u) \cap X \neq \varnothing \\
      t\app{\inj^+} & \text{otherwise}
    \end{cases}
  \end{align*}

  \paragraph{- Terms} The term \(t\uparrow X\) is defined recursively by:
  \begin{align*}
    x \uparrow X &:= \fun{x} \\
    \coh(\Delta: A)[\sigma] \uparrow X &:= (\coh(\Delta: A)\uparrow X_\sigma)\app{\sigma \uparrow X}
  \end{align*}
  \paragraph{- Substitutions} For $x \notin X_\sigma$, \(\sigma\uparrow X\) is given by:
     \[ x\app{\sigma\uparrow X} := x\app{\sigma} \]
     and for \(x\in X_\sigma\) by:
      \begin{align*}
        x^\pm\app{\sigma \uparrow X} &:= x\app{\sigma \cir \inj^\pm} \\
        \fun x\app{\sigma \uparrow X} &:= x\app{\sigma} \uparrow X
      \end{align*}
    \paragraph{- Full types} Let \(t = \coh(\Gamma : A)[\id]\). When \(\depth_X(\Gamma) = 0\),
    then \(\Gamma\uparrow X\) is again a pasting diagram and \(A\uparrow^t X\) is full, allowing
    one to define
    \[
      \coh(\Gamma:A) \uparrow X = \coh(\Gamma\uparrow X : A\uparrow^t X)[\id]
    \]
    When \(\depth_X(\Gamma) = 1\), then \(\Gamma\uparrow X\) is no longer a pasting diagram, and
    \(t\) is a composite. The term \(\coh(\Gamma:A) \uparrow X\) has been constructed by Benjamin et al.~\cite{benjamin_naturality_2025} in this case.

    \begin{example}
      The functorialisation of the composite \(f\s_0 g\) of two \(1\)\=/cells with respect to \(f\)
      is the whiskering $\comp_\Gamma\app{\fun{f},g}$ where $\Gamma$ is the following pasting context:
      \[\begin{tikzcd}[ampersand replacement=\&]
        x
        \ar[r, bend left = 35, "f", ""{below, name = A}]
        \ar[r, bend right = 35, "g"', ""{name = B}]
        \ar[from = A, to = B, Rightarrow, "a"]
        \& y\ar[r, "h"]
        \& z
      \end{tikzcd}\]
    \end{example}

    \begin{example}
      Consider the right unitor $\rho_{f} : f \s_0 \id_y \to f$ in the context $\Gamma_f = (x, y : \obj, f : x \to_\obj y)$. Letting $X = \{f\}$, we get a term \(\rho_f \uparrow X\) filling the following square:
      \begin{equation*}\label{rhonaturality}
        \begin{tikzcd}[ampersand replacement=\&]
          {f^- \s_0 \id_y} \& {f^-} \\
          {f^+ \s_0 \id_y} \& {f^+}
          \arrow["{\rho_{f^-}}", from=1-1, to=1-2]
          \arrow["{\fun{f} \s_0 \id_y}"', from=1-1, to=2-1]
          \arrow["{\rho_{\fun{f}}}"{description}, Rightarrow, from=1-2, to=2-1]
          \arrow["{\fun{f}}", from=1-2, to=2-2]
          \arrow["{\rho_{f^+}}"', from=2-1, to=2-2]
        \end{tikzcd}
      \end{equation*}
    \end{example}

\begin{example}\label{hexcontext}
  Consider the context:
  \[
    \mathbf{3} := (x, y, z, w : \s, f : x \to y, g : y \to z, h : z \to w)
  \]
  Letting $X = \{f,y,g,z,h\}$ the context $\mathbf{3} \uparrow X$ is given by:
  
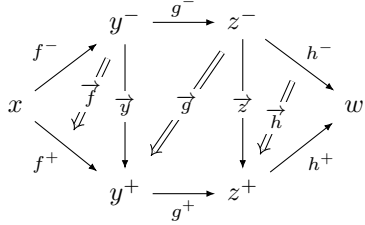
\begin{equation*}
    \begin{tikzcd}[ampersand replacement=\&]
        \& {y^-} \& {z^-} \\
        x \&\&\& w \\
        \& {y^+} \& {z^+}
        \arrow["{g^-}", from=1-2, to=1-3]
        \arrow["{\fun{y}}"{description}, from=1-2, to=3-2]
        \arrow[""{name=0, anchor=center, inner sep=0}, "{h^-}", from=1-3, to=2-4]
        \arrow["{\fun{g}}"{description}, Rightarrow, shorten >=6pt, shorten <=6pt, from=1-3, to=3-2]
        \arrow["{\fun{z}}"{description}, from=1-3, to=3-3]
        \arrow["{f^-}", from=2-1, to=1-2]
        \arrow[""{name=1, anchor=center, inner sep=0}, "{f^+}"', from=2-1, to=3-2]
        \arrow["{g^+}"', from=3-2, to=3-3]
        \arrow["{h^+}"', from=3-3, to=2-4]
        \arrow["{\fun{f}}"{description}, shorten >=6pt, shorten <=6pt, Rightarrow, from=1-2, to=1]
        \arrow["{\fun{h}}"{description}, shorten <=6pt, shorten >=6pt, Rightarrow, from=0, to=3-3]
    \end{tikzcd}
  \end{equation*}
  Consider now the term $t = f \s_0 g \s_0 h$ over $\mathbf{3}$. Its naturality with
  respect to \(X\) is a term over \(\mathbf{3}\uparrow X\) of type:
  \begin{align*}
    f^- \s_0 g^- \s_0 h^- \to f^+ \s_0 g^+ \s_0 h^+
  \end{align*}
  Given $2$-cells $a,b,c$ whose boundaries match as in the context \(\mathbf{3}\uparrow X\),
  we define their \emph{hexagonal composite}:
  \[
    \hexcomp\app{a,b,c} := (t \uparrow X)\app{a,b,c}
  \]
\end{example}
Let \(\sigma\) be a substitution substitution whose target is an iterated
suspension of the context of the hexagonal composition:
\begin{align*}
  \sigma :\ & \Gamma \to \Sigma^k (\mathbf{3} \uparrow X)
\end{align*}
Denoting \(a\), \(b\) and \(c\) the respective images of \(\fun f\),
\(\fun g\) and \(\fun h\) under the action of \(\sigma\), we use suspension
implicitly and write:
\[
  \hexcomp\app{a,b,c} := (\Sigma^k (t \uparrow X))\app{\sigma}
\]
We will use the hexagonal composition and its suspensions in for construction
of the repadding in Sec.~\ref{padding}.

\section{Padding and Repadding}\label{padding}
\noindent
This section is dedicated to the presentation of our theory of padding, which lies at the
heart of our method to construct congruences. Our technique is inspired by the
padding constructions in Fujii et al.~\cite{fujii_equivalences_2024} and Finster
et al.~\cite{finster_type_2022}, but takes a more axiomatic approach, describing
the general shape of such paddings.

\subsection{Padding}\label{subsec:padding}
\noindent
Our method for padding cells involves recursively adjusting their boundaries as
necessary, starting with the lowest dimension where they differ, proceeding
until the cell has the desired type. To capture this dimensionwise recursive
structure, we introduce a notion of filtration.

\begin{definition}
  A \emph{filtration} $\mathbf{\Gamma} = (\Gamma^i,v^i,\sigma^{i})_{i=m}^n$ of
  \emph{height} \(m\) constitutes a sequence of contexts \(\Gamma^{i}\) of dimension
  \(i\), together with a chosen variable \(v^{i}\) in context \(\Gamma^{i}\) and
  a sequence of substitutions \(\sigma^{i}\) for \(m < i \leq n\) satisfying:
  \begin{align*}
    \Gamma^{i} &\vdash \sigma^{i} : (\Gamma^{i-1} \uparrow v^{i-1}) &
    \fun{v^{i}}&\app{\sigma^{i+1}} = v^{i+1}
  \end{align*}
  A family of types $\mathbf{A} = (A^{i})_{i=m}^{n}$ is \emph{adapted} to the
  filtration \(\mathbf{\Gamma}\) when \(\Gamma^{m} \vdash v^m : A^m\), and for
  all \(i \in \{m+1,\ldots,n\}\), there exist terms \(s^{i},t^{i}\) satisfying:
  \begin{gather*}
    \begin{aligned}
      \Gamma&\vdash s^{i}:A^{i}\app{\sigma^{i+1}}
    \quad&\quad \Gamma&\vdash t^{i}:A\app{\sigma^{i+1}}
    \end{aligned}\\
    A^{i+1}=s^i\to_{A^i\app{\sigma^{i+1}}}t^i
  \end{gather*}
  Finally a set of of \emph{padding data} \(\mathbf{p} = (p^i,q^i)_{i=m}^{n-1}\)
  for the type family \(\mathbf{A}\) adapted to the filtration
  \(\mathbf{\Gamma}\) is defined mutually inductively together with its
  \emph{associated padding} \(\padded_{\mathbf{p}}\). Padding data consists in a
  family of terms \(p^{i}\) and \(q^{i}\) satisfying:
  \begin{gather*}
    \Gamma^{i+1} \vdash p^{i} : s^{i} \to \padded^{i}_\mathbf{p}\app{\inj^- \cir \sigma^{i+1}} \\
    \Gamma^{i+1} \vdash q^{i} : \padded^{i}_\mathbf{p}\app{\inj^+ \cir
    \sigma^{i+1}} \to t^{i} \\
    v^{i-1}\notin\supp(p^i)\cup\supp(q^i)
  \end{gather*}
  Its associated padding is a term
  \(\Gamma^{i}\vdash \padded^{i}_{\mathbb{p}} : A^{i}\) defined for
  \(m\leq i \leq n\) by:
  \begin{align*}
    \padded^m_\mathbf{p} &:= v^m \\
    \padded^{i+1}_\mathbf{p} &:= p^i \s_i (\padded^i_\mathbf{p} \uparrow v^i)\app{\sigma^{i+1}} \s_i q^i \tag{$\dagger$}\label{definingeq}
  \end{align*}
\end{definition}

\begin{figure}
  \begin{align*}
    \begin{array}{cc}
      \begin{tikzcd}[ampersand replacement=\&]
        \bdot
        \ar[r,"{p^{0}}"]
        \& \bdot
        \ar[r,"v^{1}"]
        \& \bdot
        \ar[r,"{q^{0}}"]
        \& \bdot
      \end{tikzcd}
      &
        \begin{tikzcd}[ampersand replacement=\&]
          \bdot
          \ar[r,"{p^{0}\app{\sigma^{2}}}"]
          \ar[rrr, bend left=60,""'{name=A'}, "s^{1}"]
          \ar[rrr, bend right=60, ""{name=D}, "t^{1}"']
          \& \bdot
          \ar[r, bend left=40, ""{name = B}, ""'{name = B'}]
          \ar[r, bend right=40, ""{name = C}, ""'{name=C'}]
          \& \bdot
          \ar[r,"{q^{0}\app{\sigma^{2}}}"]
          \& \bdot
          \ar[from=A', to=B, Rightarrow, "p^{1}"]
          \ar[from=B', to=C, Rightarrow, "v^{2}" ]
          \ar[from=C', to=D, Rightarrow, "q^{1}"]
        \end{tikzcd} \\
 \text{(a) \(1\)-dimensional padding \(\padded^{1}\)}
      & \text{(b) \(2\)-dimensional padding \(\padded^{2}\)}
      \end{array} \\\vspace{10pt}
    \begin{array}{c}
         \begin{tikzcd}[ampersand replacement=\&, column sep = 0.7cm]
        \bdot \& \bdot \& \bdot \& \bdot \& \bdot \& \bdot \& \bdot \& \bdot
        \arrow[""{name=0, anchor=center, inner sep=0},bend left = 40, from=1-1, to=1-2]
        \arrow[""{name=1, anchor=center, inner sep=0}, bend right = 40, from=1-1, to=1-2]
        \arrow["{p^2}", Rightarrow, scaling nfold=3, from=1-2, to=1-3]
        \arrow["{p^0\app{\sigma^3}}", from=1-3, to=1-4]
        \arrow[""{name=2, anchor=center, inner sep=0}, "{s^1\app{\sigma^3}}", bend left = 60, from=1-3, to=1-6]
        \arrow[""{name=3, anchor=center, inner sep=0}, "{t^1\app{\sigma^3}}"', bend right = 60, from=1-3, to=1-6]
        \arrow[""{name=4, anchor=center, inner sep=0}, bend left = 60, from=1-4, to=1-5]
        \arrow[""{name=5, anchor=center, inner sep=0}, bend right = 60, from=1-4, to=1-5]
        \arrow["{q^0\app{\sigma^3}}", from=1-5, to=1-6]
        \arrow["{q^2}", Rightarrow, scaling nfold=3, from=1-6, to=1-7]
        \arrow[""{name=6, anchor=center, inner sep=0}, bend left = 40, from=1-7, to=1-8]
        \arrow[""{name=7, anchor=center, inner sep=0}, bend right = 40, from=1-7, to=1-8]
        \arrow["{s^2}", shorten <=2pt, shorten >=2pt, Rightarrow, from=0, to=1]
        \arrow["{p^1\app{\sigma^3}}", shorten <=3pt, shorten >=3pt, Rightarrow, from=2, to=4]
        \arrow[""{name=8, anchor=center, inner sep=0}, bend right = 60, shorten <=2pt, shorten >=2pt, Rightarrow, from=4, to=5]
        \arrow[""{name=9, anchor=center, inner sep=0}, bend left = 60, shorten <=2pt, shorten >=2pt, Rightarrow, from=4, to=5]
        \arrow["{q^1\app{\sigma^3}}", shorten <=3pt, shorten >=3pt, Rightarrow, from=5, to=3]
        \arrow["{t^2}", shorten <=2pt, shorten >=2pt, Rightarrow, from=6, to=7]
        \arrow["{v^3}", shorten <=2pt, shorten >=2pt, Rightarrow, scaling nfold=3, from=8, to=9]
\end{tikzcd} \\
\text{(c) $3$-dimensional padding $\padded^3$}
    \end{array}
    \end{align*}
    \caption{Paddings of height 0 with simplified notation, e.g. writing $p^0\app{\sigma^3}$ for $p^0\app{\inj^-_{\Gamma^1}\cir \sigma^2 \cir \inj^-_{\Gamma^2}\cir \sigma^3}$.}
    \label{paddingfigure}
\end{figure}

Figure~\ref{paddingfigure} illustrates the shape of the paddings of height \(0\)
and dimension up to \(2\). We now define the \emph{unbiased padding},
illustrated in Figure~\ref{fig:padding-examples}. This is the padding appearing in
the type of the Eckmann-Hilton cells $\eh^n_{k,l}$, as well as in the respective
source and targets of $X_3$ and $X_4$ in Figure~\ref{2dEH}. It transports terms
from type \(I^{n}_{k}\) of Def.~\ref{def:id-type} to type \(I^{n}_{l}\) using
\emph{generalised unbiased unitors}. It will turn out to be self-dual, i.e.
invariant under opposites, a crucial property necessary to define the cells
\(\eh^n_{k,l}\) and assemble into the commutativity cells $\EH^n_{k,l}$.

\begin{definition}[Unbiased unitors and unbiased paddings]\label{def:unbiasedpadding}
  For \(n\geq 2\) and \(k,l < n\), we denote
  \(m = \min \{k,l\} + 1\), and we introduce the filtration
  \(\mathbf{\Gamma}^{n}_{k,l} =
  (\Gamma^{i}_{l},v^{i}_{l},\sigma^{i})_{i=m}^{n}\) where
  \begin{align*}
    \Gamma^{i}_{l} &= (x : \obj, v^i_l : I^{i-1}_{l})
    & \sigma^{i}_{l} &= \sub{x\mapsto x, \fun{v^{i}_{l}}\mapsto v^{i+1}_{l} }
  \end{align*}
  We then define a set of padding data
  \(\mathbf{u}^{n}_{k,l} = (p^{i}_{k,l},q^{i}_{k,l})_{i=m}^{n}\) for the family
  \((I^{i}_{k})_{i=m}^n\) adapted to \(\mathbf{\Gamma}^{n}_{k,l}\) with
  associated padding denoted \(\padded^i_{k,l}\), by:
  \begin{align*}
    p_{k,l}^i
  &:= \coh(\point : (\id^i_x)^{\s_{k}}\to \padded^i_{k,l}\app{(\id^i_x)^{\s_{l}}})[x] \\
  q_{k,l}^i
  &:=\coh(\point : \padded^i_{k,l}\app{(\id^i_x)^{\s_{l}}} \to
    (\id^i_xx)^{\s_{k}})[x] = \inv{(p_{k,l}^i)}
  \end{align*}
  We call the terms \(p^{i}_{k,l}\) of this family the \emph{generalised
    unbiased unitors} and its associated padding \(\padded^{i}_{k,l}\) the
  \emph{unbiased padding}.
\end{definition}

We now aim to prove that the unbiased padding satisfies a self-duality property. To do this, we will need the following lemma:
\begin{lemma}~\label{lemma:unbiased-filtration-op}
  Let $0\leq l$ and \(i > l\). The contexts \(\Gamma^{i}_{l}\) appearing in the
  unbiased filtration satisfy, for any \(M \subseteq \N_{>0}\):
  \[
    (\Gamma^{i}_{l})^{\op M} = \Gamma^{i}_{l}
  \]
\end{lemma}
\begin{proof}
It suffices to show the following
\[
(I^{i-1}_l)^{\op M} = I^{i-1}_l:
\]
By Lemma~\ref{opcomp}, we have:
\begin{align*}
    (I^{i-1}_l)^{\op M} &= ((\id^{i-1}_x)^{\s_l} \to (\id^{i-1}_x)^{\s_l})^{\op M}
    \\ &= ((\id^{i-1}_x)^{\s_l})^{\op M} \to ((\id^{i-1}_x)^{\s_l})^{\op M}
    \\ &= \id^{i-1}_x \s_l \id^{i-1}_x \tag*{\qedhere}
\end{align*}
\end{proof}

\begin{proposition}[Self-duality of unbiased padding]\label{lemma:symmetry-main}
  Let $0\leq k,l$ and let $m:=\min\{k,l\}+1$. For any $r$, and for any
  $m\leq i $:
\begin{align*}
  (\padded^i_{k,l})^{\op\{r\}}
  &=\padded^i_{k,l}
    \tag{a}\\
  \intertext{Furthermore, for \(i>m\):}
  (p_{k,l}^{i-1})^{\op\{r\}}
  &= \begin{cases}
    p_{k,l}^{i-1} & r\neq i \\
    q_{k,l}^{i-1} & r = i
  \end{cases} \\
  (q_{k,l}^{i-1})^{\op\{r\}}
  & = \begin{cases}
    q_{k,l}^{i-1} & r\neq i \\
    p_{k,l}^{i-1} & r = i
  \end{cases} \tag{b}
\end{align*}
\end{proposition}

\begin{proof}
  We proceed by induction on $i$.

  When $i = m$, then have the equality on the terms
  \(v_l^i = \padded^{i}_{k,l} = (\padded^i_{k,l})^{\op r}\).

  For $i > m$, we first show (b). Recall the definitions:
  \begin{align*}
    \point &:= (x : \obj) \\
    p_{k,l}^{i-1} &:= \coh(\point : (\id^{i-1}_x)^{\s_k}  \to \padded^{i-1}_{k,l}\app{(\id^{i-1}_x)^{\s_l}})[x] \\
    q_{k,l}^{i-1} &:=  \coh(\point : \padded^{i-1}_{k,l}\app{(\id^{i-1}_x)^{\s_l}} \to (\id^{i-1}_x)^{\s_k} )[x]
  \end{align*}
  The context \(\point\) satisfies $\point = \point^{\op\{r\}} = \point'$ and
  $\gamma_\point = \id_{\point}$. We carry out the computation of
  $(p^{i-1}_{k,l})^{\op\{r\}}\app{x}$, The one of
  $(q_{k,l}^{i-1})^{\op\{r\}}\app{x}$ being similar. If $n \neq i$, we have, by
  induction and Lemmas~\ref{opcomp} and \ref{opsub}:
  \begin{align*}
    (&p_{k,l}^{i-1})^{\op\{r\}} \\
     &= \coh(\point : ((\id^{i-1}_x)^{\s_k})^{\op\{r\}} \to (\padded^{i-1}_{k,l}\app{(\id^{i-1}_x)^{\s_l}})^{\op\{r\}})[x] \\
     &= \coh(\point : (\id^{i-1}_{x})^{\s_k} \to (\padded^{i-1}_{k,l})^{\op\{r\}}\app{(\id^{i-1}_{x})^{\s_l}})[x]
    \\
     &= \coh(\point : (\id^{i-1}_{x})^{\s_k} \to \padded^{i-1}_{k,l}\app{(\id^{i-1}_{x})^{\s_l}})[x]\\
     &= p_{k,l}^{i-1}
  \end{align*}
  And similarly, if $r = i$, we have:
  \begin{align*}
    (&p_{k,l}^{i-1})^{\op\{r\}} \\
     &= \coh(\point :
       ((\padded^{i-1}_{k,l}\app{(\id^{i-1}_x)^{\s_l}})^{\op\{i\}} \to (\id^{i-1}_x)^{\s_k})^{\op\{i\}} )[x] \\
     &= \coh(\point :
       (\padded^{i-1}_{k,l})^{\op\{r\}}\app{(\id^{i-1}_{x})^{\s_l}} \to (\id^{i-1}_{x})^{\s_k})[x]
    \\
     &= \coh(\point : \padded^{i-1}_{k,l}\app{(\id^{i-1}_{x})^{\s_l}}\to (\id^{i-1}_{x})^{\s_k})[x]\\
     &= q_{k,l}^{i-1}
  \end{align*}

  We now show (a). We have:
  $$\padded^i_{k,l} = p_{k,l}^{i-1} \s_{i-1} (\padded^{i-1}_{k,l} \uparrow
  v_l^{i-1})\app{\sigma^i} \s_{i-1} q_{k,l}^{i-1}$$ Denote the middle term
  $u := ((\padded^{i-1}_{k,l} \uparrow v_l^{i-1})\app{\sigma^i})$. We remark
  that by Lemma~\ref{lemma:unbiased-filtration-op}, the substitution
  \(\op^{\uparrow}_{\Gamma^{i-1}_{l},v^{i-1}_{l},\{r\}}\) is the identity, and
  \((\sigma^{i})^{\op} = \sigma^{i}\). Then, induction, together with
  Lemmas~\ref{opsub} and~\ref{opfunc}, if we have:
  \begin{align*}
    u^{\op \{r\}}
    &= (\padded^{i-1}_{k,l} \uparrow v_l^{i-1})^{\op\{r\}}\app{\sigma^i} \\
    &= ((\padded^{i-1}_{k,l})^{\op\{r\}} \uparrow v_l^{i-1})\app{\sigma^{i}} \\
    &= (\padded^{i-1}_{k,l}\uparrow v_l^{i-1})\app{\sigma^{i}} \\
    &=u
  \end{align*}
  Using this equation and Lemma~\ref{opcomp} and the inductive hypothesis, for
  $r\neq i$ we have:
  \begin{align*}
    (\padded^i_{k,l})^{\op\{r\}}
    &= (p_{k,l}^{i-1})^{\op\{r\}} \s_{i-1} u^{\op \{r\}} \s_{i-1} (q_{k,l}^{i-1})^{\op\{r\}}\\
    &= p_{k,l}^{i-1} \s_{i-1} u^{\op\{r\}} \s_{i-1} q_{k,l}^{i-1} \\
    &= \padded^i_{k,l}
  \end{align*}
  Similarly, if $r = i$:
    \begin{align*}
      (\padded^i_{k,l})^{\op\{i\}}
      &= (q_{k,l}^{i-1})^{\op\{i\}} \s_{i-1} u^{\op \{i\}} \s_{i-1} (p_{k,l}^{i-1})^{\op\{i\}}\\
      &= p_{k,l}^{i-1} \s_{i-1} u^{\op\{i\}} \s_{i-1} q_{k,l}^{i-1} \\
      &= \padded^i_{k,l} \tag*{\qedhere}
    \end{align*}
  \end{proof}
  
We now introduce the \emph{generalised biased unitors} and their associated
\emph{biased paddings}. These are illustrated in Figure~\ref{fig:padding-examples}
and play a key role in our construction of the cells \(\eh^{n}_{n-1,0}\),
appearing in $X_1$ of Figure~\ref{2dEH}. To shorten the construction, we leverage
a duality argument, allowing us to focus only on right unitors. In fact, we
define two flavours $\rho^n, \tilde \rho^n$ of right unitors and respective
associated padding $\padded^n_\rho, \padded^n_{\tilde \rho}$, the first
appearing in the construction of \(\eh^{n}_{n-1,0}\) and the latter in that of
\(\eh^{n}_{0,n-1}\).

\begin{definition}[Generalised unitors and biased paddings]\label{def:biased-padding}
  We define the filtration
  \(\mathbf{\Gamma}^{n}_{\rho} = (\Gamma^{i}_{\rho},v^{i},\sigma^{i}_{\rho})\) by:
  \begin{align*}
    \Gamma^i_{\rho} &= (\sphere^i, v^i : d^{i-1}_- \s_0 \id^{i-1} (d^0_+) \to d^{i-1}_+ \s_0
                   \id^{i-1} (d^0_+)) \\
    \sigma^{i}_{\rho} &= \sub{d^j_\pm \mapsto d^j_\pm, (v^{i-1})^\pm
                        \mapsto d^{i-1}_\pm, \fun{(v^{i-1})}\mapsto v^{i}}
  \end{align*}
  We then define the padding data $\mathbf{p}^n_\rho=(p_{\rho}^i,q_{\rho}^i)_{i=1}^{n-1}$ for the type family $(d^{i-1}_- \to d^{i-1}_+)_{i-1}^n$ adapted to
  \(\mathbf{\Gamma}^{n}_{\rho}\), whose associated padding we denote
  \(\padded^{n}_{\rho}\), as follows:
  \begin{gather*}
    \rho^i := \coh(\disc^i : \padded^i_{\rho}\app{d^i \s_0 \id^i_{d^0_+}} \to d^i)[\id_{\disc^i}] \\
    \begin{aligned}
      p_{\rho}^i &:= (\rho^i)^{-1}\app{d^i_-}
      \qquad&\qquad q_{\rho}^i &:= \rho^i\app{d^i_+}
    \end{aligned}
  \end{gather*}
  We also define the filtration
  \(\mathbf{\Gamma}_{\tilde{\rho}}^{n} = (\disc^i,d^{i},\sigma^{i})_{i=1}^{n}\)
  where \(\disc^{i}\) is the \(i\)-disc context context of
  Def.~\ref{def:disc-ctx}, and \(\sigma^{i}\) is the isomorphism between
  $(\disc^{i-1} \uparrow d^{i-1})$ and $\disc^i$. Consider the following type
  family adapted to the filtration \(\mathbf{\Gamma}^{n}_{\tilde\rho}\):
  \[
    (d^{i-1}_- \s_0 \id^{i-1}_{d^0_+} \to d^{i-1}_+ \s_0 \id^{i-1}_{d^0_+})_{i=1}^n
  \]
  We define padding data
  $\mathbf{p}^n_{\tilde\rho} = (p_{\tilde\rho}^i,q_{\tilde\rho}^i)_{i=1}^{n-1}$
  for this type family, whose associated padding we call
  \(\padded^{i}_{\tilde\rho}\), as follows:
  \begin{gather*}
    \tilde\rho^i := \coh(\disc^i :  \padded^i_{\tilde\rho}\app{d^i_-} \to d^i \s_0 \id^i_{d^0_+})[\id_{\disc^i}]  \\
    \begin{aligned}
      p_{\tilde\rho}^i &:= (\tilde\rho^i)^{-1}\app{d^i_-} \qquad & \qquad
    q_{\tilde\rho}^i &:= \tilde\rho^i\app{d^i_+}
    \end{aligned}
  \end{gather*}
  We call the coherences \(\rho^{n}, \tilde \rho^n\) \emph{generalised right unitors}. The
  paddings \(\padded^{n}_{\rho}\) and \(\padded^{n}_{\tilde\rho}\) are the
  \emph{right-biased paddings}. Using the duality, we define the
  \emph{generalised left unitors} and
  \emph{left-biased paddings} as follows:
  \begin{align*}
    \lambda^{n} &:= (\rho^{n})^{\op}
    & \tilde\lambda^{n} &:= (\tilde\rho^{n})^{\op} \\
    \padded^{n}_{\lambda} &:= (\padded^{n}_{\rho})^{\op}
    & \padded^{n}_{\tilde\lambda} &:= (\padded^{n}_{\tilde\rho})^{\op}
  \end{align*}
\end{definition}

\begin{figure}
  \centering
  \[
    \begin{array}{cc}
      \begin{tikzcd}[ampersand replacement=\&, row sep=-2]
        \& \bdot \ar[dd, Rightarrow, "v^{2}_{0}"]
        \ar[rd, bend left=10] \\
        \bdot
        \ar[ru, bend left=10]
        \ar[rd, bend right=10]
        \ar[rr, bend left=90, ""{below,name=A}]
        \ar[rr, bend right=90,""{name=B}]
        \&\& \bdot
        \\
        \& \bdot \ar[ru, bend right=10]
        \ar[from=A, to=1-2, Rightarrow, "p^{1}_{1,0}"]
        \ar[to=B, from=3-2, Rightarrow,"q^{1}_{1,0}"]
      \end{tikzcd}
      &
        \begin{tikzcd}[ampersand replacement=\&, row sep=-2]
          \& \bdot \ar[dd, Rightarrow, "v^{2}"]
          \ar[rd, bend left=10] \\
          \bdot
          \ar[ru, bend left=10, "d^{1}_{-}"{sloped}]
          \ar[rd, bend right=10, "d^{1}_{+}"'{sloped}]
          \ar[rr, bend left=90, ""{below,name=A}, "d^{1}_{-}"]
          \ar[rr, bend right=90, ""{name=B}, "d^{1}_{+}"']
          \&\& \bdot
          \\
          \& \bdot \ar[ru, bend right=10]
          \ar[from=A, to=1-2, Rightarrow, "(\rho^1)^{\text{-}1}"]
          \ar[to=B, from=3-2, Rightarrow,"\rho^1"]
        \end{tikzcd}
      \\
     \text{(a) unbiased padding \(\padded^{2}_{1,0}\)}
      & \text{(b) biased padding \(\padded_{\rho}^{2}\)}
      \\
        \begin{tikzcd}[ampersand replacement=\&]
          \& \bdot
          \ar[rd, bend left=30] \\
          \bdot
          \ar[ru, bend left=30]
          \ar[rd, bend right=30]
          \ar[rr, bend left=30,""{below,name=A}]
          \ar[rr, bend right=30,""{name=B}]
          \&\& \bdot
          \\
          \& \bdot \ar[ru, bend right=30]
          \ar[to=A, from=1-2, Rightarrow, "p^{1}_{0,1}"]
          \ar[from=B, to=3-2, Rightarrow,"q^{1}_{0,1}"]
          \ar[from=A, to=B, Rightarrow, "v^{2}_{1}"]
        \end{tikzcd}
      &
        \begin{tikzcd}[ampersand replacement=\&]
          \& \bdot
          \ar[rd, bend left=30] \\
          \bdot
          \ar[ru, bend left=30, "d^{1}_{-}"{sloped}]
          \ar[rd, bend right=30, "d^{1}_{+}"'{sloped}]
          \ar[rr, bend left=30,""{below,name=A}, "d^{1}_{-}"{near start, sloped}]
          \ar[rr, bend right=30,""{name=B}, "d^{1}_{+}"{below, near start, sloped}]
          \&\& \bdot
          \\
          \& \bdot \ar[ru, bend right=30]
          \ar[to=A, from=1-2, Rightarrow, "(\tilde\rho^1)^{\text{-}1}"]
          \ar[from=B, to=3-2, Rightarrow,"\tilde\rho^1"]
          \ar[from=A, to=B, Rightarrow, "d^{2}"]
        \end{tikzcd} \\
      \text{(c) unbiased padding \(\padded^{2}_{0,1}\)}
      & \text{(d) unbiased padding \(\padded^{2}_{\tilde\rho}\)}
    \end{array}
  \]

  \caption{\centering Biased and unbiased paddings in dimension \(2\). The
    unlabelled arrows are identities.}
  \label{fig:padding-examples}
\end{figure}


We now define morphisms of filtrations. Those will allow us to transport paddings
over difference filtrations. Using such morphisms, we can transport the left-biased
and right-biased paddings to the filtration of the unbiased padding, and
subsequently to relate them in Sec.~\ref{subsec:repadding}. This relation is
analogue to the equation \(\rho_{\id} = \lambda_{\id}\) from monoidal categories.

\begin{definition}
  A \emph{morphism of filtrations} $\boldsymbol{\psi}=(\psi^i)_{i=m}^n$ between
  filtrations $(\Delta^i,w^i,\tau^{i})_{i=m}^n$ and
  $(\Gamma^i,v^i,\sigma^{i})_{i=m}^n$ consists of substitutions
  $\psi^i : \Delta^i \to \Gamma^i$ such that $\{w^i\}_{\psi^i} =\{v^i\}$, and
  the following commutes for each $i$:
\begin{equation}\label{paddingmorphism}
\begin{tikzcd}[ampersand replacement=\&]
        {\Delta^{i+1}} \& {\Delta^i \uparrow w^i} \\
        {\Gamma^{i+1}} \& {\Gamma^i \uparrow v^i}
        \arrow["{\tau^{i+1}}", from=1-1, to=1-2]
        \arrow["{\psi^{i+1}}"', from=1-1, to=2-1]
        \arrow["{\psi^i \uparrow w^i}", from=1-2, to=2-2]
        \arrow["{\sigma^{i+1}}"', from=2-1, to=2-2]
\end{tikzcd}\end{equation}
Given a family of types \(\mathbf{A}=(A^i)_{i=m}^n\), and padding data
\({\mathbf{p} = (p^i,q^i)_{i=m}^{n-1}}\) for \(\mathbf{A}\) with associated
padding \(\padded_{\mathbf{p}}\), we denote:
\begin{mathpar}
  \mathbf{A}\app{\boldsymbol{\psi}}
  := (A^i\app{\psi^{i}})_{i=m}^n \and
  \mathbf{p}\app{\boldsymbol{\psi}}
  := (p^i\app{\psi^{i+1}},q^i\app{\psi^{i+1}})_{i=m}^{n-1} \and
  \padded_{\mathbf{p}}\app{\boldsymbol{\psi}} :=
(\padded^{i}_{\mathbf{p}}\app{\psi^{i}})_{i=m}^n
\end{mathpar}
\end{definition}

\begin{lemma}
\label{lemma:padding-morphism}
Given \(\boldsymbol{\psi} : \mathbf{\Delta} \to \mathbf{\Gamma}\) a morphism of
filtrations, if $\mathbf{A}$ is a type family adapted to
$\mathbf{\Gamma}$, then $\mathbf{A}\app{\boldsymbol{\psi}}$ is adapted to
$\mathbf{\Delta}$. If $\mathbf{p}$ is padding data for $\mathbf{A}$ with
associated padding \(\padded_{\mathbf{p}}\), then
$\mathbf{p}\app{\boldsymbol{\psi}}$ is padding data for
$\mathbf{A}\app{\boldsymbol{\psi}}$, with associated padding 
\(\padded_{\mathbf{p}}\app{\boldsymbol{\psi}}\).
\end{lemma}
\begin{proof}
  Consider a type family $A^i = s^{i-1} \to t^{i-1}$ adapted to
  \(\mathbf{\Gamma}\). Because \(\boldsymbol{\psi}\) is a morphism of padding
  filtrations, \(w^{m} = v^{m}\app{\psi^{m}}\). Since
  $\Gamma^{m}\vdash v^m : A^m$ and we have:
  \[
    \Gamma^{m}\vdash w^m : A^m\app{\psi^m}
  \]
  In context \(\Gamma^{i+1}\), the terms \(s^{i}\) and \(t^{i}\) have type
  \(A^{i}\app{\sigma^{i+1}}\), so in context \(\Delta^{i+1}\), the types
  \(s^i\app{\psi^{i+1}}\) and\(t^i\app{\psi^{i+1}}\) have type
  \(A^i\app{\sigma^{i+1}\cir \psi^{i+1}}\). The following equality, proved by
  \eqref{paddingmorphism} then lets us conclude that
  \(\mathbf{A}\app{\boldsymbol{\psi}}\) is adapted to \(\mathbf{\Delta}\):
  \[
    A^i\app{\sigma^{i+1}\cir
    \psi^{i+1}} = A^i\app{(\psi^i\uparrow w^i)\cir \tau^{i+1}} =
  A^i\app{\psi^i}\app{ \tau^{i+1} }
  \]

  Let \(\mathbf{p}\) be padding data for \(\mathbf{A}\). In context
  \(\Gamma^{i+1}\), the term \(p^{i}\) has type:
  \[
    s^i \to \padded^i_{\mathbf{p}}\app{\inj^-\cir \sigma^{i+1}}
  \]
  Therefore in context \(\Delta^{i+1}\), the term \(p^{i}\app{\psi^{i+1}}\) has type:
  \[
    s^i\app{\psi^{i+1}} \to
    \padded^i_{\mathbf{p}}\app{\inj^-\cir \sigma^{i+1} \cir \psi^{i+1}}
  \]
  Furthermore, by equation~(\ref{paddingmorphism}), and Lemmas~\ref{includefunctorialised}
  and~\ref{disjoint-implies-inclusion-is-identity}, since
  \(w^i \notin \supp(\padded^i_\mathbf{p}\app{\psi^i})\), the target of this
  type satisfies:
  \begin{align*}
    \padded^i_{\mathbf{p}}\app{\inj^-\cir \sigma^{i+1} \cir \psi^{i+1}}
    &= \padded^i_{\mathbf{p}}\app{\inj^-\cir (\psi^i \uparrow w^i) \cir \tau^{i+1}}\\
    &= \padded^i_{\mathbf{p}}\app{\psi^i \cir \inj^-_{\Delta^{i}} \cir \tau^{i+1}}\\
    &= \padded^i_{\mathbf{p}}\app{\psi^i \cir \tau^{i+1}}
  \end{align*}
  Similarly, one can show that:
  \[
    \Delta^{i+1}\vdash q^{i}\app{\psi^{i+1}} :
    \padded^i_\mathbf{p}\app{\inj^+\cir \sigma^{i+1}} \to t^{i}\app{\psi^{i+1}}
  \]

  Finally, consider padding data \(\padded_{\mathbf{p}}\) associated to
  \(\mathbf{p}\). We show that $\Theta^i_{\mathbf{p}}\app{\psi^i}$ satisfies the
  defining formula (\ref{definingeq}), using \eqref{paddingmorphism} and
  Lemma~\ref{funcsubs}:
  \begin{align*}
    & \Theta^{i+1}_{\mathbf{p}}\app{\psi^{i+1}} \\
    & = (p^i \s_i (\padded^i_\mathbf{p}\uparrow v^i)\app{\sigma^{i+1}}\s_i q^i)\app{\psi^{i+1}}\\
    & = p^i\app{\psi^{i+1}} \s_i (\padded^i_\mathbf{p}\uparrow v^i)\app{\sigma^{i+1}\cir \psi^{i+1}}\s_i q^i\app{\psi^{i+1}}\\
    & = p^i\app{\psi^{i+1}} \s_i (\padded^i_\mathbf{p}\uparrow v^i)\app{(\psi^i \uparrow w^i)\cir \tau^{i+1}}\s_i q^i\app{\psi^{i+1}}\\
    &= p^i\app{\psi^{i+1}} \s_i (\padded^i_\mathbf{p}\app{\psi^i} \uparrow w^i)\app{\tau^{i+1}}\s_i q^i\app{\psi^{i+1}} \tag*{\qedhere}
  \end{align*}
\end{proof}

  We now proceed with the definition of the unbiased paddings of the identity,
  which are padding data for the same type as \(\mathbf{u}^{n}_{k,0}\) and
  \(\mathbf{u}^{n}_{k,n-1}\), but distinct from them. In
  Sec.~\ref{subsec:repadding}, we define the \emph{unbiasing repaddings} to
  relate these distinct paddings.

  \begin{definition}~\label{def:biased-pad-id}
    We define two morphisms of filtrations:
    \begin{align*}
      \boldsymbol{\psi}_{\rho} : \mathbf{\Gamma}^{n}_{k,0} & \to \mathbf{\Gamma}_{\rho}^{n}
      & \boldsymbol{\psi}_{\tilde\rho} : \mathbf{\Gamma}^{n}_{k,n-1} & \to \mathbf{\Gamma}_{\tilde\rho}^{n} \\
      v^{i}\app{\psi_{\rho}^{i}} &= v^{i}_{0}
      & d^{i}\app{\psi_{\tilde\rho}^{i}} &= v^{i}_{n-1}
    \end{align*}
    Applying these morphisms to the biased paddings, we obtain new padding data,
    that we call the \emph{biased paddings of the identity}
    \begin{align*}
      \mathbf{p}^{n}_{\rho}\app{\boldsymbol{\psi}_{\rho}}
      && \mathbf{p}^{n}_{\tilde\rho}\app{\boldsymbol{\psi}_{\tilde\rho}}
    \end{align*}
  \end{definition}

  We conclude this section with a presentation of suspension of paddings. While
  it does not appear in the construction of the steps presented in
  Figure~\ref{2dEH}, this notion still plays an important role for constructing
  the cells \(\eh^{n}_{k,l}\). It is central in Lemma~\ref{lemma:eh-susp},
  allowing us to leverage the suspension meta-operation to construct
  \(\eh^{n+1}_{k+1,l+1}\) from \(\eh^{n}_{k,l}\).
\begin{definition}
  Let \(\mathbf{\Gamma} = (\Gamma^i,v^i,\sigma^{i})_{i=m}^n\) be a filtration,
  and \(\mathbf{p}= (p^i,q^i)_{i=m}^n\) be padding data. We define:
  \begin{align*}
    \Sigma\mathbf{\Gamma}&:=(\Sigma\Gamma^{i-1},\Sigma
  v^{i-1},\Sigma\sigma^{i})_{i=m+1}^{n+1} \\
    \Sigma \mathbf{p} &:= (\Sigma p^{i-1},\Sigma q^{i-1})_{i=m+1}^{n+1}
  \end{align*}
\end{definition}

\begin{restatable}{lemma}{paddingsuspension}
  \label{lemma:padding-suspension}
  If $\mathbf{\Gamma}$ is a filtration, then so is $\Sigma\mathbf{\Gamma}$. If
  the type family $\mathbf{A}$ is adapted to $\mathbf{\Gamma}$, then
  $\Sigma \mathbf{A}$ is adapted to $\Sigma\mathbf{\Gamma}$, and if $\mathbf{p}$
  is padding data for $\mathbf{A}$, then $\Sigma \mathbf{p}$ is padding data for
  $\Sigma \mathbf{A}$, with associated padding $\Sigma\Theta_\mathbf{p}$:
\end{restatable}
\begin{proof}
  First we show that \(\Sigma\mathbf{\Gamma}\) is a filtration. The variable
  \(v^{i-1}\) is a maximal dimension variable in \(\Sigma\Gamma^{i-1}\) which is
  of dimension \(i\). By Lemma~\ref{lemma:suspfunc-context}, we have:
  \[
    \Sigma(\Gamma^{i-1} \uparrow v^{i-1}) = (\Sigma \Gamma^{i-1}) \uparrow v^{i-1}
  \]
  The substitution
  \(\Sigma\sigma^{i} : \Sigma \Gamma^i\to (\Sigma\Gamma^{i-1}) \uparrow
  v^{i-1}\) satisfies, by definition of the suspension of substitutions:
  \[
    \fun{v^{i-1}}\app{\Sigma\sigma^{i}} = v^{i}
  \]

  Let \(\mathbf{A} = (A^{i} = s^{i-1} \to t^{i-1})\) a type family adapted to
  \(\mathbf{\Gamma}\). We show that the type family \(\Sigma\mathbf{A}\) is
  adapted to \(\Sigma\mathbf{\Gamma}\). Since in context \(\Gamma^{m}\) the
  variable \(v^{m}\) has type \(A^{m}\), in context \(\Sigma\Gamma^{m}\), the
  same variable has type \(\Sigma A^{m}\). Moreover, in context
  \(\Gamma^{i+1}\), the terms \(s^{i}\) and \(t^{i}\) have type
  \(A^i\app{\sigma^{i+1}}\), so by Lemma~\ref{lemma:susp-results}, in context
  \(\Sigma\Gamma^{i+1}\), the terms \(\Sigma s^{i}\) and \(\Sigma t^{i}\) have
  type:
  \[
    (\Sigma A^i)\app{\Sigma \sigma^{i+1}}
  \]

  Finally, consider a padding \(\padded_{\mathbf{p}}\) associated to
  \(\mathbf{p}\), we show that \(\Sigma\padded_{\mathbf{p}}\) is a padding
  associated to \(\mathbf{p}\). First, we note that in context \(\Gamma^{i+1}\),
  the term \(p^{i}\) has type:
  \[
    s^i \to \Theta^i_\mathbf{p}\app{\inj^-\cir\sigma^{i+1}}
  \]
  Thus, by Lemmas~\ref{lemma:susp-results} and~\ref{lemma:suspfunc-context}, the term
  \(\Sigma p^{i}\) has the following type in context \(\Sigma\Gamma^{i+1}\):
  \[
    \Sigma s^i \to (\Sigma \Theta^i_\mathbf{p})\app{\inj^- \cir \Sigma
      \sigma^{i+1}}
  \]
  We now check that $\Sigma \Theta^i_\mathbf{p}$ satisfies (\ref{definingeq})
  for $\Sigma{\mathbf{p}}$. Using
  Lemmas~\ref{suspcomp},~\ref{lemma:susp-results} and~\ref{lemma:suspfunc-term+subs}, we have:
  \begin{align*}
    \Sigma\Theta^{i+1}_\mathbf{p} &=\Sigma(p^i \s_i (\Theta^i_\mathbf{p} \uparrow v^i)\app{\sigma^{i+1}} \s_i q^i)\\
    & =\Sigma p^i \s_{i+1} (\Sigma(\Theta^i_\mathbf{p} \uparrow v^i))\app{\Sigma \sigma^{i+1}} \s_{i+1} \Sigma q^i\\
    & =\Sigma p^i \s_{i+1} (\Sigma \Theta^i_\mathbf{p} \uparrow v^i)\app{\Sigma \sigma^{i+1}} \s_{i+1} \Sigma q^i \tag*{\qedhere}
  \end{align*}
\end{proof}


\subsection{Repadding}\label{subsec:repadding}
\noindent
We now introduce repadding, which allows us to change between two padding datas
for the same type family. This will constitute the heart of the construction of
$X_2$ in Figure~\ref{2dEH}.

\begin{definition}
  Consider a filtration \((\Gamma^i,v^i,\sigma^{i})_{i=m}^n\), a type family
  $\mathbf{A}$ adapted to it, and two sets of padding data for \(\mathbf{A}\):
\begin{align*}
  \mathbf{p}_0 &= (p_0^i,q_0^i)_{i=m}^{n-1}
  & \mathbf{p}_1 &= (p^i_1, q^i_1)_{i=m}^{n-1}
\end{align*}
We define sets of \emph{repadding data} $\mathbf{r} : \mathbf{p}_0\to\mathbf{p}_1$ and their associated
repadding \(\repad^{i}_{\mathbf{r}}\) together by mutual induction. A set of
repadding data {\({\mathbf{r}: \mathbf{p}_0\to\mathbf{p}_1}\) consists of families of terms \((f^i,g^i)_{i=m}^{n-1}\) of type:}
\begin{align*}
  \Gamma^{i+1} &\vdash f^i
  : p^i_0 \s_i \repad^i_{\mathbf{r}}\app{\inj^- \cir \sigma^{i+1}} \to p^i_1 \\
  \Gamma^{i+1} &\vdash g^i
  : q^i_0 \to \repad^i_{\mathbf{r}}\app{\inj^+ \cir \sigma^{i+1}}\s_i q^i_1
\end{align*}
Its associated repadding, the term,
\(\Gamma^{i} \vdash \repad^i_{\mathbf{r}} : \padded^i_{\mathbf{p}_0} \to
\padded^i_{\mathbf{p}_1}\) is defined by:
\begin{align*}
  \repad^m_{\mathbf{r}} &:= \id_{v^m} \\
  \repad^{i+1}_{\mathbf{r}} &:= \hexcomp\app{f^i, ((\repad^i_{\mathbf{r}} \uparrow v^i)\app{\sigma^{i+1}})^{-1}, g^i}
\end{align*}
\end{definition}

\begin{figure}

  \centering
  \[
    \begin{tikzcd}[ampersand replacement=\&,column sep=5em]
      \& {\bdot} \& {\bdot} \\
      {s^i} \&\&\& {t^i} \\
      \& {\bdot} \& {\bdot}
      \arrow["{(\Theta^i_{\mathbf{p}_0} \uparrow v^i)\app{\sigma^{i+1}}}"
            , from=1-2, to=1-3]
      \arrow[from=1-2, to=3-2]
      \arrow[""{name=0, anchor=center, inner sep=0}, "{q^i_0}", from=1-3, to=2-4]
      \arrow["{\scriptstyle (\repad^i_\mathbf{r} \uparrow
        v^i)\app{\sigma^{i+1}}^{\text{-}1}}"{description}
      , shorten >=6pt, shorten <=6pt, Rightarrow, from=1-3, to=3-2]
      \arrow[from=1-3, to=3-3]
      \arrow["{p^i_0}", from=2-1, to=1-2]
      \arrow[""{name=1, anchor=center, inner sep=0}, "{p^i_1}"', from=2-1, to=3-2]
      \arrow["{(\Theta^i_{\mathbf{p}_1} \uparrow v^i)\app{\sigma^{i+1}}}"'
      , from=3-2, to=3-3]
      \arrow["{q^i_1}"', from=3-3, to=2-4]
      \arrow["{f^i}"{description}, shorten >=6pt, shorten <=6pt, Rightarrow, from=1-2, to=1]
      \arrow["{g^i}"{description}, shorten >=6pt, shorten <=6pt, Rightarrow, from=0, to=3-3]
    \end{tikzcd}
  \]
  \caption{The repadding cell \(\repad^{i+1}_{\mathbf{r}}\).}
  \label{fig:repadding}
\end{figure}

\noindent
The definition of \(\repad^{i+1}_{\mathbf{r}}\), illustrated in
Figure~\ref{fig:repadding}, uses an inverse, which exists by
Lemma~\ref{diminverse}. We now introduce the \emph{unbiasing repaddings} as the
crucial ingredients of cell $X_2$ of Figure~\ref{2dEH}, allowing us to change a
biased padding applied to the identity into an unbiased padding.


\begin{definition}[Unbiasing repaddings]\label{def:biased-repadding}
  Given \(n\in \N\), we recall the biased paddings of the identity
  \(\mathbf{p}^{n}_{\rho}\app{\boldsymbol{\psi}_{\rho}}\) from
  Def.~\ref{def:biased-pad-id} and the unbiased paddings
  \(\mathbf{u}^{n}_{n-1,0}\) from Def.~\ref{def:unbiasedpadding}. We define a
  set of repadding data
  \({\mathbf{r}^{n}_{\rho \to u} = (f^{i}_{\rho\to u},g^{i}_{\rho\to u})_{i=1}^{n-1}}\) from
  \(\mathbf{p}^{n}_{\rho}\app{\boldsymbol{\psi}_{\rho}}\) to
  \(\mathbf{u}^{n}_{n-1,0}\) and its associated repadding, {which we write just as \(\repad^{i}_{\rho\to u}\)},
  as follows:
  \begin{align*}
    f_{\rho\to u}^{i}
    &:= \coh(\point : p^i_{\rho}\app{\psi_{\rho}^{i}} \s_i \repad^i_{\rho\to u}\app{\inj^- \cir
      \sigma^{i+1}}\to p^i_{n-1,0})[x]\\
    g_{\rho\to u}^{i}
    &:= \coh(\point : {q}^i_{\rho}\app{\psi_{\rho}^{i}}\to\repad^i_{\rho\to u}\app{\inj^+ \cir \sigma^{i+1}} \s_i q^i_{n-1,0})[x]
  \end{align*}
  Similarly, we define repadding data
  \(\mathbf{r}_{\tilde\rho \to u}^{n} = (f^{i}_{\tilde\rho\to u},g^{i}_{\tilde\rho\to u})_{i=1}^{n-1}\) from
   \(\mathbf{p}^{n}_{\tilde\rho}\app{\boldsymbol{\psi}_{\tilde\rho}}\) to \(\mathbf{u}^{n}_{0,n-1}\),
  and its associated repadding \(\repad^{i}_{\tilde\rho\to u}\), as follows:
  \begin{align*}
    f_{\tilde\rho \to u}^i &:= \coh(\point : p^i_{\tilde\rho}\app{\psi_{\tilde\rho}^{i}} \s_i \repad^i_{\tilde\rho \to u}\app{\inj^- \cir \sigma^{i+1}} \to p^i_{0,n-1})[x]\\
    g_{\tilde\rho \to u}^i &:= \coh(\point : q^i_{\tilde\rho}\app{\psi_{\tilde\rho}^{i}} \to \repad^i_{\tilde\rho \to u}\app{\inj^+ \cir \sigma^{i+1}} \s_i q^i_{0,n-1})[x]
  \end{align*}
  We call the associated repaddings the \emph{right-unbiasing repaddings}.
  They provide the following equivalences, which are needed for $X_2$ in Figure~\ref{2dEH}:
\begin{align*}
  \Gamma^n_0 &\vdash \repad^n_{\rho \to u} : \padded^n_{\rho}\app{v^n_0} \to \padded^n_{n-1,0} \\
  \Gamma^n_{n-1} &\vdash \repad^n_{\tilde\rho \to u} : \padded^n_{\tilde\rho}\app{v^n_{n-1}} \to \padded^n_{0,n-1}
\end{align*}
        We also define the \emph{left-unbiasing repadding}:
\begin{align*}
  \repad_{\lambda \to u} &:= (\repad_{\rho \to u})^{\op}
  & \repad_{\tilde\lambda \to u} &:=(\repad_{\tilde\rho \to u})^{\op}
\end{align*}
By Proposition~\ref{lemma:symmetry-main}, these terms terms satisfy:
\begin{align*}
  \Gamma^n_0 &\vdash \repad^n_{\lambda \to u} : \padded^n_{\lambda}\app{v^n_0} \to \padded^n_{n-1,0} \\
  \Gamma^n_{n-1} &\vdash \repad^n_{\tilde\lambda \to u} : \padded^n_{\tilde\lambda}\app{v^n_{n-1}} \to \padded^n_{0,n-1}
\end{align*}
\end{definition}

\subsection{Pseudo-Functoriality of the Unbiased Padding}\label{subsec:pseudofunctoriality}
\noindent
The key idea for $X_3$ in Figure~\ref{2dEH} is to construct a witness that relates the unbiased padding of a composite, with the composite of unbiased paddings. We think of this as a pseudo-functoriality property, since it is exactly one of the pieces of data for pseudo-functoriality of a 2\=/functor.

\begin{proposition}\label{construction:Xi}
  In the context $(\Gamma^n_l, w : I^{n-1}_k)$ we can construct a cell:
    \[
      \Xi^n_{k,l} : \padded^n_{k,l}\app{v^n_l} \s_{n-1} \padded^n_{k,l}\app{w} \to
      \padded^n_{k,l}\app{v^n_l \s_{n-1} w}
    \]
  \end{proposition}

\begin{proof}
    Recall the filtration
    \(\mathbf{\Gamma^{n}_{k,l}} =
    (\Gamma^{i}_{l},v^{i}_{l},\sigma^{i})_{i=m}^{n}\) for the unbiased padding,
    defined in Def~\ref{def:unbiasedpadding}. Given a term \(\Gamma^{i}_{l}\vdash t:A\), we introduce the notation:
    \begin{align*}
      t \uparrow^0 v^i_{l} &:= t \\
      t \uparrow^{k+1} v^i_{l} &:= ((t \uparrow^k v^i_{l}) \uparrow v^{i+k}_{l})\app{\sigma^{i+k+1}}
    \end{align*} We construct by induction on
    $i \geq m$ terms \(\Xi^{i\uparrow n-i}_{k,l}\) and then specialise this
    construction to define \(\Xi^{n}_{k,l}:=\Xi^{n\uparrow 0}_{k,l}\). The terms
    \(\Xi^{i\uparrow n-i}_{k,l}\) that we construct will be valid in the context
    \(\Gamma^{n}_{l}\), and will have source:
    \[
      (\padded^i_{k,l}\uparrow^{n-i} v^i_l)\app{v^n_l}
      \s_{n-1}(\padded^i_{k,l}\uparrow^{n-i} v^i_l)\app{w}
    \]
    and target:
      \[
      (\padded^i_{k,l}\uparrow^{n-i} v^i_l)\app{v^n_l \s_{n-1} w}
    \]
We first claim that, for any $m < i\leq n$ and $0 \leq j \leq n-i$, the unbiased padding satisfies the following equation:
  \[
    \padded_{k,l}^i \uparrow^{j} v^i = p^{i-1}_{k,l} \s_{i-1} (\Theta_{k,l}^{i-1}
    \uparrow^{j+1} v^{i-1}_l) \s_{i-1} q^{i-1}_{k,l}
  \]

To see this, we will proceed by induction on $j$. When $j=0$, using the defining formula
  \eqref{definingeq}, we have:
  \begin{align*}
    \Theta^i_{k,l}\uparrow^0 v^n
    &= \Theta^i_{k,l} \\
    &= p^{i-1}_{k,l} \s_{i-1} (\padded^{i-1}_{k,l}\uparrow
      v^{i-1}_l)\app{\sigma^i} \s_{i-1} q^{i-1}_{k,l} \\
    &= p^{i-1}_{k,l} \s_{i-1} (\padded^{i-1}_{k,l}\uparrow^1 v^{i-1}_l)
      \s_{i-1} q^{i-1}_{k,l}
  \end{align*}

  When $j > 0$, we have, by induction and using that
  $v^{i+j-1}\notin\supp(p^{i-1}_{k,l})$ or $\supp(q^{i-1}_{k,l})$, so these are
  fixed by $\sigma^{i+j}$:
  \begin{align*}
    &\Theta^i_{k,l}\uparrow^j v^i \\
    &= ((\Theta^i_{k,l} \uparrow^{j-1} v^i) \uparrow
      v^{i+j-1})\app{\sigma^{i+j}} \\
    &= ((p^{i-1}_{k,l}
      \s_{i-1}(\padded^{i-1}_{k,l}\uparrow^{j}v^{i-1}_l)\s_{i-1}q^{i-1}_{k,l})
      \uparrow v^{i+j-1})\app{\sigma^{i+j}} \\
    &= p^{i-1}_{k,l} \s_{i-1}((\padded^{i-1}_{k,l}\uparrow^{j}v^{i-1}_l)\uparrow
      v^{i+j-1})\app{\sigma^{i+j}}\s_{i-1}q^{i-1}_{k,l} \\
    & = p^{i-1}_{k,l} \s_{i-1}(\padded^{i-1}_{k,l}\uparrow^{j+1}v^{i-1}_l)
      \s_{i-1}q^{i-1}_{k,l}
  \end{align*}
This proves the claim.

Returning to the construction of $\Xi_{k,l}^{i\uparrow n-i}$, when $i=m$, we define
    \(\Xi^{m\uparrow (n-m)}_{k,l} := \id(v^n_l \s_{n-1} w)\) When $m<i<n$, we
    have:
    \[
      \padded_{k,l}^i \uparrow^{n-i} v^i = p^{i-1}_{k,l} \s_{i-1} (\Theta_{k,l}^{i-1}
      \uparrow^{n-i+1} v^{i-1}_l) \s_{i-1} q^{i-1}_{k,l}
    \]
    by the claim above.
    This allows us define \(\Xi^{i\uparrow n-i}_{k,l}\) as a composite of an
    interchanger with an application of the whiskering of
    \(\Xi^{i-1\uparrow n-i+1}_{k,l}\). 
    Formally, to construct this interchanger we consider the pasting context $\mathbb{X}^{N}$ for \(N\geq 2\), given by:
    \[
    \left(
      \begin{array}{ll}
        d^0_L,d^0_- : \obj,\
        d^1_L : d^0_L \to d^0_-, \
        d^0_+ : \obj \\
        d^1_-, d^1_+ : d^{0}_- \to d^{0}_+,\ \ldots,\ d^{N-2}_-, d^{\label{fig:ps-interchanger-chi}N-2}_+ : d^{N-3}_- \to d^{N-3}_+,  \\
        d^{N-1}_-, d^{N-1}_0 : d^{N-2}_- \to d^{N-2}_+ ,\ d^N_T : d^{N-1}_- \to d^{N-1}_0 \\
        d^{N-1}_+ : d^{N-2}_0 \to d^{N-2}_+ ,\ d^{N}_B : d^{N-1}_0 \to d^{N-1}_+, \\
        d^0_R : \obj, d^1_R : d^0_+ \to d^0_R
      \end{array}
    \right)
  \]
    The
  pasting contexts \(\mathbb{X}^{2}\) and \(\mathbb{X}^{3}\) are illustrated in
  Figure~\ref{fig:ps-interchanger-xi}.
  
  In this context, given a term \(\mathbb{X}^{N}\vdash t:A\) whose
  \(0\)\=/dimensional source is \(d^{0}_{-}\) and whose \(0\)\=/target is
  \(d^{0}_{+}\), we write \(w(t)\) for the whiskering \(d^{1}_{L} \s_{0} t \s_{0} d^{1}_{R}\). We then define:
  \begin{align*}
    \chi_N := \coh(\mathbb{X}^{N} : w(d^{N}_{T}) \s_{N-1} w(d^{N}_{B}) \to w(d^{N}_{T}\s_{N-1}d^{1}_{R}))[\id_{\mathbb{X}^N}]
  \end{align*}
    
    Equipped with the interchangers $\chi_N$, we define $\Xi^{i\uparrow n-i}_{k,l}$ as the composite summarised by the diagram below, where \(t = \Theta_{k,l}^{i-1} \uparrow^{n-i+1} v^{i-1}_l\):
    \[
      \begin{tikzcd}[row sep=12pt]
        \left( p^{i-1}_{k,l} \!\s_{i-1}\! t \app{v^n_l} \!\s_{i-1}\!
          q^{i-1}_{k,l}\right) \!\s_{n-1}\! \left(p^{i-1}_{k,l} \!\s_{i-1}\!
          t\app{w} \!\s_{i-1}\! q^{i-1}_{k,l}\right)
        \ar[d, "\scriptstyle{\Sigma^{i-1}\chi_{n-i+1}}"]\\
        p^{i-1}_{k,l} \s_{i-1} \left( t\app{v^n_l}\s_{n-1} t\app{w} \right) \s_{i-1}
        q^{i-1}_{k,l}
        \ar[d,"\scriptstyle{\text{whiskering of $\Xi^{i-1\uparrow{n-i+1}}_{k,l}$}}"]\\
        p^{i-1}_{k,l} \s_{i-1} t\app{v^n_l \s_{n-1} w} \s_{i-1} q^{i-1}_{k,l}
      \end{tikzcd}
    \]
    Finally, for $i=n$ we define $\Xi^{n\uparrow 0}_{k,l}$ as the composite of two
    steps presented below, with \(t = \Theta_{k,l}^{n-1}\uparrow v^{n-1}_l\):
    \[
      \begin{tikzcd}[row sep=12pt]
        \left(p^{n-1}_{k,l} \!\s_{n\text{-}1}\!  t \app{v^n_l} \!\s_{n\text{-}1}\! q^{n-1}_{k,l}\right) \!\s_{n\text{-}1}\!
        \left(p^{n-1}_{k,l} \!\s_{n\text{-}1}\! t\app{w} \!\s_{n \text{-} 1}\! q^{n-1}_{k,l}\right)
        \ar[d,"\substack{\scriptstyle{\text{associators +}} \\
          \scriptstyle{\text{cancellator for the inverses
              $q^{n-1}_{k,l}$ and $p^{n-1}_{k,l}$}}}"]\\
        p^{n-1}_{k,l} \s_{n\text{-}1} \left( t\app{v^n_l} \s_{n\text{-}1} t \app{w} \right) \s_{n\text{-}1}
        q^{n-1}_{k,l}
        \ar[d,"\scriptstyle{\text{whiskering of $\Xi^{n-1\uparrow 1}_{k,l}$}}"]\\
        p^{n-1}_{k,l} \s_{n\text{-}1} t \app{v^n_l \s_{n\text{-}1} w} \s_{n\text{-}1} q^{n-1}_{k,l}
      \end{tikzcd}
    \]
\end{proof}

\begin{figure}
  \centering
  \[\begin{tikzcd}[column sep = 1cm]
        d^0_L & d^0_- & d^0_+ & d^0_R
        \arrow["{d^1_L}", from=1-1, to=1-2]
        \arrow[""{name=0, anchor=center, inner sep=0}, "{d^1_-}", bend left = 60, shift left = 5pt, from=1-2, to=1-3]
        \arrow[""{name=1, anchor=center, inner sep=0}, "{d^1_+}"', bend right = 60, shift right = 5pt, from=1-2, to=1-3]
        \arrow[""{name=2, anchor=center, inner sep=0}, "{d^1_0}"{description}, from=1-2, to=1-3]
        \arrow["{d^1_R}", from=1-3, to=1-4]
        \arrow["{d^2_T}", shorten <=3pt, shorten >=3pt, Rightarrow, from=0, to=2]
        \arrow["{d^2_B}", shorten <=3pt, shorten >=3pt, Rightarrow, from=2, to=1]
      \end{tikzcd}\]

    \[\begin{tikzcd}[column sep = 1cm]
        d^0_L & d^0_- &[1cm] d^0_+ & d^0_R
        \arrow["{d^1_L}", from=1-1, to=1-2]
        \arrow[""{name=0, anchor=center, inner sep=0}, "{d^1_-}", bend left = 80, shift left = 5pt, from=1-2, to=1-3]
        \arrow[""{name=1, anchor=center, inner sep=0}, "{d^1_+}"', bend right = 80, shift right = 5pt, from=1-2, to=1-3]
        \arrow["{d^1_R}", from=1-3, to=1-4]
        \arrow[""{name=2, anchor=center, inner sep=0}, "{d^2_-}"'{pos=0.4}, bend right = 70, shift right = 2pt, shorten <=7pt, shorten >=7pt, Rightarrow, from=0, to=1]
        \arrow[""{name=3, anchor=center, inner sep=0}, "{d^2_+}"{pos=0.4}, bend left = 70, shift left = 2pt, shorten <=7pt, shorten >=7pt, Rightarrow, from=0, to=1]
        \arrow[""{name=4, anchor=center, inner sep=0}, "{d^2_0}"{description}, shorten <=6pt, shorten >=6pt, Rightarrow, from=0, to=1]
        \arrow["{d^3_T}", shorten <=2pt, shorten >=4pt, Rightarrow, scaling nfold=3, from=2, to=4]
        \arrow["{d^3_B}", shorten <=4pt, shorten >=2pt, Rightarrow, scaling nfold=3, from=4, to=3]
\end{tikzcd}\]

  \caption{The pasting contexts $\mathbb{X}^2$ and $\mathbb{X}^3$}
    \label{fig:ps-interchanger-xi}
\end{figure}

Note that in the construction above, only the case $\Xi^{n\uparrow 0}_{k,l}$ used the specific padding data $\mathbf{u}^n_{k,l}$, where the witnesses of invertibility were used. The constructions of $\Xi^{i\uparrow n-1}_{k,}$ for $i<n$, however, work, for arbitrary padding data. We record this in the following slight generalisation:

\begin{lemma}\label{lemma:xi-generalised}
  Let $\mathbf{\Gamma} = (\Gamma^i,v^i,\sigma^i)_{i=m}^n$ be a filtration,
  $\mathbf{A}$ be a type adapted to \(\mathbf{\Gamma}\) and \(\mathbf{p}\)
  padding data for \(\mathbf{A}\). Suppose given a context \(\Delta\) together with
  terms and substitutions
  \begin{mathpar}
    \Delta \vdash v : s \to_B t
    \and \Delta \vdash w : t \to_B u \\
    \Delta  \vdash  \sigma_v : \Gamma\uparrow v^{n}
    \and \Delta \vdash  \sigma_w : \Gamma\uparrow v^{n}
    \and \Delta \vdash \sigma_{v \s w} : \Gamma^n \uparrow v^n
  \end{mathpar}
  such that
  \begin{align*}
    \fun{v^n}\app{\sigma_v} &= v
    & \fun{v^n}\app{\sigma_w} &= w
    & \fun{v^n}\app{\sigma_{v \s w}} &= v\s_n w
  \end{align*}
  and for every over variable \(x\),
  \[
    x\app{\sigma_{v}} = x\app{\sigma_{w}} = x\app{\sigma_{v\s w}}.
  \]
  Then, for any $0\leq k \leq n-m$, there exists a term
  $\Xi_{\mathbf{p}}^{n-k\uparrow k+1}$ which is derivable in context \(\Delta\)
  with type:
  \begin{align*}
    ((\padded^{n-k} \uparrow^k v^{n-k}) \uparrow v^n)\app{\sigma_v} &\s_n
    ((\padded^{n-k} \uparrow^k v^{n-k}) \uparrow v^n)\app{\sigma_w} \\
    \to & ((\padded^{n-k} \uparrow^k v^{n-k}) \uparrow v^n)\app{\sigma_{v*_n w}}
  \end{align*}
\end{lemma}
\begin{proof}
  The proof is exactly similar to that of Proposition~\ref{construction:Xi}.
\end{proof}

\subsection{Iterated Padding}\label{subsec:iterated}
\noindent
We conclude our section on padding with additional results regarding relating
the nested padding of a padding with a padding performed in a single step. This is used in
Corollary~\ref{corollary:ehpkl} to construct generalisations of the cells \(\eh^{n}_{k,l}\).

\begin{definition}
  Let \(\mathbf{\Gamma} = (\Gamma^{i},v^{i},\sigma^{i})_{i=m}^{n}\) be a
  filtration and \(B^{i}\) a type family adapted to it. Denote \(A^{i}\) the
  type family defined by \(\Gamma^{i} \vdash v^{i} : A^{i}\). Let
  $\Gamma \setminus v^i$ be the context obtained by removing $v^i$. Since $v^i$
  is locally maximal, this context is well-formed. Moreover, due to the
  dimension of $B^i$, we have $\Gamma^i \setminus v^i \vdash B^i$. We define a
  family
  \(\mathbf{\Gamma}_{/B} := (\Gamma^{i}_{/B},w^{i},\sigma^{i}_{/B})^{n}_{i=m}\)
  as follows:
\begin{align*}
  \Gamma^{i}_{/B}
  &:= (\Gamma^{i}\setminus v^i ,w^{i} : B^{i}) \\
  \Gamma^{i}_{/B}
  &\vdash \sigma^{i}_{/B} : \Gamma^{i-1}_{/B}
    \uparrow w^{i} \\
  x\app{\sigma^{i}_{/B}}
  &:=
    \begin{cases}
      w^{i} & \text{if } x = \fun{w^{i-1}}\\
      x\app{\sigma^{i}}& \text{if }x\in \Gamma^i \setminus v^i
    \end{cases}
\end{align*}
\end{definition}

\begin{restatable}{proposition}{iteratedpadding}\label{prop:iteratedpadding}
  For any filtration \(\mathbf{\Gamma}\) and type family \(B\) adapted to it,
  the family \(\mathbf{\Gamma}_{/B}\) is a filtration. Moreover, a type family
  \(C\) is adapted to \(\mathbf{\Gamma}\) if and only if it is adapted to
  \(\mathbf{\Gamma}_{/B}\).
\end{restatable}
\begin{proof}
  We first show that $\sigma^i_{/B}$ is well-typed. Denoting
  $B^i = s^{i-1} \to_{B^{i-1}[\sigma^i]} t^{i-1}$, we necessarily have the following, showing that $\sigma^i_{/B}$ is well-typed, and thus that
  $\mathbf{\Gamma}_{/B}$ is a filtration:
  \begin{align*}
    (w^{i-1})^- \app{\sigma^{i}} &= s^{i-1}
    & (w^{i-1})^+\app{\sigma^{i}}&= t^{i-1}
  \end{align*}
  For the second part, note that any type family $C = (C^i)_{i=m}^n$ adapted to
  either $\mathbf{\Gamma}$ or $\mathbf{\Gamma}_{/B}$ must satisfy:
  \[
    \Delta^i \vdash C^i
  \]
  Since $A^m=B^m$, and since \(\sigma^{i}\) and \(\sigma^{i}_{/B}\) coincide on
  \(\Delta^{i-1}\), $C$ is adapted to $\mathbf{\Gamma}$ if and only if it is
  adapted to $\mathbf{\Gamma}_{/B}$.
\end{proof}

\begin{restatable}{proposition}{paddingcompose}\label{prop:padding-compose}
  Given a filtration \(\mathbf{\Gamma}\) with two types families \(B\) and \(C\)
  adapted to it. Suppose we have padding data \(\mathbf{p} = (p^i_-,p^i_+)_{i=m}^{n-1}\) for \(B\) adapted to
  \(\mathbf{\Gamma}\) and padding data \(\mathbf{q}=(q^i_-,q^i_+)_{i=m}^{n-1}\) for \(C\) adapted to
  \(\mathbf{\Gamma}_{/B}\). Then there exists padding data
  \(\mathbf{q}\square\mathbf{p} = (q^i_-\boxminus p^i_-, p^i_+ \boxplus q^i_+)_{i=m}^{n-1}\) for \(C\) adapted to \(\mathbf{\Gamma}\) and
  equivalences:
  \[
    \Gamma^{i}\vdash \mu^{i}_{\mathbf{q},\mathbf{p}} :
    \padded_{\mathbf{q}}^{i}\app{\padded_{\mathbf{p}}^{i}} \to \padded^{i}_{\mathbf{q}\square\mathbf{p}}
  \]
\end{restatable}
\begin{proof}
  Throughout this proof, we suppose that the filtration is given by
  \(\mathbf{\Gamma} = (\Gamma^{i},v^{i},\sigma^{i})\) and we write:
  \begin{align*}
    A^i &= a^{i-1}_- \to a^{i-1}_+ &  B^i &= b^{i-1}_- \to b^{i-1}_+ &  C^i &= c^{i-1}_- \to c^{i-1}_+
  \end{align*}
  We write $w^i$ for the chosen variable of $\Gamma^i_{/B}$ in the filtration $\mathbf{\Gamma}_{/B}$. We construct my mutual the following:
  \begin{itemize}
  \item Padding data
    $\mathbf{q}\square\mathbf{p} = (q^i_-\boxminus p^i_-,p^i_+\boxplus
    q^i_+)_{i=m}^{n-1}$ for the type $C$ adapted to the filtration
    $\mathbf{\Gamma}$.
  \item Equivalences
    $\Gamma^i : \mu^i_{\mathbf{q},\mathbf{p}} : \padded^i_{q}\app{\padded^i_{p}}
    \to \padded^i_{\mathbf{q}\square\mathbf{p}}$.
  \end{itemize}
  We first define \(q^i_-\boxminus p^i_-\) and \(p^i_+\boxplus q^i_+\)
  in terms of \(\mu^{i}_{\mathbf{p},\mathbf{q}}\):
  \begin{align*}
    q^{i}_{-}\boxminus p^i- &:= q^i_- \s_i (\padded^i_{\mathbf{q}} \uparrow w^i)\app{p^i_-} \s_i \mu^i_{\mathbf{q},\mathbf{p}}\app{\inj^-\cir\sigma^{i+1}}\\
    p^{i}_{+}\boxplus q^i_+ &:= (\mu^i_{\mathbf{q},\mathbf{p}}\app{\inj^+\cir\sigma^{i+1}})^{-1} \s_i (\padded^i_\mathbf{q}\uparrow w^i)\app{p^i_+} \s_i q^i_+
  \end{align*}
  Thus we are left with defining the equivalences
  \(\mu^{i}_{\mathbf{p},\mathbf{q}}\). When $i=m$, it suffices to chose $\mu^m_{\mathbf{q},\mathbf{p}}:=\id_{v^{m}}$.
  Now, supposing that \(\mu^{i}_{\mathbf{q},\mathbf{p}}\) has been constructed,
  we construct \(\mu^{i+1}_{\mathbf{q},\mathbf{p}}\) in three main steps. The
  source of \(\mu^{i+1}_{\mathbf{q},\mathbf{p}}\) is the following, writing $*$ for $*_i$:
  \[
    \padded^{i+1}_\mathbf{q}\app{\padded^{i+1}_\mathbf{p}} = q^i_-\! \s
    (\padded_{\mathbf{q}} \uparrow w^{i})\app{p^i_-\!\! \s\! (\padded_\mathbf{p}
      \uparrow v^{i}) \app{\sigma^{i+1}} \!\s p^i_+} \s q^i_+
  \]
  Our first step consist of a ternary variation of the pseudo-functoriality
  witness \(\Xi^{i\uparrow 1}_{\mathbf{q}}\) from
  Lemma~\ref{lemma:xi-generalised}. We can define this ternary
  variation \(X\) using associators and the binary one as follows:
  \[
    \begin{tikzcd}
      (\padded_{\mathbf{q}} \uparrow w^{i})\app{p^i_- \s (\padded_\mathbf{p}
        \uparrow v^{i}) \app{\sigma^{i+1}} \s p^i_+}
      \ar[d,"\text{\small (associator)}"]\\
      (\padded_{\mathbf{q}} \uparrow w^{i})\app{(p^i_- \s (\padded_\mathbf{p}
        \uparrow v^{i}) \app{\sigma^{i+1}}) \s p^i_+} \ar[d,"\Xi^{i\uparrow 1}_{\mathbf{q}}"]\\
      (\padded_{\mathbf{q}} \uparrow w^{i}) \app{p^i_- \s (\padded_\mathbf{p}
        \uparrow v^{i}) \app{\sigma^{i+1}}} \s (\padded_{\mathbf{q}} \uparrow
      w^{i})\app{p^i_+}
      \ar[d, "\Xi^{i\uparrow 1}_{\mathbf{q}}"]\\
      ((\padded_{\mathbf{q}}\! \uparrow\! w^{i})\app{p^i_-}\!\s\! (\padded_{\mathbf{q}}
      \!\uparrow\! w^{i})\app{ (\padded_{\mathbf{p}} \!\uparrow\! v^{i})
        \app{\sigma^{i+1}}})\! \s\! (\padded_{\mathbf{q}} \!\uparrow\! w^{i})\app{p^i_+}
      \ar[d,"\text{\small (associator)}"]\\
      (\padded_{\mathbf{q}} \uparrow w^{i})\app{p^i_-}\!\s\! (\padded_{\mathbf{q}}
      \!\uparrow\! w^{i})\app{ (\padded_\mathbf{p} \!\uparrow\! v^{i})
        \app{\sigma^{i+1}}}\! \s\! (\padded_{\mathbf{q}} \!\uparrow\! w^{i})\app{p^i_+}
    \end{tikzcd}
  \]
  By Lemma~\ref{funcsubs}, the target of this cell is equal to:
  \[
    (\padded_{\mathbf{q}} \uparrow w^{i})\app{p^i_-}\s
    (\padded_{\mathbf{q}}\app{\padded_{\mathbf{p}}} \uparrow w^{i})
      \app{\sigma^{i+1}} \s (\padded_{\mathbf{q}} \uparrow w^{i})\app{p^i_+}
  \]
  To proceed further, we use the term
  \(\mu^{i}_{\mathbf{q},\mathbf{p}} \uparrow v^{i}\), which in context
  \(\Gamma^{i}\uparrow v^{i}\) has type:
  \[
    \mu^{i}_{\mathbf{q},\mathbf{p}}\app{\inj^-} \s_{i} (\padded^{i}_{\mathbf{q}\square\mathbf{p}}\uparrow v^{i}) \to
    (\padded^{i}_{\mathbf{q}}\app{\padded^{i}_{\mathbf{p}}}\uparrow v^{i}) \s_{i} \mu^{i}_{\mathbf{q},\mathbf{p}}\app{\inj^+}
  \]
  We now construct a new cell denoted
  \(Y\), defined as a composite of \(5\) steps:
  \[
    \begin{tikzcd}
      \padded^{i}_{\mathbf{q}}\app{\padded^{i}_{\mathbf{p}} }\uparrow v^{i}
      \ar[d,"\text{\small (unitor)}"]\\
      (\padded^{i}_{\mathbf{q}}\app{\padded^{i}_{\mathbf{p}}}\uparrow v^{i})
      \s_{i}
      \id(\padded^{i}_{\mathbf{q}}\app{\padded^{i}_{\mathbf{p}}}\app{\inj^+})
      \ar[d,"\text{\small ((whiskering of
        $\scriptstyle{\varepsilon_{\mu^i_{\mathbf{q},\mathbf{p}}\app{
              \inj^+}}^{-1}}$))}"]
      \\
      (\padded^{i}_{\mathbf{q}}\app{\padded^{i}_{\mathbf{p}}}\uparrow v^{i})
      \s_{i}
      (\mu^i_{\mathbf{q},\mathbf{p}}\app{\inj^+  } \s_i
      (\mu^i_{\mathbf{q},\mathbf{p}}\app{\inj^+})^{-1})
      \ar[d,"\text{\small (associator)}"]\\
      ((\padded^{i}_{\mathbf{q}}\app{\padded^{i}_{\mathbf{p}}}\uparrow v^{i})
      \s_{i} \mu^i_{\mathbf{q},\mathbf{p}}\app{\inj^+}) \s_i
      (\mu^i_{\mathbf{q},\mathbf{p}}\app{\inj^+})^{-1}
      \ar[d,"\text{\small (whiskering of
        $\scriptstyle{(\mu^i_{\mathbf{q},\mathbf{p}} \uparrow v^i)^{-1}}$)}"]
      \\
       (\mu^{i}_{\mathbf{q},\mathbf{p}}\app{\inj^-} \s_{i}
       (\padded^{i}_{\mathbf{q}\square\mathbf{p}}\uparrow  v^{i})) \s_{i}
       (\mu^{i}_{\mathbf{q},\mathbf{p}}\app{\inj^+})^{-1}
       \ar[d,"\text{\small (associator)}"]\\
       \mu^{i}_{\mathbf{q},\mathbf{p}}\app{\inj^-} \s_{i}
       (\padded^{i}_{\mathbf{q}\square\mathbf{p}}\uparrow v^{i}) \s_{i}
       (\mu^{i}_{\mathbf{q},\mathbf{p}}\app{\inj^+})^{-1}
    \end{tikzcd}
  \]
  By Lemma~\ref{invsubs}, the target of \(Y\) rewrites as:
  \[
    \mu^{i}_{\mathbf{q},\mathbf{p}}\app{\inj^-\cir \sigma^{i+1}} \s_{i}
    (\padded^{i}_{\mathbf{q}\square\mathbf{p}}\uparrow v^{i}) \s_{i}
    (\mu^{i}_{\mathbf{q},\mathbf{p}}\app{\inj^+\cir\sigma^{i+1}})^{-1}
  \]
  This allows us to define the cell \(\mu^{i+1}_{\mathbf{q},\mathbf{p}}\) as a
  ternary composite as follows:
  \begin{gather*}
    \begin{aligned}
      \Tilde{p_{-}} &= (\padded_{\mathbf{q}} \uparrow w^{i})\app{p^i_-}
      &  \Tilde{p_{+}} &= (\padded_{\mathbf{q}} \uparrow w^{i})\app{p^i_+}
    \end{aligned}\\
    \mu^{i+1}_{\mathbf{q},\mathbf{p}}
    := \left(
      q^{i}_{-} \s (
      X \s_{i+1} (
      \Tilde{p_{-}}\s
      Y
      \s
      \Tilde{p_{+}}
      )) \s q^{i}_{+}
    \right) \s_{i+1} \text{(associator)}
  \end{gather*}
  Here the associator is determined by the type:
  \[
    \begin{tikzcd}
      f_{1} \s (f_{2} \s (f_{3} \s f_{4} \s f_{5}) \s f_{6} ) \s f_{7}
      \ar[d]\\
      (f_{1}\s f_{2} \s f_{3}) \s f_{4} \s (f_{5} \s f_{6} \s f_{7})
    \end{tikzcd}
  \]
This construction is illustrated in Figure~\ref{stacking}.
\end{proof}

\begin{figure}[t]
    \[
    \begin{array}{ccc}
      \begin{tikzcd}[ampersand replacement=\&,row sep=huge]
        {a^0_-} \\
        {b^0_-} \\
        {c^0_-}
        \arrow["{p^0_-}", from=2-1, to=1-1]
        \arrow["{q^0_-}", from=3-1, to=2-1]
\end{tikzcd}   & \begin{tikzcd}[ampersand replacement=\&,row sep=huge]
        {a^0_-} \& {a^0+} \\
        {b^0_-} \& {b^0_+} \\
        {c^0_-} \& {c^0_+}
        \arrow[""{name=0, anchor=center, inner sep=0, pos=.56}, "{a^1_-}", from=1-1, to=1-2]
        \arrow["{p^1_+}", from=1-2, to=2-2]
        \arrow["{p^0_-}", from=2-1, to=1-1]
        \arrow[""{name=1, anchor=center, inner sep=0}, "{b^1_-}"{description}, from=2-1, to=2-2]
        \arrow["{q^0_+}", from=2-2, to=3-2]
        \arrow["{q^0_-}", from=3-1, to=2-1]
        \arrow[""{name=2, anchor=center, inner sep=0}, "{c^1_-}"', from=3-1, to=3-2]
        \arrow["{p^1_-}"', shorten <=6pt, shorten >=4pt, Rightarrow, from=1, to=0]
        \arrow["{q^1_-}"', shorten <=4pt, shorten >=4pt, Rightarrow, from=2, to=1]
\end{tikzcd} & \begin{tikzcd}[ampersand replacement=\&,row sep=huge]
        {a^0_-} \&\& {a^0+} \\
        {b^0_-} \&\& {b^0_+} \\
        {c^0_-} \&\& {c^0_+}
        \arrow[""{name=0, anchor=center, inner sep=0}, "{a^1_+}", curve = {height = -18pt}, from=1-1, to=1-3]
        \arrow["{p^1_+}", from=1-3, to=2-3]
        \arrow["{p^0_-}", from=2-1, to=1-1]
        \arrow[""{name=3, anchor=center, inner sep=0}, "{b^1_+}"{description}, curve = {height = -18pt}, from=2-1, to=2-3]
        \arrow["{q^0_+}", from=2-3, to=3-3]
        \arrow["{q^0_-}", from=3-1, to=2-1]
        \arrow[""{name=4, anchor=center, inner sep=0}, "{c^1_-}"', curve = {height = 18pt}, from=3-1, to=3-3]
        \arrow[""{name=5, anchor=center, inner sep=0}, "{c^1_+}"{description}, curve = {height = -18pt}, from=3-1, to=3-3]
        \arrow["{q^1_+}"{description, pos=0.73}, shift left=5, curve={height=-6pt}, shorten <=5pt, shorten >=2pt, Rightarrow, from=3, to=5]
        \arrow["{c^2_-}"', shorten <=3pt, shorten >=3pt, Rightarrow, from=4, to=5]
        \arrow["{p^1_+}"{description, pos=0.73}, shift left=5, curve={height=-6pt}, shorten <=5pt, shorten >=2pt, Rightarrow, from=0, to=3]
        \arrow[""{name=1, anchor=center, inner sep=0}, "{a^1_-}"{description}, curve = {height =18pt}, from=1-1, to=1-3, crossing over]
        \arrow["{a^2_-}"', shorten <=5pt, shorten >=3pt, Rightarrow, from=1, to=0]
        \arrow[""{name=2, anchor=center, inner sep=0}, "{b^1_-}"{description}, curve = {height = 18pt}, from=2-1, to=2-3, crossing over]
        \arrow["{p^1_-}"{description, pos=0.73}, shift left=5, curve={height=-6pt}, shorten <=2pt, shorten >=2pt, Rightarrow, from=2, to=1]
        \arrow["{q^1_-}"{description, pos=.73}, shift left=5, curve={height=-6pt}, shorten <=2pt, shorten >=2pt, Rightarrow, from=4, to=2]
        \arrow["{b^2_-}"', shorten <=5pt, shorten >=5pt, Rightarrow, from=2, to=3]
\end{tikzcd} \\
        q^0_-\boxminus p^0_- & q^1_- \boxminus p^1_-& \text{$q^2_- \boxminus p^2_-$ ($q^2_-$ and $p^2_-$ fill the centre of the ``cylinders")}
    \end{array}
    \]
    \caption{\centering The constructions of $q^i_-\boxminus p^i_-$, for $i = 0,1,2$, where
      the filtration has height $0$.}
    \label{stacking}
    \end{figure}

\section{Construction of Eckmann-Hilton Cells}\label{fullconstruction}
\noindent
We now construct the cells $\eh^n_{k,l}$ for \(0\leq k,l < n\) with \(k\neq l\),
using our theory of padding. We first present the construction of
$\eh^n_{n-1,0}$ and $\eh^n_{0,n-1}$, then describe how suspension allows the construction of $\eh^{n+1}_{k+1,l+1}$ from \(\eh^{n}_{k,l}\), while naturality allows the construction of
$\eh^{n+1}_{k,l}$ from \(\eh^{n}_{k,l}\). This covers all cases. To simplify the notation, we introduce
the contexts for these cells and types of these cells:
\begin{align*}
  \ehctx^n & := (x : \obj, a,b : \id^{n-1}_x \to \id^{n-1}_x) \\
  \ehty^n_{k,l} &:= a \s_k b \to \padded^n_{k,l}\app{a \s_l b}
\end{align*}
One can check that this type is valid in context
\(\ehctx^n\). Our main result, Theorem~\ref{thm:half-eh}, gives the construction of cells $\eh^n_{k,l}$ such that the following judgement is derivable:
\begin{equation}\label{eq:eh-goal}
    \ehctx^n \vdash \eh^n_{k,l} : \ehty^n_{k,l}
\end{equation}
We begin with the base cases $\eh^n_{n-1,0}$ and $\eh^n_{0,n-1}$.

\begin{lemma}\label{lemma:eh-half-base-valid}
  For every \(n\geq 2\), we can construct
  cells $\eh^n_{n-1,0}$ and $\eh^n_{0,n-1}$ satisfying:
  \begin{align*}
      \ehctx^n &\vdash \eh^n_{n-1,0}: \ehty^n_{n-1,0}
      & \ehctx^n &\vdash \eh^n_{0,n-1}: \ehty^n_{0,n-1}
  \end{align*}
\end{lemma}
\begin{proof}
  The construction of \(\eh^{n}_{n-1,0}\) follows exactly the structure in the
  steps $X_1,X_2,X_3,X_4$ shown in Figure~\ref{2dEH}. We recall the generalised biased unitors (Def.~\ref{def:biased-padding}),
  unbiasing repaddings (Def.~\ref{def:biased-repadding}), and
  pseudofunctoriality of the unbiased padding (Proposition~\ref{construction:Xi}):
  \begin{align*}
    \disc^n &\vdash \rho^n : \padded^n_{\rho}\app{d^n \s_0 \id^n_{d^0_+}} \to
              d^n  \\
    \disc^n &\vdash \lambda^n : \padded^n_{\lambda}\app{ \id^n_{d^0_-} \s_{0} d^n} \to d^n  \\
    \Gamma^n_0 &\vdash \repad^n_{\rho \to u} : \padded^n_{\rho}\app{v^n_0} \to \padded^n_{n-1,0} \\
    \Gamma^n_0 &\vdash \repad^n_{\lambda \to u} : \padded^n_{\lambda}\app{v^n_0} \to \padded^n_{n-1,0} \\
    (\Gamma^n_0,w : I^{n-1}_0) &\vdash \Xi^n_{n-1,0} : \padded^n_{n-1,0}\app{v^n_0} *_{n-1} \padded^n_{n-1,0}\app{w}
    \\ & \hspace{2.7cm}\to \padded^n_{n-1,0}\app{v^n_0 *_{n-1} w}
  \end{align*}
  
  The remaining ingredient is the final interchange step (corresponding to $X_4$ in Figure~\ref{2dEH}). We define a family of pasting context $\mathbb{Z}^{N}$ for \(N\geq 2\), as the $0$\=/gluings of two $N$\=/discs. The
  contexts \(\mathbb{Z}^{2}\) and \(\mathbb{Z}^{3}\) are illustrated in
  Figure~\ref{fig:ps-interchanger-zeta}, and the general formula for
  \(\mathbb{Z}^{N}\) is:
  \[\left(\begin{array}{ll}
    d^0_L, d^0_0 : \obj, \
    d^1_{L-}, d^1_{L+} : d^0_L \to d^0_0, \\
    d^i_{L-}, d^i_{L+} : d^{i-1}_{L-} \to d^{i-1}_{L+} &\text{for $1<i<N$} \\
     d^N_L : d^{N-1}_{L-} \to d^{N-1}_{L+} \\
      d^0_R : \obj,\
     d^1_{R-}, d^1_{R+} : d^0_0 \to d^0_R, \\
    d^i_{R-}, d^i_{R+} : d^{i-1}_{R-} \to d^{i-1}_{R+} & \text{for $1<i<N$} \\
    d^N_R : d^{n-1}_{R-} \to d^{N-1}_{R+})
\end{array}\right)\]
We then define:
\begin{align*}
  \zeta^{N} := \coh(\mathbb{Z}^{N}\! :\! (d^N_L \!\s_0\! \id_{d^{N-1}_{R-}})\! \s_{N-1}\! (\id_{d^{N-1}_{L+}} \!\s_0
  d^N_R)\! \to\! d^N_L \!\s_0\! d^N_R)[\id_{\mathbb{Z}^N}]
\end{align*}
\begin{figure}
  \centering
  \[\begin{tikzcd}[ column sep = 2cm]
        {d^0_L} & {d^0_0} & {d^0_R}
        \arrow[""{name=0, anchor=center, inner sep=0}, "{d^1_{L-}}", bend left = 45, from=1-1, to=1-2]
        \arrow[""{name=1, anchor=center, inner sep=0}, "{d^1_{L+}}"', bend right = 45, from=1-1, to=1-2]
        \arrow[""{name=2, anchor=center, inner sep=0}, "{d^1_{R-}}", bend left = 45, from=1-2, to=1-3]
        \arrow[""{name=3, anchor=center, inner sep=0}, "{d^1_{R+}}"', bend right = 45, from=1-2, to=1-3]
        \arrow["{d^2_L}", shorten <=3pt, shorten >=3pt, Rightarrow, from=0, to=1]
        \arrow["{d^2_R}", shorten <=3pt, shorten >=3pt, Rightarrow, from=2, to=3]
\end{tikzcd}\]
\[\begin{tikzcd}[column sep = 2cm]
        {d^0_L} & {d^0_0} & {d^0_R}
        \arrow[""{name=0, anchor=center, inner sep=0}, "{d^1_{L-}}", bend left = 80, from=1-1, to=1-2]
        \arrow[""{name=1, anchor=center, inner sep=0}, "{d^1_{L+}}"', bend right= 80, from=1-1, to=1-2]
        \arrow[""{name=2, anchor=center, inner sep=0}, "{d^1_{R-}}", bend left = 80, from=1-2, to=1-3]
        \arrow[""{name=3, anchor=center, inner sep=0}, "{d^1_{L+}}"', bend right = 80, from=1-2, to=1-3]
        \arrow[""{name=4, anchor=center, inner sep=0}, "{d^2_{L-}}"'{pos=0.45}, bend right = 60, shorten <=4pt, shorten >=4pt, Rightarrow, from=0, to=1]
        \arrow[""{name=5, anchor=center, inner sep=0}, "{d^2_{L+}}"{pos=0.45}, bend left = 60, shorten <=4pt, shorten >=4pt, Rightarrow, from=0, to=1]
        \arrow[""{name=6, anchor=center, inner sep=0}, "{d^2_{R-}}"'{pos=0.45}, bend right = 60, shorten <=4pt, shorten >=4pt, Rightarrow, from=2, to=3]
        \arrow[""{name=7, anchor=center, inner sep=0}, "{d^2_{R+}}"{pos=0.45}, bend left = 60, shorten <=4pt, shorten >=4pt, Rightarrow, from=2, to=3]
        \arrow["{d^3_L}", shorten <=2pt, shorten >=2pt, Rightarrow, scaling nfold=3, from=4, to=5]
        \arrow["{d^3_R}", shorten <=2pt, shorten >=2pt, Rightarrow, scaling nfold=3, from=6, to=7]
      \end{tikzcd}
    \]
  \caption{The pasting contexts $\mathbb{Z}^2$ and $\mathbb{Z}^3$}
  \label{fig:ps-interchanger-zeta}
\end{figure}%
We  construct the cell $\eh^n_{n-1,0}$ as the following \(4\)-ary
  composite, using the above ingredients, as described in Figure~\ref{2dEH}:
  \[
    \begin{tikzcd}[row sep=17pt]
      a \s_{n-1} b
      \ar[d, "\scriptstyle{(\rho^n)^{-1}\app{a} \s_{n-1} (\lambda^n)^{-1}\app{b}}"] \\
      \padded^n_{\rho}\app{a \s_0 \id^n_x} \s_{n-1} \padded^n_{\lambda}\app{\id^n_x
        \s_0 b}
      \ar[d, "\scriptstyle{\repad^n_{\rho\to u}\app{a \s_0 \id^n_x} \s_{n-1}
        \repad^n_{\lambda\to u}\app{\id^n_x \s_0 b}}"]\\
      \padded^n_{n-1,0}\app{a \s_0 \id^n_x} \s_{n-1}
      \padded^n_{n-1,0}\app{\id^n_x \s_0 b}
      \ar[d, "\scriptstyle{\Xi^n_{n-1,0}\app{a \s_0 \id^n_x, \id^n_x \s_0
          b}}"]\\
      \padded^n_{n-1,0}\app{(a \s_0 \id^n_x) \s_{n-1} (\id^n_x \s_0 b)}
      \ar[d, "\scriptstyle{(\padded^n_{n-1,0}\uparrow v^n_0)\app{\zeta^n}}"]
      \\
      \padded^n_{n-1,0}\app{a \s_0 b}
    \end{tikzcd}
  \]

  Similarly, for \(\eh^{n}_{0,n-1}\) recall the cells:
  \begin{align*}
    \disc^n &\vdash \tilde\rho^n :  \padded^n_{\tilde\rho}\app{d^n} \to d^n \s_0
              \id^n_{d^0_+} \\
     \disc^n &\vdash \tilde\lambda^n :  \padded^n_{\tilde\lambda}\app{d^n} \to
               \id^n_{d^0_+}\s_{0} d^n \\
     \Gamma^n_0 &\vdash \repad^n_{\tilde\rho \to u} : \padded^n_{\tilde\rho}\app{v^n_{n-1}} \to \padded^n_{0,n-1} \\
     \Gamma^n_0 &\vdash \repad^n_{\tilde\lambda \to u} : \padded^n_{\tilde\lambda}\app{v^n_{n-1}} \to \padded^n_{0,n-1} \\
     (\Gamma^n_{n-1},w : I^{n-1}_{n-1}) &\vdash \Xi^n_{0,n-1} : \\
     \padded^n_{0,n-1}\app{v^n_{n-1}} &*_{n-1} \padded^n_{0,n-1}\app{w} \to \padded^n_{n-1,0}\app{v^n_{n-1} *_{n-1} w}
\end{align*}
We then define $\eh^n_{0,n-1}$ as the following composite:
\begin{align*}
  \begin{tikzcd}[row sep=17pt]
    a \s_0 b
    \ar[d,"\scriptstyle{(\zeta^n)^{-1}}"]\\
    (a \s_0 \id^n_x) \s_{n-1} (\id^n_x \s_0 b)
    \ar[d,"\scriptstyle{(\tilde\rho^n)^{-1}\app{a} \s_{n-1} (\tilde\lambda^n)^{-1}\app{b}}"]\\
    \padded^n_{\tilde\lambda}\app{a} \s_{n-1} (\padded^n_{\tilde\lambda})^{\op}\app{b}
    \ar[d,"\scriptstyle{\repad^n_{\tilde\rho \to u}\app{a} \s_{n-1} \repad^n_{\tilde\lambda\to u}\app{b}}"]\\
    \padded^n_{0,n-1}\app{a} \s_{n-1} \padded^n_{0,n-1}\app{b}
    \ar[d,"\scriptstyle{\Xi^n_{0,n-1}\app{a,b}}"]\\
    \padded^n_{0,n-1}\app{a \s_{n-1} b}
  \end{tikzcd}
\end{align*}
This completes the proof.
\end{proof}

\begin{lemma}\label{lemma:eh-susp}
  Assuming that a cell \(\eh^n_{k,l}\) satisfying the judgement~\eqref{eq:eh-goal} is defined, we can define a cell \(\eh^{n+1}_{k+1,l+1}\) such that:
  \[
    \ehctx^{n+1}\vdash \eh^{n+1}_{k+1,l+1} : \ehty^{n+1}_{k+1,m+1}
  \]
\end{lemma}
\begin{proof}
  The cell
  \( \Sigma \eh^n_{k,l}\app{a,b} \) is of type:
  \[
    a \s_{k+1} b \mapsto (\Sigma \padded^n_{k,l})\app{ a \s_{l+1}
      b}
  \]
  To obtain a cell of the desired type, we use the morphism of filtrations:
  \begin{align*}
    \boldsymbol{\psi}_{\susp} :
    \mathbf{\Gamma}^{n+1}_{k+1,l+1} & \to \Sigma\mathbf{\Gamma}^{n}_{k,l} \\
    \Sigma v^{i}_{l}\app{\psi_{\susp}^{i+1}} &= v^{i+1}_{l+1}
  \end{align*}
  We then define repadding data
  \(\mathbf{r}^{n}_{k,l}=(f^i_{k,l},g^i_{k,l})_{i=m+1}^n\) from
  \(\Sigma\mathbf{u}^{n}_{k,l}\app{\boldsymbol{\psi}_{\susp}}\) to
  \(\mathbf{u}^{n+1}_{k+1,l+1}\), whose associated repadding is denoted
  \(\repad^{i}_{\Sigma(k,l)\to(k+1,l+1)}\). We define this repadding as follows, denoting \(j=i+1\):
  \begin{align*}
    f_{k,l}^i &:= \coh(\point :
                \Sigma p_{k,l}^{i-1}\app{\psi_{\susp}^i} \s_{j} \repad^i_{\Sigma(k,l)\to(k+1,l+1)}\app{\inj^-\cir\,
                \sigma^{j}} \to p_{k+1,l+1}^i) [x] \\
    g_{k,l}^i &:= \coh(\point :\Sigma q^{i-1}_{k,l}\app{\psi_{\susp}^i} \to
                \repad^i_{\Sigma(k,l)\to(k+1,l+1)}\app{\inj^+ \cir\, \sigma^{j}} \s_{j} q^i_{k+1,l+1})[x]
  \end{align*}
  The associated repadding then has type:
  \[
    \Gamma^{n+1}_{l+1} \vdash \repad^n_{\Sigma(k,l)\to(k+1,l+1)} : (\Sigma\padded^n_{k,l})\app{v^{n+1}_{l+1}} \to \padded^{n+1}_{k+1,l+1}
  \]
  We thus define the cell as follows:
  \[
    \eh^{n+1}_{k+1,l+1} := \Sigma\eh^{n}_{k,l} \s_{n} \repad^{n}_{\Sigma(k,l)\to(k+1,l+1)} \qedhere
  \]
\end{proof}

\begin{lemma}\label{lemma:eh-naturality}
  Assuming that a cell \(\eh^n_{k,l}\) satisfying the judgement~\eqref{eq:eh-goal} is
  defined, we can define a cell \(\eh^{n+1}_{k,l}\) such that:
  \[
    \ehctx^{n+1}\vdash \eh^{n+1}_{k,l}: \ehty^{n+1}_{k,l}
  \]
\end{lemma}
\begin{proof}
  We construct the cell as the following composite:
  \[
    \begin{tikzcd}[row sep=17pt]
      a \s_k b
      \ar[d,"\text{\small (unitor)}"] \\
      (a \s_k b) \s_n \id (\id^n_x \s_k \id^n_x)
      \ar[d,"\scriptstyle{(a\s_k b)\s_n \xi}"]\\
      (a \s_k b) \s_n (\eh^{n}_{k,l}\app{\id^n_x,\id^n_x} \s_n q^n_{k,l})
      \ar[d, "\text{\small (associator)}"]\\
      ((a \s_k b) \s_n \eh^{n}_{k,l}\app{\id^n_x, \id^n_x}) \s_n q^n_{k,l}
      \ar[d,"\text{\small (naturality)}"] \\
      (\eh^{n}_{k,l}\app{\id^n_x, \id^n_x} \s_n (\Theta^n_{k,l}\uparrow
      v_l^n)\app{a\s_l b}) \s_n q^n_{k,l}
      \ar[d,"\text{\small (associator)}"]\\
      \eh^{n}_{k,l}\app{\id^n_x, \id^n_x} \s_n (\Theta^n_{k,l}\uparrow
      v_l^n)\app{a\s_l b} \s_n q^n_{k,l}
      \ar[d,"\scriptstyle{\xi' \s_n q^n_{k,l}}"]\\
      p^n_{k,l} \s_n (\Theta^n_{k,l}\uparrow v_l^n)\app{a\s_l b} \s_n q^n_{k,l}
    \end{tikzcd}
  \]
  The step labelled ``naturality'' is an application of the inverse of naturality of the cell \(\eh^{n}_{k,l}\), and $\xi$ and
  $\xi'$ are the unique coherences of the required type in  the context $\point$.\qedhere

\end{proof}

\begin{theorem}\label{thm:half-eh}
  For every \(0\leq k,l < n\) with \(k\neq l\), we can construct a cell
  \(\eh^{n}_{k,l}\) such that:
  \[
    \ehctx^n\vdash \eh^{n}_{k,l} : \ehty^n_{k,l}
  \]
  This witnesses that $a *_k b$ is congruent to $a *_l b$.
\end{theorem}
\begin{proof}
  This is obtained by Lemmas~\ref{lemma:eh-half-base-valid},~\ref{lemma:eh-susp} and~\ref{lemma:eh-naturality}
\end{proof}


\begin{corollary}\label{corrolary:commutativity}
  Given \(0\leq k,l <n\) with $k \neq l$, we construct cells $\EH^n_{k,l}$ such that the following judgements are derivable:
  \begin{align*}
    \ehctx^n \vdash \EH^n_{k,l} : a \s_k b \to b \s_k a
  \end{align*}
\end{corollary}
\begin{proof}
We make the following definition:
\[
    \EH^n_{k,l} := \eh^n_{k,l}\app{a,b} \s_n ((\eh^n_{k,l})^{\op\{l+1\}}\app{b,a})^{-1}
\]
  The judgement follows from Theorem~\ref{thm:half-eh} and
  Proposition~\ref{lemma:symmetry-main}.
\end{proof}

We can also extend our construction of $\eh^n_{k,l}$ to include the case where a
padding $\padded^n_{p,-}$ appears in both the source and the target, as in the
vertical morphisms appearing in the Eckmann-Hilton sphere in Figure~\ref{ehsphere}.

\begin{corollary}\label{corollary:ehpkl}
  For \(n \in \N\) and \(p,k,l \leq n\) with \(k\neq l\), there exist terms:
  \begin{align*}
    \ehctx^n &\vdash \eh^{n}_{p,k,l} : \padded^{n}_{p,k}\app{a\s_{k}b}
               \to \padded^{n}_{p,l}\app{a\s_{l}b}
  \end{align*}
\end{corollary}
\begin{proof}
  If \(p=k\), we define
  \(H^{n}_{k,k,l} := H^{n}_{k,l}\), and if \(p=l\), we define
  \(H^{n}_{l,k,l}:= (H^{n}_{l,k})^{-1}\). Suppose that \(p,k,l\) are pairwise
  disjoint, then by Prop.~\ref{prop:padding-compose}, we get terms
  \(\mu^{n}_{p,k,l} :=
  \mu^{n}_{\mathbf{u}^{n}_{p,k},\mathbf{u}^{n}_{k,l}}\app{a\s_{l}b}\).
  In context \(\Gamma^{n}_{l}\), the term \(\mu^{n}_{p,k,l}\) has type:
  \[
    \padded^{n}_{p,k}\app{\padded^{n}_{k,l}\app{a\s_{l}b}} \to
    \padded^{n}_{\mathbf{u}^{n}_{p,k}\square\mathbf{u}^{n}_{k,l}}\app{a\s_{l} b}
  \]
  We then define repadding data $(f^i_{p,k,l},g^i_{p,k,l})$ in the point context \(\mathbb{P} = (x:\obj)\), with associated repadding \(\repad^n_{(p,k)\square(k,l)\to(p,l)}\):
  \begin{align*}
    f^i_{p,k,l} &:= \coh(\point : p^i_{p,k}\boxminus p^i_{k,l} \s_i \repad^i_{(p,k)\square(k,l)\to(p,l)}\app{\inj^-\cir\sigma^{i+1}} \to p^i_{p,l})[x]\\
    g^i_{p,k,l} &:= \coh(\point : q^i_{k,l}\boxplus q^i_{p,k} \to \repad^i_{(p,k)\square(k,l)\to(p,l)}\app{\inj^+\cir\sigma^{i+1}}\s_i q^i_{p,m})[x]
  \end{align*}
  We then have:
  
  \[
    \Gamma^{n}_{l} \vdash \repad^{n}_{(p,k)\square(k,l)\to(p,l)} : \padded^{n}_{\mathbf{u}^{n}_{p,k}\square\mathbf{u}^{n}_{k,l}}\app{a\s_{l}b}
    \to \padded^{n}_{p,l}\app{a\s_{l}b}
  \]
  This lets us define the term \(H^{n}_{p,k,l}\) as follows:
  \[
    (\padded^{n}_{p,k}\uparrow v^{n}_{k})\app{\eh^{n}_{k,l}}
    \s_{n}
    \mu^{n}_{p,k,l}\app{a\s_{l}b}\s_{n}\repad^{n}_{(p,k)\square(k,l)\to(p,l)}\app{a\s_{l}b}\qedhere
  \]
\end{proof}

This completes our construction of all cells appearing in the Eckmann-Hilton spheres in all dimensions. By the work of Benjamin and Markakis, these are all equivalences \cite{benjamin_invertible_2024}.

\section*{Acknowledgements}

\noindent
We would like to thank Alex Corner, Eric Finster and Alex Rice for helpful conversations.

\bibliography{bibliography.bib}
\bibliographystyle{elsarticle-num}

\appendix
\renewcommand{\thesection}{\Alph{section}}

\section{Implementation}\label{sec:implementation}
\noindent
The type theory $\catt$ is implemented as a proof assistant.
The proof assistant reads \texttt{.catt} files, and
typechecks the terms defined therein.

The provided version has our constructions of the cells \(\eh^{n}_{k,l}\) and
\(\EH^{n}_{k,l}\) implemented as new built-in operations, accessible under the names \verb!H! and
\verb!EH!. When invoked with suitable arguments, these will trigger our construction to be executed within the proof assistant. As an example, the following commands typecheck and print the
terms $\eh^3_{2,0}$ and $\EH^3_{2,0}$:
\begin{lstlisting}[language=catt]
check H(3,2,0)
check EH(3,2,0)
\end{lstlisting}

 Using this automation, we can easily
compare the size of the terms generated as $n$, $k$, and $l$ vary. To assess the
complexity of the terms we produce, we use an output
method where all the subterms are recursively defined through ``let-in''
definitions. If a subterm appears multiple times, it is only defined once, and
the corresponding name is reused. This allows us to factor out the
high degree of repetition in the terms we produce, thus giving a reasonable lower-bound of the work a user would have to do to define those terms manually.

We have used this method to generate a range of pre-computed output artifacts, which can be found on github~\cite{benjamin_catt_software_2024} in the directory \verb|results|. As
$n-\max\{k,l\}$ increases, the size of the output artifact grows
rapidly, and we find that terms with $n-\max\{k,l\} > 4$ are
typically too large to be computed and type-checked on our available resources, due to the memory overhead required by the type-checker. Performance analysis indicates  that it is the naturality
step that dominates the complexity of the proof terms in the limit. In Figure~\ref{table:eh-chars}, we list the sizes of a variety of artifacts that we have constructed.

Types in Martin-L\"of Type Theory have the structure of an
\(\omega\)\=/groupoid~\cite{lumsdaine_weak_2009, vandenberg_types_2011,
  altenkirch_syntactical_2012}, and this has been exploited by Benjamin to implement a pipeline that can convert \catt terms to elements of identity types in Homotopy Type Theory, within the prover Rocq~\cite{benjamin_generating_2024}.  We have run this on a selection of our generated terms, and in each case Rocq has successfully validated the resulting structures. It may be interesting to explore opportunities to integrate such Rocq outputs as part of larger proof terms in Homotopy Type Theory.

\begin{figure}[b]
    \centering\small
    \def\jts{\,\,}
     \begin{tabular}{|@{\jts}c@{\jts}|@{\jts}r@{\jts}|@{\jts}r@{\jts}|@{\jts}r@{\jts}|@{\jts}r@{\jts}|@{\jts}r@{\jts}|@{\jts}r@{\jts}|}
     \hline
          $n$ & $\eh^n_{1,0}$ & $\eh^n_{2,1}$ & $\eh^n_{3,2}$ & $\eh^n_{2,0}$ & $\eh^n_{3,1}$ & $\eh^n_{3,0}$\\
          \hline
          $2$ & 5,340 &  & & & & \\
          $3$ & 67,208 & 6,993 & & 44,209 & &\\
          $4$ & 5,339,606 & 116,343 & 8,152 &  3,117,243 & 73,981 & 2,615,998 \\
          $5$ & \scriptsize(overflow) & \scriptsize(overflow) & 178,592 & \scriptsize(overflow) & 6,176,548 & \scriptsize(overflow)
          \\\hline
     \end{tabular}
    \caption{\centering Character counts for selected artifacts $\eh^n_{k,l}$.}
    \label{table:eh-chars}
\end{figure}

\section{Interactions Between Meta-Operations}
\noindent
Here we record some lemmas about the various interactions between suspensions, opposites, functoriality, and substitutions. Some of these results are already known and we give references where appropriate.

\begin{lemma}\label{lemma:susp-results}
  Let $\sigma : \Delta \to \Gamma $ be a substitution. Then:
  \begin{itemize}
  \item If \(\Gamma\vdash t : A\), $\Sigma(t\app{\sigma})=(\Sigma t)\app{\Sigma\sigma}$
  \item If \(\Gamma\vdash A\), $\Sigma(A\app{\sigma})=(\Sigma A)\app{\Sigma\sigma}$
  \end{itemize}
  Moreover, if \(\Gamma\vdashps\), then
  \(\Sigma\partial^{\pm}\Gamma = \partial^{\pm}\Sigma\Gamma\), and
  \(\Sigma\id_{\Gamma} = \id_{\Sigma\Gamma}\).
\end{lemma}

\begin{proof}
This was proved by Benjamin~\cite[Lemma 71]{benjamin_type_2020}.
\end{proof}

\begin{lemma}\label{suspcomp}
  For any family of terms \(t_{0},\ldots,t_{n}\) in a context \(\Gamma\) such
  that \(t_0 \s_k \dots \s_k t_n\) is well defined in \(\Gamma\), the following
  equality holds:
  \[
    \Sigma(t_0 \s_k \dots \s_k t_n) = (\Sigma t_0) \s_{k+1} \dots \s_{k+1}
    (\Sigma t_n)
  \]
  For any term \(\Gamma\vdash t:A\), the following equality holds:
  \[
    \Sigma(\id^n_{t}) = \id^n_{\Sigma t}
  \]
  \end{lemma}
  \begin{proof}
    For the first claim, we prove the more general statement that for any
    pasting context \(\Gamma\), we have
    \(\Sigma\comp_{\Gamma} = \comp_{\Sigma\Gamma}\). We prove this by induction
    on the pasting context \(\Gamma\). If \({\Gamma = \disc^{n}}\) is a disc, then
    we note that up to the \(\alpha\)\=/conversion renaming \(d^{n}\) into
    \(d^{n+1}\), we have \(\Sigma\disc^{n} = \disc^{n+1}\). Then, up to the same
    \(\alpha\)\=/conversion, we have:
    \(\Sigma\comp_{\disc^{n}} = d^{n} = \comp_{\disc^{n+1}}\). When \(\Gamma\)
    is not a disc, we have by induction and Lemma~\ref{lemma:susp-results}.
    \begin{align*}
      \Sigma\comp_{\Gamma}
      & = \coh(\Sigma\Gamma,\Sigma\comp_{\partial^{-}\Gamma}
        \to \comp_{\partial^{+}\Gamma})[\Sigma\id_{\Gamma}] \\
      & = \coh(\Sigma\Gamma,\comp_{\partial^{-}\Gamma} \to
        \comp_{\partial^{+}\Gamma})[\id_{\Sigma\Gamma}] \\
      & = \comp_{\Sigma\Gamma}
    \end{align*}
    Here we use the fact the \(\partial^{\pm}\) reduces the dimension of pasting
    contexts and that the unique pasting context of dimension \(0\) is a disc
    for this induction to be well-founded.

    For the second statement, we proceed by induction on $n$. By definition,
    $\Sigma(\id^0_t) = \Sigma t = \id^0_{\Sigma t}$. For $n>0$, we have:
    \begin{align*}
      \Sigma(\id^n_t)
      &=\coh(\Sigma \disc^n : \Sigma d^n \to \Sigma d^n)[\Sigma \id^{n-1}_t] \\
      &= \coh(\disc^{n+1} : d^{n+1} \to d^{n+1})[\Sigma \id^{n-1}_t]\\
      &= \coh(\disc^{n+1} : d^{n+1} \to d^{n+1})[ \id^{n-1}_{\Sigma t}]\\
      &= \id^n_{\Sigma t} \tag*{\qedhere}
    \end{align*}
  \end{proof}

\begin{lemma}\label{opsub}
  Let $M\subseteq\mathbb{N}_{>0}$ and
  \(\sigma:\Delta\to\Gamma\). For any term $\Gamma\vdash t : A$, we have:
    \[ (t\app{\sigma})^{\op M} = t^{\op M}\app{\sigma^{\op M}}\]
    Similarly, for any substitution
    $\tau : \Gamma \to \Theta$, we have:
    \[
      (\tau \cir \sigma)^{\op M} = \tau^{\op M}\cir \sigma^{\op M}
    \]
  \end{lemma}
  \begin{proof}
    This is functoriality of the opposites construction, proved by Benjamin and Markakis~\cite[§5.2]{benjamin_hom_2024}, together with the fact that well-typed terms $\Gamma \vdash t : A$ of dimension $n$ are in bijection with substitutions $\Gamma \vdash \sigma : \disc^n$.
\end{proof}

\begin{lemma}\label{opcomp}
  Let \(\Gamma\) be a context and \(M\subseteq\mathbb{N}_{>0}\), for any family
  $t_0,\dots t_n$ of terms in \(\Gamma\) such that \(t_0 \s_k \dots \s_k t_n\)
  is well defined, we have:
  \[
    (t_0 \s_k \dots \s_k t_n)^{\op M}\!=\!\begin{cases}
      (t_0)^{\op M}\!\s_k\!\dots\!\s_k\!(t_n)^{\op M} &\!k\!+\!1\!\notin M\\
      (t_n)^{\op M}\!\s_k\!\dots\!\s_k\!(t_0)^{\op M} &\!k\!+\!1\!\in M
    \end{cases}
  \]
  For any term \(t\) in \(\Gamma\), we have
  $$(\id^m_t)^{\op M} = \id^m_{t^{\op M}}$$
\end{lemma}
\begin{proof}
  We first prove the first claim, by proving a more general statement, that is,
  for any pasting context \(\Gamma\), we have
  \((\comp_{\Gamma})^{\op M} = \comp_{\Gamma'}\app{\gamma^{-1}}\), where
  \(\Gamma'\) is the unique pasting context isomorphic to \(\Gamma^{\op M}\) and
  \(\gamma\) is the isomorphism. We prove this by induction on \(\Gamma\). First
  when \(\Gamma = \disc^{n}\) is a disc, we have $(\disc^n)' = \disc^n$, with
  the isomorphism \(\gamma\) swapping \(d^{k}_{-}\) and \(d^{k}_{+}\) for
  every \(k \in M\) such that \(k < n\), and acting as the identity on all other variables.
  Thus, we have
  \[
    (\comp_\Gamma) ^{\op M} = d^n = d^n\app{\gamma^{-1}} =
    \comp_{\Gamma'} \app{\gamma^{-1}}
  \]
  If \(\Gamma\) is not a disc, we distinguish two cases. If
  $\dim(\Gamma) \notin M$, then we have the equality
  \(\partial^{\pm}(\Gamma') = (\partial^{\pm}\Gamma)'\)~\cite[Lemma 16]{benjamin_hom_2024}, and thus:
  \begin{align*}
    &(\comp_\Gamma)^{\op M} \\
    &= \coh(\Gamma' : (\comp_{\partial^- \Gamma})^{\op M}\app{\gamma} \to (\comp_{\partial^+ \Gamma})^{\op M}\app{\gamma})[\gamma^{-1}] \\
    &= \coh(\Gamma' : \comp_{(\partial^- \Gamma)'} \to \comp_{(\partial^+
      \Gamma)'})[\gamma^{-1}]\\
    &=  \coh(\Gamma' : \comp_{\partial^- (\Gamma')} \to \comp_{(\partial^+ \Gamma')})[\gamma^{-1}] \\
    &= \comp_{\Gamma'}\app{\gamma^{-1}}
  \end{align*}
  On the other hand, if $\dim(\Gamma) \in M$n then we have the equality
  \(\partial^{\pm}(\Gamma') = (\partial^{\mp}\Gamma)'\)~\cite[Lemma 16]{benjamin_hom_2024}, and thus:
  \begin{align*}
    &(\comp_\Gamma)^{\op M}\\
    &= \coh(\Gamma' : (\comp_{\partial^+ \Gamma})^{\op M}\app{\gamma} \to (\comp_{\partial^- \Gamma})^{\op M}\app{\gamma})[\gamma^{-1}] \\
    &= \coh(\Gamma' : \comp_{(\partial^+ \Gamma)'} \to \comp_{(\partial^-
      \Gamma)'})[\gamma^{-1}]\\
    &=  \coh(\Gamma' : \comp_{\partial^- (\Gamma')} \to \comp_{(\partial^+
      \Gamma')})[\gamma^{-1}]\\
    &= \comp_{\Gamma'}\app{\gamma^{-1}}
  \end{align*}
  As before, we use the fact the \(\partial^{\pm}\) reduces the dimension of
  pasting contexts and that the unique pasting context of dimension \(0\) is a
  disc for this induction to be well-founded.

  For the second statement, we proceed by induction on $m$. When $m = 0$, we
  have $(\id^m_t)^{\op M} = t^{\op M} = \id^0_{t^{{\op M}}}$ as required. When
  $m>0$, we have:
  \begin{align*}
    (\id^m_t)^{\op M}
    &= \coh((\disc^m)' : d^m\app{\gamma} \to d^m\app{\gamma})[(\id^{m-1}_t)^{\op
      N}] \\
     &= \coh(\disc^m : d^m \to d^m)[(\id^{m-1}_t)^{\op M}] \\
     &= \coh(\disc^m : d^m \to d^m)[\id^{m-1}_{t^{\op M}}]\\
     &= \id^m_{t^{\op M}} \tag*{\qedhere}
  \end{align*}
\end{proof}

\begin{lemma}\label{disjoint-implies-inclusion-is-identity}
  Let \(\Gamma\) be a context and \(X\subseteq\Var(\Delta)\) a set of
  maximal-dimension variables. Then for any term \(\Gamma\vdash t:A\) such that
  \(\supp(t) \cap X = \varnothing\), we have \(t\app{\inj^{\pm}} = t\). Moreover,
  for any substitution \(\Gamma\vdash\sigma:\Delta\) such that
  \(\supp(\sigma)\cap X = \varnothing\), we have
  \(\sigma\cir\inj^{\pm} = \sigma\).
\end{lemma}
\begin{proof}
  We prove these two results by mutual induction. For a term $t = x$ which is a
  variable, by hypothesis, $x \notin X$, so $x\app{\inj^\pm} = x$. For the term
  $\coh(\Theta : B)[\tau]$ we have that \({\supp(t) = \supp(\tau)}\), and we see thus by
  induction that:
  \begin{align*}
    t\app{\inj^\pm}=\coh(\Theta : B)[\tau \cir \inj^\pm]=t
  \end{align*}
  For the empty substitution \(\langle\rangle\), we have
  \(\langle\rangle\cir \inj^{\pm} = \langle\rangle\). For the substitution
  \(\langle\sigma,x\mapsto t\rangle\), we have \(\supp(t)\subseteq \supp(\sigma)\), thus by
  induction,
  \[
    \sub{\sigma,x\mapsto t} \cir \inj^{\pm} =
    \sub{\sigma\cir\inj^{\pm},x\mapsto t\app{\inj^{\pm}}} = \sub{\sigma,t}
    \qedhere
  \]
\end{proof}

\begin{lemma}\label{lemma:support-preserved}
  Let \(\Delta\) be a context. Then for any term \(\Delta\vdash t:A\), we have that $\supp(A)\subseteq \supp(t)$.
  Moreover, for any term \(\Delta\vdash t:A\),
  and any substitution $\Gamma \vdash \sigma : \Delta$, we have:
  \begin{align*}
    \supp(t\app{\sigma}) = \bigcup_{y \in \supp(t)}\supp(y\app{\sigma})\\
    \supp(A\app{\sigma}) = \bigcup_{y \in \supp(A)}\supp(y\app{\sigma})
  \end{align*}
\end{lemma}
\begin{proof} This was proved by Dean et al.~\cite[Lemma 7.3]{dean_computads_2024}.
\end{proof}

\begin{lemma}\label{includefunctorialised}
  Let $\Gamma \vdash \sigma :\Delta$ be a substitution. Let
  $X\subseteq\Var(\Gamma)$ be an up-closed set of variables of depth
  at most \(1\) in \(\Gamma\) and \(\sigma\). Then the inclusions
  $\Delta \uparrow X_\sigma \vdash \inj^\pm_\Delta : \Delta$ and
  $\Gamma \uparrow X \vdash \inj_{\Gamma}^\pm : \Gamma$ satisfy:
  \[
    \inj^\pm_\Delta \cir (\sigma \uparrow X) = \sigma \cir \inj^{\pm}_\Gamma
  \]
\end{lemma}
\begin{proof}
  We will show that they coincide on every variable \(x\). If
  \(x \notin X_{\sigma}\), then by definition
  \[x\app{\inj_\Delta^\pm \cir \sigma\uparrow X}=x\app{\sigma}\]
  Since $\supp(x\app{\sigma})\cap X = \varnothing$, by
  Lemma~\ref{disjoint-implies-inclusion-is-identity}
  \[x\app{\sigma\cir \inj_\Gamma^\pm}=x\app{\sigma}\]
  proving the equality. If \(x\in X_{\sigma}\),
  then by definition,
  \[
    x\app{\inj_{\Delta}^{\pm} \cir(\sigma\uparrow X)} =
    x^{\pm}\app{\sigma\uparrow X} =  x\app{\sigma\cir\inj_\Gamma^\pm} \qedhere
  \]
\end{proof}

\begin{lemma}\label{funcsubs}
  Let \(\Gamma\vdash \sigma :\Delta\) a substitution between contexts of the
  same dimension, and \(X\in\Var(\Gamma)\) a set of variables of depth \(0\)
  with respect to \(\Gamma\) and \(\sigma\). Then for any term
  \(\Delta\vdash t:A\) such that \(\depth_Xt[\sigma] = 0\), we have
  \begin{align*}
    (A \uparrow^{t} X_{\sigma})\app{\sigma\uparrow X}
    &= A\app{\sigma}\uparrow^{t\app{\sigma}} X \\
    (t \uparrow X_{\sigma})\app{\sigma \uparrow X}
    & = t\app{\sigma} \uparrow X
  \end{align*}
  Similarly, for any substitution $\Delta \vdash \tau : \Theta$ such
  that \(\depth_X(\tau\cir \sigma) = 0\), we have:
  \[
    (\tau \uparrow X_\sigma) \cir (\sigma \uparrow X) = (\tau \cir \sigma) \uparrow X
  \]
\end{lemma}
\begin{proof}
  The equality of types is a consequence of Lemma~\ref{includefunctorialised},
  along with the fact that since \(A\) has disjoint support from \(X\), we have
  \(A\app{\sigma\uparrow X} = A\app{\sigma}\):
  \begin{align*}
    (A \uparrow^{t}X_{\sigma})&\app{\sigma\uparrow X} \\
    &= t\app{\inj^{-}\cir(\sigma \uparrow
      X)}\to_{A\app{\sigma}}t\app{\inj^{+}\cir(\sigma\uparrow X)}\\
    &= t\app{\sigma\cir \inj^{-}}\to_{A\app{\sigma}}
      t\app{\sigma\cir\inj^{+}}\\
    &= (A\app{\sigma}\uparrow^{t\app{\sigma}} X)
  \end{align*}

  We prove the equalities on terms and substitution by mutual induction. For a
  term $t=x$ which is a variable, if $ x \in X_\sigma$ then by definition:
  \[(x \uparrow X_\sigma)\app{\sigma\uparrow X} = x\app{\sigma}\uparrow X\]
  If $x \notin X_\sigma$ then
  \[(x \uparrow X_\sigma) \app{\sigma\uparrow X} = x\app{\sigma}\]
  whereas and since \(x\app{\sigma}\cap X = \varnothing\), we also have
  \[x\app{\sigma} \uparrow X = x\app{\sigma}\]
  For the term
  $\coh(\Theta : B)[\tau]$, then if $X_{\tau \cir \sigma} \neq \varnothing$ we
  have by induction, denoting \(u = \coh(\Theta : B)[\id_{\Theta}\):
  \begin{align*}
    (&t \uparrow X_\sigma)\app{\sigma \uparrow X}\\
     &= \coh(\Theta \uparrow (X_{\sigma})_{\tau} : u[\inj^-] \to u[\inj^+])[(\tau \uparrow X_\sigma) \cir (\sigma \uparrow X)] \\
     &= \coh(\Theta \uparrow (X_{\tau\cir\sigma}) : u[\inj^-] \to u[\inj^+])[(\tau \cir \sigma) \uparrow X)]\\
     &= t\app{\sigma} \uparrow X
  \end{align*}
  If $X_{\tau \cir \sigma} = \varnothing$, by
  Lemma~\ref{lemma:support-preserved} we have
  $\supp(t\app{\sigma})\cap X = \varnothing$. Then
  by induction together with
  Lemmas~\ref{disjoint-implies-inclusion-is-identity}
  and~\ref{includefunctorialised}, we have:
  \begin{align*}
    (t \uparrow X_\sigma)\app{\sigma \uparrow X}
    &= t\app{\sigma \uparrow X} \\
    &= t\app{\inj^\pm_\Delta \cir \sigma \uparrow X} \\
    & = t\app{\sigma \cir \inj^\pm_\Gamma} \uparrow X\\
    &= t\app{\sigma}
  \end{align*}
  For the empty substitution $\sub{}$, we have:
  $$(\langle\rangle \uparrow X_\sigma) \cir (\sigma \uparrow X) =
  \sub{} = (\sub{} \cir \sigma)\uparrow X$$
  For substitutions of the form
  $\sub{\tau, x\mapsto t}$, if $x \in X_{\tau \cir \sigma}$, we have, by
  induction and Lemma~\ref{includefunctorialised}:
  \begin{align*}
    &(\sub{\tau,x\mapsto t}\uparrow X_\sigma)
      \cir (\sigma \uparrow X) \\
    &= \left\langle \begin{array}{l}
      (\tau\uparrow X_\sigma)\cir(\sigma \uparrow X), x^{\pm}\mapsto t\app{\inj^{\pm}\cir(\sigma \uparrow X)}, \\
      \fun{x}\mapsto (t\uparrow X_\sigma)\app{\sigma \uparrow X}
    \end{array}\right\rangle \\
    &= \left\langle \begin{array}{l}
      (\tau \cir \sigma)\uparrow X, x^{\pm}\mapsto t\app{\inj^{\pm}\cir(\sigma \uparrow X)}, \\
      \fun{x}\mapsto (t\uparrow X_\sigma)\app{\sigma \uparrow X}
    \end{array}\right\rangle \\
    &= \sub{(\tau \cir \sigma)\uparrow X, x^{\pm}\mapsto t\app{\sigma \cir
      \inj^\pm},\fun{x}\mapsto (t\uparrow X_\sigma)\app{\sigma \uparrow X}}\\
    &= \sub{(\tau \cir \sigma)\uparrow X, x^{\pm}\mapsto t\app{\sigma \cir
      \inj^\pm},\fun{x}\mapsto t\app{\sigma}\uparrow X} \\
    &= (\sub{\tau,x\mapsto t} \cir \sigma)\uparrow X
  \end{align*}
  On the other hand, if \(x\notin X_{\tau\cir\sigma}\), we have by induction,
  \begin{align*}
    \sub{&\tau,x\mapsto t}\uparrow X_{\sigma} \cir(\sigma\uparrow X)\\
         &= \sub{\tau\uparrow X_{\sigma} \cir (\sigma\uparrow X)
           ,x\mapsto t\app{\sigma\uparrow X}} \\
         &= \sub{(\tau\cir\sigma)\uparrow X,x \mapsto t\app{\sigma\uparrow X}}\\
         &= (\sub{\tau,x\mapsto t}\cir\sigma) \uparrow X \tag*{\qedhere}
  \end{align*}
\end{proof}

\begin{lemma}\label{lemma:suspfunc-context}
  For every context \(\Gamma\vdash\) and every \(X\in\U(\Gamma)\) such that \mbox{\(\depth_{X}(\Gamma)=0\)}:
    \begin{align*}
      \Sigma(\Gamma \uparrow X) &= (\Sigma \Gamma) \uparrow X \\
      \Sigma \inj^\pm_{\Gamma, X} &= \inj^{\pm}_{\Sigma \Gamma, X}
    \end{align*}
\end{lemma}
\begin{proof}
  We proceed by structural induction on $\Gamma$. For the empty context
  $\emptycontext$, we have:
  \[
    \Sigma(\emptycontext \uparrow \varnothing) = (\Sigma \emptycontext)\uparrow \varnothing
    = (N : \obj, S : \obj).
  \]
  For the context $(\Gamma,x: A)$, denote \(X' = X\setminus\{x\}\). Then, if
  $x \in X$ we have:
  \begin{align*}
    \Sigma ((\Gamma,x:A)\uparrow X)
          & =\Sigma (\Gamma \uparrow X', x^\pm : A, \fun{x} : x^- \to_{A} x^+) \\
          & = (\Sigma(\Gamma\uparrow X'), x^\pm : \Sigma{A}, \fun{x} : x^-
            \to_{\Sigma A} x^+) \\
          & = ((\Sigma\Gamma)\uparrow X', x^\pm : \Sigma{A}, \fun{x} : x^-
            \to_{\Sigma A} x^+) \\
          & = \Sigma(\Gamma, x : A) \uparrow \Sigma X
  \end{align*}
  On the other hand, if $x \notin X$, we have:
  \begin{align*}
    \Sigma((\Gamma, x : A) \uparrow X)
    &= \Sigma(\Gamma \uparrow X', x : A) \\
    &= (\Sigma(\Gamma \uparrow X'), x : \Sigma A)\\
    &= ((\Sigma \Gamma) \uparrow X', x : \Sigma A)\\
    &= \Sigma(\Gamma, x : A) \uparrow \Sigma X
  \end{align*}
  For the second statement, consider a variable \(x\) of \(\Gamma\).
  If \(x\notin X\), then we have:
  \[
    x\app{\inc^{\pm}_{\Sigma\Gamma,X}} = x = \Sigma(x\app{\inc^{\pm}_{\Gamma,X}}).
  \]
  If \(x\in X\), then:
   \[
    x\app{\inc^{\pm}_{\Sigma\Gamma,X}} = x^{\pm} = x\app{\Sigma\inc^{\pm}_{\Gamma,X}}.
  \]
  Finally, since \(N,S \notin X\), we have:
  \begin{align*}
    N\app{\inc^{\pm}_{\Sigma\Gamma,X}} &= N = N\app{\Sigma\inc^{\pm}_{\Gamma,X}}\\
    S\app{\inc^{\pm}_{\Sigma\Gamma,X}} &= S = S\app{\Sigma\inc^{\pm}_{\Gamma,X}}\\
  \end{align*}
  The two substitutions thus coincide on all variables and therefore are equal.
\end{proof}

\begin{lemma}\label{lemma:suspfunc-term+subs}
  For every context \(\Gamma\) and every \(X\in\U(\Gamma)\) such that \(\depth_{X}(\Gamma)=0\),
  the following hold:
  \begin{itemize}
  \item For any term \(\Gamma\vdash t:A\) such that \(\depth_{X}(t) = 0\), we
    have:
    \[
      \Sigma(t \uparrow X) = (\Sigma t) \uparrow X
    \]
  \item For any substitution $\Gamma \vdash \sigma : \Delta$ such that
    \(\depth_{X}(\sigma) = 0\), we have:
    \[
      \Sigma(\sigma \uparrow X) = (\Sigma \sigma) \uparrow X
    \]
  \item For any term \(\Gamma\vdash t:A\) such that \(\depth_{X}(t) = 1\), we
    have:
    \[
      \Sigma(A \uparrow^{t} X) = (\Sigma A) \uparrow^{\Sigma t} X
    \]
  \end{itemize}
\end{lemma}
\begin{proof}
  We prove the first two statements together by mutual induction. If $t = x$ is
  a variable in \(X\), then:
  \[
    \Sigma(x \uparrow X) = \fun{x} = (\Sigma x) \uparrow \Sigma X
  \]
  If \(t = \coh_{\Delta,B}[\sigma]\) and $\sigma^{-1} X = \emptycontext$,
  then:
 \[
   \Sigma(t \uparrow X) = \Sigma t = (\Sigma t) \uparrow \Sigma X
 \]
 If $\sigma^{-1}X \neq \emptycontext$, since \(N,S\notin X\), we have that
 \(\sigma^{-1} X = (\Sigma\sigma)^{-1} X\). Denoting \(Y = \sigma^{-1} X\) and
 \(u = \coh_{\Delta,B}[\id]\), we have by Lemma~\ref{lemma:suspfunc-context}:
 \begin{align*}
   \Sigma(t\uparrow X)
   &= \coh_{\Sigma (\Delta \uparrow Y), (\Sigma
     u)\app{\Sigma\inj_{\Delta,Y}^-} \to (\Sigma u)\app{\Sigma\inj_{\Delta,Y}^+}}
     [\Sigma(\sigma \uparrow X)] \\
   &= (\Sigma t) \uparrow \Sigma X
 \end{align*}
 For the second statement, for the empty substitution $\sub{}$, we have, since
\mbox{\(N,S\notin X\):}
 \[
   \Sigma(\sub{}\uparrow X) = \sub{N\mapsto N, S\mapsto S} = (\Sigma \sub{})
   \uparrow X
 \]
 For the substitution $\sub{\sigma, x \mapsto t}$, if $x \notin X$, we have:
 \begin{align*}
   \Sigma(\sub{\sigma, x\mapsto t} \uparrow X)
   &=\sub{\Sigma(\sigma\uparrow X), x \mapsto \Sigma t}\\
   &= \sub{(\Sigma \sigma)\uparrow  X, x \mapsto \Sigma t}\\
   &= (\Sigma \sub{\sigma,x\mapsto t}) \uparrow  X
 \end{align*}
 On the other hand, if $x \in X$, then by the inductive hypothesis, by
 Lemma~\ref{lemma:suspfunc-context}, and by the following equation~\cite[Lemma~71]{benjamin_type_2020}
 \[\Sigma(t\app{\inj_{\Gamma,X}^{\pm}}) = (\Sigma
 t)\app{\Sigma\inj_{\Gamma,X}^{\pm}}\]
 we may compute that:
 \begin{align*}
   \Sigma&(\sub{\sigma, x\mapsto t} \uparrow X)\\
   &= \sub{\Sigma(\sigma \uparrow X), x^\pm \mapsto
     \Sigma(t\app{\inj_{\Gamma,X}^{\pm}}), \fun{x}\mapsto \Sigma(t \uparrow X)}\\
   & = \sub{(\Sigma \sigma)\uparrow X, x^\pm \mapsto \Sigma(t\app{\inj_{\Gamma,X}^\pm}),
     \fun x \mapsto (\Sigma t)\uparrow X} \\
   &=\sub{(\Sigma \sigma)\uparrow \Sigma X, x^\pm \mapsto (\Sigma t)
     \app{\inj_{\Sigma \Gamma,X}^\pm}, \fun{x} \mapsto (\Sigma t)\uparrow
     X}\\
   &= (\Sigma \sub{\sigma,x\mapsto t}) \uparrow X
 \end{align*}

 Finally, for the last statement, write \(A = u \to v\) and \(n = \dim A\).
 If
 \(\Var(v)\cap X = \varnothing\) then the source of \(\Sigma (A\uparrow^{t} X)\) is
 given by \(\Sigma (t\app{\inl_{\Gamma,X}})\). On the other hand, the source of
 \((\Sigma A)\uparrow^{\Sigma t} X\) is \((\Sigma t)\app{\inl_{\Sigma\Gamma,X}}\).
 Again by Lemma~\ref{lemma:suspfunc-context} and the same equality as above, we
 may deduce that the two sources agree. If
 \(\Var{(v)}\cap X = \varnothing\), then the source of \(\Sigma (A\uparrow^{t} X)\)
 is \(\Sigma ((t\app{\inl_{\Gamma,X}}) \s_{n} (v\uparrow X))\), while the source of
 \((\Sigma A)\uparrow^{\Sigma t} X\) is
 \((\Sigma t)\app{\inl_{\Sigma\Gamma,X}} \s_{n+1} ((\Sigma v)\uparrow X)\). By the
 first part of the lemma and the same reasoning as in the previous case, we see
 that the two sources agree. A similar argument shows that the target are also
 equal, proving that the two types coincide.
\end{proof}

\begin{lemma}\label{opfuncinclusion}
  Let $\Gamma$ be a \(n\)\=/dimensional context and ${X\subseteq\Var(\Gamma)}$ a set of
  variables of dimension $n$. Then for any
  $M \subseteq \mathbb{N}_{>0}$, there exists an isomorphism:
  \begin{equation}\label{context-op-func}
        \op^{\uparrow}_{\Gamma,X,M} :  (\Gamma \uparrow X)^{\op M}
        \overset{\sim}{\to}
        (\Gamma^{\op M})\uparrow X
  \end{equation}
  Moreover, the source and target inclusions
  $\Gamma \uparrow X \vdash \inj_\Gamma^\pm : \Gamma$ and
  $\Gamma^{\op M}\uparrow X \vdash \inj^\pm_{\Gamma^{\op M}}:
  \Gamma^{\op M}$ satisfy
  \begin{equation}\label{inclusion-op-func}
    (\inj_\Gamma^\pm)^{\op M} =
    \begin{cases}
      \inj^\pm_{\Gamma^{\op M}}\cir\op^{\uparrow}_{\Gamma,X,M} & n + 1 \notin M \\
      \inj^\mp_{\Gamma^{\op M}}\cir\op^{\uparrow}_{\Gamma,X,M} & n + 1 \in M \\
    \end{cases}
  \end{equation}
  If $\Gamma$ is a pasting context, denote \(\Gamma'\) the unique pasting
  context isomorphic to \(\Gamma^{\op M}\) and
  $\Gamma' \vdash \gamma_\Gamma : \Gamma^{\op M}$ the associated isomorphism.
  Similarly, denote by \((\Gamma\uparrow X)'\) the unique pasting context isomorphic to
  \((\Gamma \uparrow X)^{\op M}\) and denote by
  ${(\Gamma\uparrow X)'\ \vdash \gamma_{\Gamma \uparrow X} : \Gamma^{op
    M}\uparrow X^{\op M}}$ the associated isomorphism, then:
  \begin{align}
    (\Gamma \uparrow X)'
    &= \Gamma' \uparrow X\label{derived-context-op-func}\\
    \gamma_{\Gamma} \uparrow X
    &= \op^{\uparrow}_{\Gamma,X,M} \cir
      \gamma_{\Gamma \uparrow X}\label{eq:triangle-op-func}\\
    \gamma^{-1}_\Gamma \cir (\inj_{\Gamma^{\op M}}^\pm) \cir
    \gamma_{\Gamma\uparrow X}
    &= \begin{cases}
      \inj^\pm_{\Gamma'} & n + 1 \notin M \\
      \inj^\mp_{\Gamma'} & n + 1 \in M \\
    \end{cases}\label{derived-inclusion-op-func}
  \end{align}
\end{lemma}
\begin{proof}
  Before proving the lemma, we note that \(\Gamma\uparrow X\) is a pasting
  context~\cite[Lemma~87]{benjamin_type_2020}, so \((\Gamma\uparrow X)'\) exists.
  We first prove the isomorphism \eqref{context-op-func} by structural induction
  on \(\Gamma\). For the empty context \(\varnothing\), we have:
  \[
    (\varnothing \uparrow X)^{\op M} = \varnothing = \varnothing^{\op M}\uparrow
    X
  \]
  For contexts of the form $(\Gamma,x: A)$, if $x \notin X$, we have:
  \begin{align*}
    ((\Gamma, x : A)\uparrow X)^{\op M}
     &=((\Gamma\uparrow X)^{\op M}, x : A^{\op M}) \\
     &= (\Gamma^{\op M} \uparrow X, x :
       A^{\op M})\\
     &= (\Gamma, x : A)^{\op M} \uparrow X
  \end{align*}
  If $x \in X$, and $n + 1 \in M$, then we have:
    \begin{align*}
      (&(\Gamma, x : A)\uparrow X)^{\op M} \\
       &= ((\Gamma\uparrow(X\setminus\{x\}))^{\op M}, x^\pm : A^{\op M}, \fun{x} : x^+ \to x^-)
    \end{align*}
    Thus we define the isomorphism:
    \begin{gather*}
      \op^{\uparrow}_{\Gamma,X,M} :
      ((\Gamma,x:A)\uparrow X)^{\op M} \to (\Gamma\uparrow,x:A)^{\op M} \\
      \sub{\op^{\uparrow}_{\Gamma,X\setminus\{x\},N},x^{\pm} \mapsto
        x^{\mp},\fun x\mapsto \fun x }
    \end{gather*}
    If \(x\in X\) and $n+1\notin M$, then we have:
    \begin{align*}
      (&(\Gamma, x : A)\uparrow X)^{\op M} \\
       &= ((\Gamma\uparrow(X\setminus\{x\}))^{\op M}, x^\pm : A^{\op
         N}, \fun{x} : x^- \to x^+)
    \end{align*}
    Thus we define the isomorphism:
    \begin{gather*}
      \op^{\uparrow}_{\Gamma,X,M} :
      ((\Gamma,x:A)\uparrow X)^{\op M} \to (\Gamma\uparrow,x:A)^{\op M} \\
      \sub{\op^{\uparrow}_{\Gamma,X\setminus\{x\},N},x^{\pm} \mapsto
        x^{\pm},\fun x\mapsto \fun x }
    \end{gather*}
    We then prove \eqref{inclusion-op-func} by showing that the two
    substitutions coincide on every variable. Let \(x \in \Var(\Gamma)\),
    if $x \notin X$, then we have:
    $$ x\app{(\inj^\pm_{\Gamma})^{\op M}}= x
    = x\app{\inj^\pm_{\Gamma^{\op M}}\cir\op^{\uparrow}_{\Gamma,X,M}}$$
    If $x \in X$, and \(n+1\in M\), then we have:
    \[
      x\app{(\inj^\pm_{\Gamma})^{\op M}} = x^\pm =x \app{\inj^\pm_{\Gamma^{\op
            N}}\op^{\uparrow}_{\Gamma,X,M}}
    \]
    If $x \in X$, and \(n+1\notin M\), then we have:
    \[
      x\app{(\inj^\pm_{\Gamma})^{\op M}} = x^\mp =x \app{\inj^\pm_{\Gamma^{\op
            N}}\op^{\uparrow}_{\Gamma,X,M}}
    \]

    Equation~\eqref{derived-context-op-func} follows from
    \eqref{context-op-func} by uniqueness, since \(\Gamma' \uparrow X\) is a
    pasting context isomorphic to
    \(\Gamma^{\op M}\uparrow X = (\Gamma\uparrow X)^{\op M}\).
    Equation~\eqref{eq:triangle-op-func} is then a consequence of this equality,
    obtained by noticing that both \(\gamma_{\Gamma}\uparrow X\) and
    \(\op^{\uparrow}_{\Gamma,X,M}\cir \Gamma\uparrow X\) are isomorphism whose
    source is the pasting context \((\Gamma \uparrow X)'\). Thus there must be
    equal since pasting context have no non-trivial
    automorphisms~\cite{finster_typetheoretical_2017}. Finally,
    Equation~\eqref{derived-inclusion-op-func} follows from
    \eqref{inclusion-op-func} at the level of variables, since by definition the
    substitutions \(\gamma_{\Gamma}\) and \(\gamma_{\Gamma\uparrow X}\) act as
    the identity on every variable. The substitution have the same source and
    target and coincide on every variable, thus, they are equal.
\end{proof}

\begin{lemma} \label{opfunc}
  Let $\Gamma$ be a context,
  \(X\subseteq\Var(\Gamma)\) a set of maximal-dimensional variables of
  \(\Gamma\) and \(M \subseteq \N_{>0}\).
  For any term $\Gamma\vdash t:A$ such
  that \(\depth_X(t) = 0\), we have that:
  \begin{align*}
    (A\uparrow^{t} X)^{\op M}\app{\op^{\uparrow}_{\Gamma,X,M}} &= A^{\op M}
    \uparrow^{t^{\op M}} X\\
    (t\uparrow X)^{\op M}\app{\op^{\uparrow}_{\Gamma,X,M}} &= t^{\op M} \uparrow X
  \end{align*}
  Moreover for any substitution $\Gamma \vdash \sigma : \Delta$ such that
  \(\depth_X(\sigma) \le 0\), we have:
  \begin{align*}
    (\sigma \uparrow X)^{\op M}\cir \op^{\uparrow}_{\Gamma,X,M}
    &= (\op^{\uparrow}_{\Delta,X_{\sigma},N})^{-1} \cir \sigma^{\op M} \uparrow X
  \end{align*}
\end{lemma}
\begin{proof}
  We prove the statements together by mutual induction. First note that by
  definition, the substitution \(\op^{\uparrow}_{\Gamma,X,M}\) acts as the
  identity on every variable that is not in \(X\), and thus for any term whose
  support does not intersect \(X\), we have we have
  \(t\app{\op^{\uparrow}_{\Gamma,X,M}} = t\), and consequently, we have:
  \[
    (t\uparrow X)^{\op M}\app{\op^{\uparrow}_{\Gamma,X,M}} = t^{\op M} = t^{\op M} \uparrow X
  \]
  Thus it suffices to prove the result for terms whose support intersect \(X\).
  For a term \(t=x\) which is a variable, necessarily \(x\in X\) and thus we
  have:
  \[
    (x\uparrow X)^{\op M}\app{\op^{\uparrow}_{\Gamma,X,M}} = \fun x = x\uparrow
    X
  \]
  For the term \(\coh(\Delta : B)[\sigma]\), we denote
  $u = \coh(\Delta : B)[\id_\Delta]$, and
  $v = \coh(\Delta ' : B^{\op M}\app{\gamma_\Delta})[\id_{\Delta'}] = u^{\op
    N}\app{\gamma_\Delta}$
  We note that by induction and Lemmas~\ref{funcsubs} and~\ref{opfuncinclusion},
  we have the equalities:
  \begin{align*}
    \gamma^{-1}_{\Delta \uparrow
    X_\sigma}\cir(\sigma
    \uparrow X)^{\op} \cir(\op^{\uparrow}_{\Gamma,X,M})
    & = (\gamma^{-1}_{\Gamma}\cir (\sigma \uparrow X))^{\op} \\
    (B \uparrow^{u} X_{\sigma})^{\op M}\app{\gamma_{\Delta\uparrow X}}
    &= B^{\op M}\app{\gamma_{\Delta}}\uparrow^{v} X_{\sigma}\\
    (\Delta\uparrow X_{\sigma})' &= \Delta'\uparrow X_{\sigma}
  \end{align*}
  Using this equality together with
  Lemma~\eqref{derived-context-op-func} shows:
  \begin{align*}
    (&t \uparrow X)^{\op M}\app{\op^{\uparrow}_{\Gamma,X,M}}\\
     & = \coh(\Delta'\uparrow X_{\sigma} :
       (B^{\op M}\app{\gamma_{\Gamma}}\uparrow^{v} X))
       [(\gamma^{-1}_{\Delta}\cir (\sigma \uparrow X))^{\op}] \\
     & = t^{\op} \uparrow X
  \end{align*}
  Given a term \(\Gamma\vdash t:A\), of dimension \(n\), by
  Lemma~\ref{opfuncinclusion}, we have, if \(n+1\in M\):
  \begin{align*}
    (&A\uparrow^{t} X)^{\op M}\app{\op^{\uparrow}_{\Gamma,X,M}}\\
    &= t^{\op M}\app{(\inj^{+}_{\Gamma})^{\op M}\cir\op^{\uparrow}_{\Gamma,X,M}}
      \to
      t^{\op M}\app{(\inj^{-}_{\Gamma})^{\op M}\cir\op^{\uparrow}_{\Gamma,X,M}}\\
    &= t^{\op M}\app{\inj^{+}_{\Gamma^{\op M}}}
      \to
      t^{\op M}\app{\inj^{-}_{\Gamma^{\op M}}}\\
    &= A^{\op M} \uparrow^{t^{\op M}} X
  \end{align*}
  On the other hand,  if \(n+1\notin M\):
  \begin{align*}
    (&A\uparrow^{t} X)^{\op M}\app{\op^{\uparrow}_{\Gamma,X,M}}\\
    &= t^{\op M}\app{(\inj^{-}_{\Gamma})^{\op M}\cir\op^{\uparrow}_{\Gamma,X,M}}
     \to
      t^{\op M}\app{(\inj^{+}_{\Gamma})^{\op M}\cir\op^{\uparrow}_{\Gamma,X,M}}\\
    &= t^{\op M}\app{\inj^{-}_{\Gamma^{\op M}}}
      \to
      t^{\op M}\app{\inj^{+}_{\Gamma^{\op M}}}\\
    &= A^{\op M} \uparrow^{t^{\op M}} X
  \end{align*}
  For the empty substitution $\sub{}$, we have:
  \[
    (\sub{} \uparrow \varnothing)^{\op
      N}\cir\op^{\uparrow}_{\varnothing,\varnothing,X} = \sub{} = \sub{}^{\op
      N}\uparrow \varnothing
  \]
  For a substitution of the form
  $\Gamma\vdash \sub{\sigma, x \mapsto t}:(\Gamma,x:A)$, if $ x \notin X$, we
  have by induction:
  \begin{align*}
    (\sub{\sigma, x \mapsto t} \uparrow X)^{\op M}
    &= \sub{(\sigma\uparrow X)^{\op M}, x \mapsto t^{\op M}} \\
    &= \sub{\sigma^{\op M} \uparrow X^{\op M}, x \mapsto t^{\op M}}\\
    &= \sub{\sigma,x\mapsto t}^{\op M} \uparrow X
  \end{align*}
  If $x \in X$ and $\dim(x) + 1 \notin M$, then we have, by induction and
  Lemma~\ref{opfuncinclusion}:
  \begin{align*}
    (&\sub{\sigma, x\mapsto t} \uparrow X)^{\op M}\cir\op^{\uparrow} \\
     &= \left\langle \begin{array}{l}
       (\sigma\uparrow X)^{\op M}\cir\op^{\uparrow},x^{\pm} \mapsto t^{\op M}
       \app{(\inj_\Gamma^\pm)^{\op M} \cir\op^{\uparrow}},
       \\ \fun{x} \mapsto (t \uparrow X)^{\op M}\app{\op^{\uparrow}}
     \end{array}
       \right\rangle
    \\ &= \left\langle
         \begin{array}{l}
           \sigma^{\op M}\uparrow X,x^{\pm}\mapsto t^{\op}\app{\inj_\Gamma^\pm},
           \\
           \fun{x}\mapsto t^{\op M} \uparrow X
         \end{array}\right\rangle \\
     &= \sub{\sigma, x\mapsto t}^{\op M} \uparrow X
  \end{align*}
  Similarly, if $\dim(x)+1\in M$, then:
  \begin{align*}
    (&\sub{\sigma, x\mapsto t} \uparrow X)^{\op M}\cir\op^{\uparrow} \\
     &= \left\langle \begin{array}{l}
       (\sigma\uparrow X)^{\op M}\cir\op^{\uparrow},x^{\pm} \mapsto t^{\op M}
       \app{(\inj_\Gamma^\mp)^{\op M} \cir\op^{\uparrow}},
       \\\fun{x} \mapsto (t \uparrow X)^{\op M}\app{\op^{\uparrow}}
     \end{array}
       \right\rangle
    \\ &= \left\langle
         \begin{array}{l}
           \sigma^{\op M}\uparrow X,x^{\pm}\mapsto t^{\op}\app{\inj_\Gamma^\pm},
           \\
           \fun{x}\mapsto t^{\op M} \uparrow X
         \end{array}\right\rangle \\
     &= \sub{\sigma, x\mapsto t}^{\op M} \uparrow X \tag*{\qedhere}
  \end{align*}
\end{proof}

\section{Interactions with Inverses}
\noindent
We record some lemmas about inverses. First, we present the definition.
\begin{definition}
  We say a coherence term $t = \coh(\Gamma : A)[\sigma]$ is \emph{invertible} if either:
  \begin{itemize}
  \item[(a)]$t$ is a coherence.
  \item[(b)] $t$ is a composite, and $\sigma$ satisfies the invertible image condition.
  \end{itemize}
  Where we say a substitution $\Delta \vdash \sigma : \Gamma$ satisfies the \emph{invertible image condition} if the images of all maximal-dimension variables of $\Gamma$ under $\sigma$ are invertible.
\end{definition}
The work of Benjamin and Markakis~\cite{benjamin_invertible_2024} shows that these are exactly the
equivalences of $\catt$, justifying the definition.
\begin{definition}
    We define, by mutual induction:
    \begin{itemize}
        \item Let $t= \coh(\Gamma : u \to v)[\sigma]$ be an invertible term. Let $n = \dim(t)$. We define:
        \[
        t^{-1} := \begin{cases}
            \coh(\Gamma : v \to u)[\sigma] &\text{if (a)} \\
            \coh(\Gamma' : A^{\op \{n\}}\app{\gamma})[\gamma^{-1}\cir\overline{\sigma}]&\text{if (b)}
        \end{cases}
        \]
        \item Let $\Delta \vdash \sigma : \Gamma$ be a substitution satisfying the invertible image condition. Let $n = \dim(\Gamma)$. We define:
        \begin{align*}
            \Delta \vdash \overline{\sigma} &: \Gamma^{\op\{n\}} \\
            x\app{\overline{\sigma}}&:= \begin{cases}
                (x\app{\sigma})^{-1} &\dim(x)=n \\
                x\app{\sigma}&\dim(x)<n
            \end{cases}
        \end{align*}
    \end{itemize}
\end{definition}
\begin{lemma}\label{invsubs}
    If $\Delta \vdash t : A$ is invertible, and $\Gamma \vdash \sigma : \Delta$ is a substitution, then $t \app \sigma$ is invertible, and the following also holds:
    \[
    (t\app{\sigma})^{-1}=t^{-1}\app{\sigma}
    \]
    If $\Gamma \vdash \sigma : \Delta$ amd $\Delta \vdash \tau : \Theta$ are
    substitutions, and $\tau$ satisfies the invertible image condition, then so
    does $\tau \cir \sigma$, and the following also holds:
    \[
    \overline{\tau \cir \sigma} = \overline{\tau} \cir \sigma
    \]
\end{lemma}
\begin{proof} We prove the two
  statements by mutual induction. For the term
  $\coh(\Theta : u \to v)[\tau]$, if it satisfies (a), then so does
  $\coh(\Theta : u \to v)[\tau \cir \sigma] = t\app{\sigma}$, so it is also
  invertible, and we have:
  \[
    (t\app{\sigma})^{-1}=\coh(\Theta : v \to u)[\tau \cir
    \sigma]=t^{-1}\app{\sigma}
  \]
  If $\coh(\Theta : u \to v)[\tau]$ satisfies (b), then $\tau$ satisfies the
  invertible image condition, and so by the inductive hypothesis for the second
  statement, so does $\tau \cir \sigma$, so
  $\coh(\Theta : u \to v)[\tau \cir \sigma]=t\app{\sigma}$ is invertible.
  Moreover, we have:
    \begin{align*}
      (t\app{\sigma})^{-1}
      &= \coh(\Theta' : (u \to
        v)^{\op\{n\}}\app{\gamma})[\gamma^{-1}\cir\overline{\tau \cir \sigma}]
      \\
      &= \coh(\Theta' : (u \to
        v)^{\op\{n\}}\app{\gamma})[\gamma^{-1}\cir\overline{\tau} \cir
        \sigma]\\
      &= t^{-1}\app{\sigma}
    \end{align*}

    For the empty substitution $\sub{}$ which trivially satisfies the invertible
    image condition, we have:
    \[
      \overline{\sub{}\cir \sigma} =
      \overline{\sub{}}=\sub{}=\sub{}\cir\sigma=\overline{\sub{}}\cir\sigma
    \]
    For the substitution $\sub{\tau,x \mapsto t}$, then either $\dim(x)=n$ or
    $\dim(x)<n$. In the first case, we have by induction:
    \begin{align*}
      \overline{\sub{\tau,x\mapsto t} \cir \sigma}
      &=\sub{\overline{\tau\cir\sigma},x\mapsto (t\app{\sigma})^{-1}}\\
      &=\sub{\overline{\tau}\cir\sigma,x\mapsto t^{-1}\app{\sigma}} \\
      &= \overline{\sub{\tau,x\mapsto t}}\cir \sigma
    \end{align*}
    In the second case, we have by induction:
    \begin{align*}
      \overline{\sub{\tau,x\mapsto t} \cir \sigma}
      &=\sub{\overline{\tau\cir\sigma},x\mapsto
        t\app{\sigma}}\\
      &=\sub{\overline{\tau}\cir\sigma,x\mapsto t\app{\sigma}} \\
      &= \overline{\sub{\tau,x\mapsto t}}\cir \sigma \tag*{\qedhere}
    \end{align*}
  \end{proof}

\end{document}